\documentclass[10pt,reqno]{article}
\usepackage{amsfonts,amsrefs,latexsym,amsmath, amssymb, mathrsfs, verbatim}
\usepackage{url,color}
\usepackage{geometry}
\usepackage{lipsum}
\usepackage{pifont}
\usepackage{subfigure}
\usepackage{upgreek}
\usepackage{fancyhdr}
\usepackage[linktocpage]{hyperref}
\usepackage{calligra}
\usepackage{marvosym}
\usepackage[percent]{overpic}
\usepackage{pict2e}
\usepackage{tensor}
\usepackage{amsthm}
\usepackage[hypcap=false]{caption}
\usepackage{xcolor}
\usepackage{wasysym}
\usepackage{accents}
\usepackage{enumitem}
\usepackage{tikz}
\usetikzlibrary{shapes.geometric, arrows}
\usetikzlibrary{mindmap}
\setlist[enumerate]{label*=\arabic*.}
\numberwithin{equation}{section}
\pdfoutput=1
\topmargin - .23 in

\newtheorem{theorem}{Theorem}[section]
\newtheorem{proposition}{Proposition}[section]
\newtheorem{lemma}[proposition]{Lemma}

\theoremstyle{definition}
\newtheorem{definition}{Definition}[section]
\newtheorem{remark}[proposition]{Remark}

\DeclareMathAlphabet{\mathcalligra}{T1}{calligra}{m}{n}
\DeclareFontShape{T1}{calligra}{m}{n}{<->s*[2.2]callig15}{}

\newcommand{\ent}{s}
\newcommand{\gradEnt}{S}
\newcommand{\vgradientEnt}{\vec{\gradEnt}}
\newcommand{\pressure}{\mathfrak{p}}
\newcommand{\density}{\varrho}

\newcommand{\enth}{H}
\newcommand{\lnenth}{h}

\newcommand{\temp}{\theta}

\newcommand{\tempoverEnth}{\mathfrak{q}}

\newcommand{\vort}{\omega}
\newcommand{\vvort}{\vec{\vort}}
\newcommand{\vortrenormalized}{\omega}
\newcommand{\vVort}{\vec{\vortrenormalized}}

\newcommand{\Int}{\text{int}}
\newcommand{\Domain}{\mathcal{R}}
\newcommand{\speed}{c}

\newcommand{\DivGradEnt}{\mathcal{D}}
\newcommand{\vectvort}{\vec{\modivort}}
\newcommand{\modivort}{\mathcal{C}}

\newcommand{\variables}{\Psi}
\newcommand{\vvariables}{\vec{\variables}}

\newcommand{\norm}[1]{\left\lVert#1\right\rVert}
\newcommand{\onenorm}[4]{\norm{#1}_{L_{#2}^{#3}(#4)}}
\newcommand{\twonorms}[6]{\norm{#1}_{L_{#2}^{#3}L_{#4}^{#5}(#6)}}

\newcommand{\threenorms}[8]{\norm{#1}_{L_{#2}^{#3}L_{#4}^{#5}L_{#6}^{#7}(#8)}}
\newcommand{\holdernorm}[5]{\norm{#1}_{C_{#2}^{#3,#4}(#5)}}
\newcommand{\semiholdernorm}[5]{\norm{#1}_{\dot{C}_{#2}^{#3,#4}(#5)}}
\newcommand{\holdertwonorms}[7]{\norm{#1}_{L_{#2}^{#3}C_{#4}^{#5,#6}(#7)}}
\newcommand{\holderthreenorms}[9]{\norm{#1}_{L_{#2}^{#3}L_{#4}^{#5}C_{#6}^{#7,#8}(#9)}}
\newcommand{\sobolevnorm}[3]{\norm{#1}_{H^{#2}(#3)}}
\newcommand{\abs}[1]{\left\vert#1\right\vert}
\newcommand{\sgabs}[1]{\left\vert#1\right\vert_{\gsphere}}

\newcommand{\curl}{\mbox{\upshape curl}}

\newcommand{\antisymmetic}{\epsilon}

\newcommand{\asphere}{\upomega^{A}}
\newcommand{\bsphere}{\upomega^{B}}
\newcommand{\csphere}{\upomega^{C}}

\newcommand{\sangle}{\upomega}
\newcommand{\sangleone}{\sangle_{(1)}}
\newcommand{\sangletwo}{\sangle_{(2)}}
\newcommand{\derinormal}{\frac{\p}{\p w}}
\newcommand{\deriasphere}{\frac{\p}{\p\asphere}}
\newcommand{\deribsphere}{\frac{\p}{\p\bsphere}}
\newcommand{\dericsphere}{\frac{\p}{\p\csphere}}

\newcommand{\Tstar}{T_{*}}
\newcommand{\Trescale}{T_{*;(\lambda)}}

\newcommand{\region}{\mathcal{M}}
\newcommand{\intregion}{\region^{(\text{Int})}}
\newcommand{\extregion}{\region^{(\text{Ext})}}
\newcommand{\Stwo}{\mathbb{S}^2}
\newcommand{\stu}{S_{t,u}}

\newcommand{\St}{\Sigma_t}

\newcommand{\Szero}{\Sigma_0}
\newcommand{\initialSzero}{\Sigma_0^{w(\lambda)}}
\newcommand{\coneu}{\mathcal{C}_{u}}

\newcommand{\gensmoothfunction}{\mathrm{f}}
\newcommand{\quadsmoothfunction}{\mathscr{Q}}
\newcommand{\linsmoothfunction}{\mathscr{L}}

\newcommand{\lgensmoothfunction}{\mathrm{f}_{(\vLunit)}}

\newcommand{\Ricfour}[2]{\mathbf{Ric}_{#1 #2}}
\newcommand{\Riemfour}[4]{\mathbf{Riem}_{#1 #2 #3 #4}}

\newcommand{\spherenormal}{N}

\newcommand{\nulllapse}{b}
\newcommand{\lapse}{a}
\newcommand{\volume}{\upsilon}
\newcommand{\spheresecondfund}{\uptheta}
\newcommand{\modtorsion}{\tilde{\upzeta}}
\newcommand{\uzeta}{\underline{\upzeta}}
\newcommand{\hchi}{\hat{\upchi}}
\newcommand{\uchi}{\underline{\upchi}}
\newcommand{\huchi}{\hat{\uchi}}
\newcommand{\Restrace}{\mytr_{\congsphere}}
\newcommand{\reschi}{\tilde{\upchi}}
\newcommand{\chismall}{\tilde{\upchi}^{\text{(Small)}}}
\newcommand{\resuchi}{\tilde{\uchi}}

\newcommand{\Jen}[1]{^{(#1)} \mkern-3mu \mathbf{J}}
\newcommand{\Jenarg}[2]{{^{(#1)} \mkern-3mu \mathbf{J}^{#2}}}
\newcommand{\Jenwithlowerarg}[2]{{^{(#1)} \mkern-3mu \mathbf{P}^{#2}}}
\newcommand{\Harg}[2]{{^{(#1)} \mkern-3mu H_{#2}}}

\newcommand{\deform}[1]{{^{(#1)} \mkern-1mu \pmb{\pi}}}

\newcommand{\minkowski}{M}
\newcommand{\gacoustical}{\gfour_{Acou}}
\newcommand{\gfour}{\mathbf{g}}
\newcommand{\rescaledgfour}{\widetilde{\mathbf{g}}}
\newcommand{\gt}{g}

\newcommand{\gsphere}{g \mkern-8.5mu / }
\newcommand{\esphere}{e \mkern-8.5mu / }

\newcommand{\congsphere}{\widetilde{g} \mkern-8.5mu / }
\newcommand{\tip}{\mathbf{z}}

\newcommand{\conformalfactor}{\upsigma}
\newcommand{\Chfour}{\pmb{\Gamma}}

\newcommand{\gvol}{\varpi_\gfour}
\newcommand{\tvol}{\varpi_g}

\newcommand{\spherevol}{\varpi_{\gsphere}}
\newcommand{\flatspherevol}{\varpi_{\esphere}}

\newcommand{\cphi}{\ln\left(\rgeo^{-2}v\right)}

\newcommand{\sphereproject}{{\Pi \mkern-12mu / } \, }
\newcommand{\Sigmatproject}{\underline{\Pi}}

\newcommand{\inhom}[1]{\mathfrak{F}_{(#1)}}
\newcommand{\remainder}[1]{\mathfrak{R}_{(#1);\upnu}}
\newcommand{\remainde}[1]{\mathfrak{R}}

\newcommand{\mytr}{{\mbox{\upshape tr}}}
\newcommand{\gtr}{\mytr_{\gsphere}}
\newcommand{\ggtr}{\mytr_{g}}
\newcommand{\Timelike}{\mathbf{T}}
\newcommand{\materialderivative}{\mathbf{B}}
\newcommand{\Lunit}{L}
\newcommand{\uLunit}{\underline{L}}
\newcommand{\Lgeo}{L_{(Geo)}}

\newcommand{\p}{\partial}
\newcommand{\pfour}{\pmb\partial}
\newcommand{\vLunit}{\vec{\Lunit}}

\newcommand{\Dfour}{\mathbf{D}}
\newcommand{\angD}{ {\Dfour \mkern-14mu / \,}}
\newcommand{\angnabla}{ {\nabla \mkern-14mu / \,} }

\newcommand{\angdiv}{\mbox{\upshape{div} $\mkern-17mu /$\,}}
\newcommand{\angcurl}{\mbox{\upshape{curl} $\mkern-17mu /$\,}}
\newcommand{\anglap}{ {\Delta \mkern-12mu / \, } }

\newcommand{\rgeo}{\tilde{r}}

\newcommand{\modmass}{\check{\upmu}}
\newcommand{\hodgemass}{{\upmu \mkern-10mu / \,}}
\newcommand{\mass}{\upmu}
\newcommand{\umodmass}{\bar{\modmass}}
\newcommand{\massone}{\hodgemass_{(1)}}
\newcommand{\masstwo}{\hodgemass_{(2)}}

\newcommand{\Lie}{\mathcal{L}}

\newcommand{\angLie}{ { \mathcal{L} \mkern-10mu / } }
\newcommand{\littlewood}{P_\upnu}

\newcommand{\boxg}{\square_{\gfour}}%Table of contents

\newcommand{\diff}{\mathrm d}

\newcommand{\neighborhood}{\mathfrak{O}}

\newcommand{\zero}{\mathcal{O}}
\newcommand{\average}[1]{\overline{#1}}

\newcommand{\Energyconformal}{\mathfrak{C}}
\newcommand{\iEnergyconformal}{\Energyconformal^{(i)}}
\newcommand{\eEnergyconformal}{\Energyconformal^{(e)}}

\newcommand{\modifiedcurrent}[1]{\mkern-1mu ^{(#1)} \mkern-3mu \widetilde{\mathbf{J}}}
\title{Rough Solutions of the Relativistic Euler Equations}

\date{\today}       
 
\begin{document}
	\author{Sifan Yu
		\thanks{The author was supported in part by NSF grant \# 2054184.}
		\thanks{Department of Mathematics, Vanderbilt University, Nashville, TN, USA. sifan.yu@vanderbilt.edu.}}
	\maketitle         
	\begin{abstract}
		We prove that the time of classical existence of smooth solutions to the relativistic Euler equations can be bounded from below in terms of norms that measure the ``(sound) wave-part" of the data in Sobolev space and ``transport-part" in higher regularity Sobolev space and H\"older spaces.  The solutions are allowed to have nontrivial vorticity and entropy. We use the geometric framework from \cite{TheRelativistivEulerEquations}, where the relativistic Euler flow is decomposed into a ``wave-part", that is, geometric wave equations for the velocity components, density and enthalpy, and a ``transport-part", that is, transport-div-curl systems for the vorticity and entropy gradient. Our main result is that the Sobolev norm $H^{2+}$ of the variables in the ``wave-part" and the H\"older norm $C^{0,0+}$ of the variables in the ``transport-part" can be controlled in terms of initial data for short times. We note that the Sobolev norm assumption $H^{2+}$ is the optimal result for the variables in the ``wave-part." Compared to low regularity results for quasilinear wave equations and the 3D non-relativistic compressible Euler equations, the main new challenge of the paper is that when controlling the acoustic geometry and bounding the wave equation energies, we must deal with the difficulty that the vorticity and entropy gradient are four-dimensional space-time vectors satisfying a space-time div-curl-transport system, where the space-time div-curl part is not elliptic. Due to lack of ellipticity, one can not immediately rely on the approach taken in \cite{TheRelativistivEulerEquations} to control these terms. To overcome this difficulty, we show that the space-time div-curl systems imply elliptic div-curl-transport systems on constant-time hypersurfaces plus error terms that involve favorable differentiations and contractions with respect to the four-velocity. By using these structures, we are able to adequately control the vorticity and entropy gradient with the help of energy estimates for transport equations, elliptic estimates, Schauder estimates, and Littlewood-Paley theory.
		\\
	\end{abstract}
	\tableofcontents
	\newpage
\section{Introduction}\label{sectionintro}
This paper is concerned with the special relativistic Euler equations on the Minkowski background $(\mathbb{R}^{1+3},\minkowski)$, where $\minkowski$ is the Minkowski metric. For use throughout the article, we fix a coordinate system $\{x^\alpha\}_{\alpha=0,1,2,3}$, relative\footnote{Throughout this article, we use the notation that Greek ``space-time" indices take on the values $0,1,2,3$, while Latin ``spatial indices" take on the values $1,2,3$. We use Einstein summation convention throughout the paper.} to which $\minkowski_{\alpha\beta}:=\text{diag}(-1,1,1,1)$, where the speed of light is set to be $1$. For these equations, there is considerable freedom in the choice of state-space variables, that is, the fundamental unknowns of the PDEs. In this work, we choose the logarithmic enthalpy $\lnenth$, the entropy $\ent$, and the four-velocity $v$, which is a future directed $\minkowski$-timelike vectorfield normalized as $\minkowski(v,v) = -1$.  We allow for non-trivial vorticity. All other unknowns in the system can be considered as functions of the state-space variables. We denote the pressure as $\pressure=\pressure(\lnenth,\ent)$, the fluid density as $\density=\density(\lnenth,\ent)$ and speed of sound as $\speed:=\sqrt{\frac{\p\pressure}{\p\density}}$. In this coordinate system, the relativistic Euler Equations can be expressed as\footnote{For Greek and Latin indices, for any vectorfield or one-form $V$, we lower and raise indices with the Minkowski metric $\minkowski_{\alpha\beta}:=\text{diag}(-1,1,1,1)$ and its inverse by using the notation  $(V_\flat)_\beta:=\minkowski_{\alpha\beta}V^\alpha$ and $(V^\sharp)^\beta:=(\minkowski^{-1})^{\alpha\beta}V_\alpha$.}:
\begin{subequations}\label{firstorderEuler}
	\begin{align}
	v^\kappa\p_\kappa\lnenth+\speed^2\p_\kappa v^\kappa&=0,\\
	v^\kappa\p_\kappa (v_\flat)_\alpha+\p_\alpha \lnenth+(v_\flat)_\alpha v^\kappa\p_\kappa \lnenth-\tempoverEnth\p_\alpha \ent&=0,\\
	v^\kappa\p_\kappa \ent&=0,
	\end{align}
\end{subequations}
where $\tempoverEnth:=\frac{\temp}{\enth}$ is temperature over enthalpy, which can be expressed as $\tempoverEnth=\tempoverEnth(\lnenth,\ent)$. Also see Section \ref{deffluid}-\ref{RelEuler} for the details.\\

Our work intimately depends on a new formulation of the equations derived by Disconzi-Speck \cite{TheRelativistivEulerEquations}, where the authors found that the flow splits into a ``sound-wave-part” (``wave-part" for short) for $(\lnenth,\ent,v)$ and a ``transport-div-curl-part” (``transport-part" for short) for the vorticity $\vort$ and the entropy gradient $\gradEnt$. Schematically, the geometric formulation takes the following form\footnote{We denote schematic spatial partial derivatives and space-time partial derivatives by $\p$ and $\pfour$ respectively. Also, we use the following schematic notations throughout the paper where $A,B,C$ are arrays of variables:
	\begin{itemize}
		\item $\linsmoothfunction[A](B)$ denotes any scalar-valued function that is linear in the components of $B$ with coefficients that are a function of the components of $A$.
		\item $\quadsmoothfunction[A](B,C)$ denotes any scalar-valued function that is quadratic in the components of $B$ and $C$ with coefficients that are a function of the components of $A$.
\end{itemize}}:\\

\textbf{Wave equations}	
\begin{align}
\label{intro1}\boxg\variables=\linsmoothfunction(\vvariables)[\vectvort,\DivGradEnt]+\quadsmoothfunction(\vvariables)[\pfour\vvariables,\pfour\vvariables].
\end{align}

\textbf{Transport equations}
\begin{subequations}\label{intro2}
	\begin{align}
\label{intro2.1}	&\materialderivative\vort^\alpha=\linsmoothfunction(\vvariables,\vvort,\vgradientEnt)[\pfour\vvariables],\\ \label{intro2.2}&\materialderivative(\gradEnt^\sharp)^\alpha=\linsmoothfunction(\vvariables,\vgradientEnt)[\pfour\vvariables].
	\end{align}
\end{subequations}

\textbf{Transport-Div-Curl system}
\begin{subequations}\label{intro3}
	\begin{align}
\label{intro3.1}&\materialderivative \modivort^\alpha=\inhom{\modivort^\alpha}:=\quadsmoothfunction(\vvariables)[\pfour\vvariables,(\pfour\vort,\pfour\vgradientEnt,\pfour\vvariables)]+\quadsmoothfunction(\vgradientEnt)[\pfour\vvariables,\pfour\vvariables]+\linsmoothfunction(\vvariables,\vvort,\vgradientEnt)[\pfour\vvariables,\pfour\vgradientEnt],\\
\label{intro3.2}&\materialderivative\DivGradEnt=\inhom{\DivGradEnt}:=\linsmoothfunction(\vvariables,\vgradientEnt)[\pfour\vort]+\quadsmoothfunction(\vvariables)[\pfour\vvariables,(\pfour\vgradientEnt,\pfour\vvariables)]+\quadsmoothfunction(\vgradientEnt)[\pfour\vvariables,\pfour\vvariables]+\linsmoothfunction(\vvariables,\vvort,\vgradientEnt)[\pfour\vvariables],\\
\label{intro3.3}&\text{vort}^\alpha(S)=0,\\ 
\label{intro3.4}&\p_\alpha\vort^\alpha=\linsmoothfunction(\vort)[\pfour\vvariables],
\end{align}
\end{subequations}
where $\vvariables:=(v^0,v^1,v^2,v^3,h,s)$, $\gfour=\gfour(\vvariables)$ is a solution-dependent Lorentzian metric which governs the geometry of sound waves, $\modivort\simeq\text{vort}(\vort)$ and $\DivGradEnt\simeq\text{div} \gradEnt$ are special modified fluid variables, and  $\materialderivative$ is the material derivative, which is parallel to $v^\alpha$. See Definition \ref{vorticity} for the definition of $\vort$, Definition \ref{D:entgradient} for the definition of $\gradEnt$, and Section \ref{newformulation} for the precise definitions of $\gfour,\modivort,\DivGradEnt,\materialderivative$ and more details of the geometric formulation of the equations. This formulation reveals miraculous regularity and geometric properties of the flow, which is used in a fundamental way in the present work. These geometric properties are not visible in first order equations (\ref{firstorderEuler}).

In this paper, we show that under low regularity assumptions on the ``wave-part" (see Section \ref{Data} for more details of ``wave-part") of the initial data, the regularity of solutions of the relativistic Euler equations can be preserved for a short time. Specifically, we assume that the ``wave-part" of the data belongs to $H^{2+}$, and that the ``transport-part" $\modivort$ and $\DivGradEnt$ are in H\"older space $C^{0,0+}$. Our proof shows, in particular, that it is possible to avoid instantaneous shock formation, which in \cite{counterexamplestolocalexistenceforquasilinearwaveequations} was shown to occur in the irrotational case (i.e., for quasilinear wave equations) for initial data in $H^2$. \textbf{In particular, our regularity assumptions are optimal with respect to the ``sound-wave-part" of the data.} One cannot hope to avoid singularities globally in time: it is known that, even in the irrotational and isentropic case, the compression of sound waves can cause shocks to develop from regular initial data in finite time. Moreover, in more than one space dimension and away from symmetry, these singularities are known to be stable as in \cite{Theformationofshocksin3-dimensionalfluids}. 

For the irrotational and isentropic case in \cite{Theformationofshocksin3-dimensionalfluids}, the relativistic Euler equations reduce to a system of covariant quasilinear wave equations for the first derivatives $\psi_\alpha:=\p_\alpha\phi$ of a potential function $\phi$ of the following form:
\begin{align}\label{quasi}
\square_{\gfour(\psi)}(\psi_\alpha)=\quadsmoothfunction(\psi)[\pfour\psi,\pfour\psi].
\end{align} 
Classical local well-posedness in $H^{(5/2)+}$ for the quasilinear wave system (\ref{quasi}) can be obtained by applying energy estimates and Sobolev embedding, see Kato \cite{kato}. Starting in the late 90s, the regularity needed for local well-posedness for quasilinear wave equations was improved in a series of works by Bahouri-Chemin, Smith-Tataru and Klainerman-Rodnianski, see \cite{Equationsd'ondesquasilinearireseteffetdispersif,Equationsd'ondesquasilineairesetestimationdeStrichartz,Strichartzestimatesforsecondorderhyperbolicoperators,Strichartzestimatesforoperatorswithnonsmooth,ACommutingVectorfieldsApproachtoStrichartz,ImprovedLocalwellPosedness}.
 The optimal result for low regularity $H^{2+}$ of quasilinear wave equations was first achieved by Smith-Tataru in \cite{Sharplocalwell-posednessresultsforthenonlinearwaveequation}. In  \cite{AGeoMetricApproach}, Wang reached the same result as in \cite{Sharplocalwell-posednessresultsforthenonlinearwaveequation} by using a geometric approach. With the presence of vorticity, Disconzi-Luo-Mazzone-Speck, Wang, and Zhang-Andersson proved low-regularity local well-posedness result for the 3D compressible Euler equations in \cite{3DCompressibleEuler}, \cite{Roughsolutionsofthe3-DcompressibleEulerequations} and \cite{Alterproof} respectively. In all three works, the regularity of ``wave-part" is in the optimal level $H^{2+}$. We will discuss the details of the assumptions for the data in \cite{3DCompressibleEuler} and \cite{Roughsolutionsofthe3-DcompressibleEulerequations} in Section \ref{Overviewofpreviousresults}.

Compared to the non-relativistic case, the first fundamental form of $\St:=\{t\}\times\mathbb{R}^3$ is no longer conformally flat in the relativistic case, leading to more complicated geometry. One of the main challenges of this paper is that, (\ref{intro1})-(\ref{intro3}) seemingly suffers from a loss of derivative. This is because at the level of regularity, $\modivort,\DivGradEnt\simeq\pfour^2\vvariables$, which is an issue since $\modivort$ and $\DivGradEnt$ show up as the source terms in the right-hand side of the wave equation (\ref{intro1}). In \cite{3DCompressibleEuler}, this was solved by using Hodge theory on the spacelike hypersurfaces $\St$. In our case, the transport-div-curl system (\ref{intro3}) is a space-time (non-elliptic Hodge) system, from which we have to extract a quasilinear elliptic Hodge system on the spacelike hypersurfaces, that is, we rewrite (see Proposition \ref{NewDivCurl}) the space-time div-curl system (\ref{intro3}) into a spatial elliptic div-curl system with source terms that can be controlled only due to the special structure of the equations:
 \begin{align}
\label{Gabintro}G^{ab}\p_a((\vort_\flat),\gradEnt)_b&=F,\\
\p_a((\vort_\flat),\gradEnt)_b-\p_b((\vort_\flat),\gradEnt)_a&=H_{ab}.
\end{align}
In (\ref{Gabintro}), $G^{-1}:=G^{-1}(v)$ is the inverse of a Riemannian metric on constant-time hypersurfaces (see equation (\ref{D:Gab}) for the definition of $G^{-1}$). By using these structures, we are able to adequately control the vorticity and entropy gradient by using energy estimates for transport equations, elliptic estimates, Schauder estimates, and Littlewood-Paley theory.\\
\indent
We now state the main results of this paper.
\begin{theorem}[Main theorem]\label{IntroMainTheorem}
	Consider a smooth\footnote{By smooth we mean as smooth as necessary for the analysis arguments to go through. We note that all of our quantitative estimates depend only on the Sobolev and H\"older norms.} solution to the relativistic Euler equations whose initial data on the initial Cauchy hypersurface $\Sigma_{0}:=\{0\}\times\mathbb{R}^3$ satisfies following assumptions for some real number $2<N<5/2$, $0<\alpha<1$, $c_1>0$ and $D$:\\
	1. \textbf{``Wave part"}: $\sobolevnorm{h, v}{N}{\Sigma_0}\leq D$,\\
	2. \textbf{``Transport part"}: $\sobolevnorm{\vort}{N}{\Sigma_0}+\sobolevnorm{s}{N+1}{\Sigma_0}\leq D$, In addition, Modified fluid variables $\modivort$ and $\DivGradEnt$ ($\modivort\simeq\textnormal{vort}(\vort)$ and $\DivGradEnt\simeq\textnormal{div} \gradEnt$, see Subsection \ref{v-ortho} for the definition of operator \textnormal{vort} and $\vort$, Definition \ref{D:entgradient} for the definition of $\gradEnt$, and Section \ref{modifiedfluidvariables} for the definition of $\modivort$ and $\DivGradEnt$)  satisfy the H\"older-norm bound $\holdernorm{\modivort,\DivGradEnt}{}{0}{\alpha}{\Sigma_0}\leq D$,\\
	3. The image of data functions are contained in an interior of a compact subset $\Domain$ (defined in Section \ref{Data}) and the enthalpy $\enth$ is positive, i.e. $\enth\geq c_1>0$.\\
	Then the solution's time of classical existence $T:=T(D,\Domain)>0$ can be controlled in terms of only $D$ and $\Domain$. Moreover, the Sobolev and some H\"older regularities of the data are propagated by the solution.
\end{theorem}
See Section \ref{mainidea} for the main ideas behind proving Theorem \ref{IntroMainTheorem}.
\subsection{Overview of Previous Low-Regularity Results}\label{Overviewofpreviousresults} 
There has been many developments on low regularity problems for quasilinear wave equations and the non-relativistic 3D compressible Euler equations in past two decades. For quasilinear wave equations of the form (\ref{quasi}), Bahouri-Chemin \cite{Equationsd'ondesquasilineairesetestimationdeStrichartz} and Tataru \cite{Strichartzestimatesforoperatorswithnonsmooth} independently showed local well-posedness with $H^{(2+\frac{1}{4})+}$ data. The improvements rely on Strichartz estimates based on Fourier integral parametrix representations. Bahouri-Chemin improved their earlier result to $H^{(2+\frac{1}{5})+}$ in \cite{Equationsd'ondesquasilinearireseteffetdispersif}. Tataru pushed the results down to $H^{(2+\frac{1}{6})+}$ in \cite{Strichartzestimatesforsecondorderhyperbolicoperators} and Klainerman reached the same level in \cite{ACommutingVectorfieldsApproachtoStrichartz}. Klainerman-Rodnianski achieved $H^{(2+\frac{2-\sqrt{3}}{2})+}$ in \cite{ImprovedLocalwellPosedness}.
	  The optimal low regularity result $H^{2+}$ for generic quasilinear wave equations was first achieved by Smith-Tataru in \cite{Sharplocalwell-posednessresultsforthenonlinearwaveequation} by using wave-packets and properties of the geometry of characteristic light cones that were introduced in \cite{ImprovedLocalwellPosedness}. Besides the improvements over Sobolev exponents, a commuting vectorfield approach for Strichartz estimates was introduced by Klainerman in \cite{ACommutingVectorfieldsApproachtoStrichartz}, and a fundamental decomposition of a Ricci component of $\gfour$ was used for improving the regularity in the causal geometry by Klainerman-Rodnianski in \cite{ImprovedLocalwellPosedness}. Recently, Wang gave a second proof of Smith-Tataru \cite{Sharplocalwell-posednessresultsforthenonlinearwaveequation} by using this geometric approach. The proof in Wang \cite{AGeoMetricApproach} relied on an upgraded version of Klainerman-Rodnianski's vectorfield method with the help of conformal energy estimates. We again emphasize that, for the general quasilinear wave equation of the form (\ref{quasi}), it is impossible to prove any well-posedness result with data in $H^2$. Specifically, Lindblad provided an example of ill-posedness for a quasilinear wave equation with $H^2$ initial data in \cite{counterexamplestolocalexistenceforquasilinearwaveequations}.  For the non-relativistic compressible Euler flow with vorticity and entropy, under the $H^{2+}$ assumptions on ``wave-part" and "transport part" of the data, Disconzi-Luo-Mazzone-Speck \cite{3DCompressibleEuler} and Wang \cite{Roughsolutionsofthe3-DcompressibleEulerequations} proved local well-posedness result for 3D compressible Euler equations. By assuming the H\"older regularity $C^{0,\alpha}$ for the data of the modified fluid variables $\modivort$ and $\DivGradEnt$ (our analogue of $\modivort$ and $\DivGradEnt$ are defined in Definition \ref{modifiedfluidvariables}), the authors are able to prove a Schauder estimate in \cite{3DCompressibleEuler} for a transport-div-curl system in order to propagate the vorticity and entropy gradient along the waves. A method of decomposing the velocity is given in \cite{Roughsolutionsofthe3-DcompressibleEulerequations} for the isentropic compressible Euler equations, which allowed Wang to remove the H\"older assumption on vorticity. In \cite{Alterproof}, Zhang-Andersson combined the methods in \cite{Sharplocalwell-posednessresultsforthenonlinearwaveequation} and \cite{Roughsolutionsofthe3-DcompressibleEulerequations} to give an alternate proof of the same result as in Wang \cite{Roughsolutionsofthe3-DcompressibleEulerequations}.
	\subsection{A Brief Overview of the Strategy of the Proof}
	Klainerman \cite{ACommutingVectorfieldsApproachtoStrichartz}, Klainerman-Rodnianski \cite{ImprovedLocalwellPosedness}, and Wang \cite{AGeoMetricApproach} developed a geometric approach for proving the low regularity well-posedness for the quasilinear wave equations. The new formulation (\ref{intro1})-(\ref{intro3}) provided by Disconzi-Speck \cite{TheRelativistivEulerEquations} makes it possible to import the geometric techniques from \cite{ACommutingVectorfieldsApproachtoStrichartz,ImprovedLocalwellPosedness,AGeoMetricApproach} to the ``sound-wave-part" of compressible Euler flow. The main difference with the wave problem is the addition of another characteristic speed into the problem, namely, the ``transport-part". These two parts of the equations and solutions interact with each other, which creates substantial difficulties for understanding the Euler flow. See Section \ref{introgeometricformulation} for further discussions of the geometric formulation. See also Luk-Speck \cite{ThehiddennullstructureofthecompressibleEulerequationsandapreludetoapplications} for similar formulations for 3D isentropic compressible Euler equation and Speck \cite{ANewFormulationofthe3DCompressibleEulerEquationswithDynamicEntropy} for 3D compressible Euler equations with any equation of state. The main tool for controlling the solution in the low-regularity setting, by using energy estimate (see Christodoulou \cite[Chapter 1]{Theformationofshocksin3-dimensionalfluids} for the energy current and its properties) and Littlewood-Paley theory, is the following estimates\footnote{We denote the constant-time hypersurface at time $t$ by $\Sigma_{t}$. Moreover, $A\lesssim B$ means $A\leq C\cdot B$ for some universal constant $C$ depending on region $\mathcal{R}$ and data $D$.}:
	
	\begin{align}\label{standardenergy}
	\norm{(\lnenth,\ent,v)}_{H^{2+\varepsilon}(\St)}&\lesssim\norm{(\lnenth,\ent,v)}_{H^{2+\varepsilon}(\Sigma_{0})}
	+\int_{0}^{t}\left(\norm{\pfour(\lnenth,\ent,v)}_{L_x^\infty(\Sigma_{\tau})}+1\right)\norm{(\lnenth,\ent,v)}_{H^{2+\varepsilon}(\Sigma_{\tau})}\diff\tau.
	\end{align} 
	In order to make (\ref{standardenergy}) useful, one needs to control $\norm{\pfour(\lnenth,\ent,v)}_{L^1_tL_x^\infty}$.
	Since one is not able to apply Sobolev embedding to recover the bound below $H^{5/2+}$, we instead use a geometric approach to show the following \textbf{Strichartz estimates}: $\norm{\pfour(\lnenth,\ent,v)}_{L^2_tL_x^\infty}\lesssim\Tstar^{2\delta}$. This is done by a bootstrap argument with bootstrap assumptions $\norm{\pfour(\lnenth,\ent,v)}_{L^2_tL_x^\infty}\leq1$, where $\Tstar$ is the bootstrap time. In order to prove the Strichartz estimates, we apply a series of reductions. We reduce the Strichartz estimates to a decay estimates by using a $\mathcal{T}\mathcal{T}^*$ argument, then to a conformal energy estimate (see Definition \ref{definitionofconformalenergy} for the definition of conformal energy and Theorem \ref{BoundnessTheorem} for boundness theorem for conformal energy) by Littlewood-Paley theory. See Section \ref{reductionofstrichartz} for overview of the reduction, Section \ref{sectionstructure} for an extended overview of a global structure and Section \ref{sectionreduction}-\ref{sectiongeometrysetup} for details. 
	
	A crucial step in our geometric approach is the introduction of an acoustical function $u$ satisfying the acoustical eikonal equation $\left(\gfour^{-1}\right)^{\alpha\beta}\p_\alpha u\p_\beta u=0$, where the acoustical metric $\gfour=\gfour(\lnenth,v,\ent)$ is a Lorentzian metric (see Definition \ref{D:acousticmetric} for the definition of $\gfour$) distinct from the Minkowski metric. With the help of $u$, we construct a null frame and control the acoustic geometry along acoustic null cones, which are the level sets of $u$ (see the figure on page \pageref{figuregeometry}). This allows us to derive suitable estimates for conformal energy. Disconzi-Luo-Mazzone-Speck \cite{3DCompressibleEuler} and Wang \cite{Roughsolutionsofthe3-DcompressibleEulerequations} showed that given good control over the ``transport-part", we can run the machinery of Strichartz estimates for the ``wave-part" where we treat the ``transport-part" as a favorable source term. Thus, good control over the ``transport-part" is a crucial component of our analysis. Inspired by the analysis of the 3D non-relativistic compressible Euler case in \cite{3DCompressibleEuler}, we derive elliptic and Schauder estimates for the transport-div-curl systems to bound the $H^{1+\varepsilon}$ and $C^{0,\alpha}$ norms of $\modivort,\DivGradEnt$ when the wave part is rough.
	
	 The main new difficulty that is not found in 3D non-relativistic compressible Euler is: due to the space-time structure of the relativistic Euler flow, we encounter space-time velocity, vorticity, and the div-curl system where the ellipticity is not immediately apparent. To overcome this difficulty, we exploit two crucial aspects. We first notice that the $v$-directional derivative of the vorticity and entropy gradient is favorable due to the transport phenomena. To obtain control of $v$-orthogonal directional derivatives, we reduce the space-time div-curl system of vorticity and entropy gradient to a dynamic div-curl system on the constant-time hypersurfaces. By combining these special structures of relativistic Euler equations with Littlewood-Paley decomposition and properties of pseudodifferential operators, we derive estimates for vorticity and entropy.

	We present the logical graph of this paper in Section \ref{mainidea}.
	\subsection{Geometric Formulation of the Relativistic Euler Equations}\label{introgeometricformulation} Due to the coupling of sound waves with vorticity and entropy in the equations (\ref{intro1})-(\ref{intro3}), when considering the relativistic Euler equations with an arbitrary equation of state, 
	one needs to precisely and carefully split the dynamics into a ``wave-part", which describes the propagation of sound waves, and a ``transport-part", which describes the evolution of vorticity and entropy. For the 3D non-relativistic compressible Euler equations with any equation of state, Speck \cite{ANewFormulationofthe3DCompressibleEulerEquationswithDynamicEntropy} derived a system consisting of geometric wave equations and transport-div-curl equations. This geometric formulation is used for the low-regularity problem in \cite{3DCompressibleEuler,Roughsolutionsofthe3-DcompressibleEulerequations}. See also Luk-Speck \cite{Shockformationinsolutionstothe2DcompressibleEulerequationsinthepresenceofnon-zerovorticity} and \cite{ThehiddennullstructureofthecompressibleEulerequationsandapreludetoapplications} for the geometric formulation of the compressible Euler in the barotropic case and its application to the shock formation problem. Disconzi-Speck \cite{TheRelativistivEulerEquations} derived the geometric formulation of the relativistic Euler equations with vorticity and dynamic entropy that we used in this paper. It allows us to describe the influence of transport phenomena on the wave part of the system and the acoustic geometry with rough sound wave data given in the relativistic Euler flow. These geometric formulations have origins in Christodoulou and Christodoulou-Miao's proof of stable shock formation for the relativistic Euler equations and non-relativistic 3D compressible Euler equations in the irrotational and isentropic case \cite{Theformationofshocksin3-dimensionalfluids}\cite{CompressibleflowandEuler'sequations}. 

	The new geometric formulation from \cite{TheRelativistivEulerEquations} (see (\ref{intro1})-(\ref{intro3}) and Proposition \ref{WaveEqunderoringinalacoustical} for the formulation) splits the dynamics into a ``wave-part", which consists of geometric wave equations for the fluid variables $\lnenth, v, \ent$, and a ``transport-div-curl-part", which governs the transport equations of special vorticity, entropy gradient and modified fluid variables $\modivort, \DivGradEnt$ ($\modivort, \DivGradEnt$ are special combinations of variables whose essential terms are the vorticity of vorticity and divergence of entropy gradient that are defined in Definition \ref{modifiedfluidvariables}), and div-curl systems for the special vorticity and entropy gradient. The advantage of the geometric formulation is that one can do analysis on both ``wave-part" and ``transport-part", which are highly coupled. Here we briefly summarize the new formulation and its connection to establishing the Strichartz estimate: 
\begin{itemize}

\item  The ``wave-part" of the formulation involves wave equations with principal part $\boxg$. Properties of this operator are intimately related to the acoustic geometry, which is constructed via an acoustical function $u$. Here, $u$ is a solution to the acoustical eikonal equation $\left(\gfour^{-1}\right)^{\alpha\beta}\p_\alpha u\p_\beta u=0$, where the acoustical metric $\gfour=\gfour(\lnenth,v,\ent)$ is the Lorentzian metric defined in Definition $\ref{D:acousticmetric}$. With the help of $u$, we construct a null frame and derive some transport equations as well as div-curl systems for some particular connection coefficients along acoustic null cones, which are the level sets of $u$ (see the figure on page \pageref{figuregeometry}). We note that these equations for the connection coefficients are derived from basic geometry considerations and are independent of the relativistic Euler equations. By using a delicate decomposition of certain curvature components, which are highly tied to the geometric wave equations (\ref{waveEQ}), we can control a large group of geometric quantities that are fundamental for deriving the conformal energy estimates. We emphasize already that achieving control of these geometric quantities is essential for controlling certain conformal energy for solutions to the linear wave equation corresponding to the acoustical metric $\gfour$, i.e., solutions $\varphi$ to the PDE $\square_{\gfour(\vvariables)}\varphi=0$. It is crucial to control the conformal energy in order to derive the decay estimates, which we again emphasize are the main ingredient needed to obtain the desired Strichartz estimate. We will describe conformal energy and decay estimates with more details in Section \ref{mainidea}.
\item The ``transport-div-curl-part" of the formulation allows one to control the vorticity and entropy at one derivative level above standard estimates. The analysis uses transport estimates as well as Hodge estimates at constant-time hypersurfaces. This is highly non-trivial and more complicated compared to non-relativistic 3D compressible Euler because the Hodge system that we encounter is a space-time div-curl system. In total, we are able to show that the transport terms are ``good" source terms in the wave equation estimates. We point out the vorticity and entropy gradient also appear in PDEs that we use to control the acoustic geometry because of the geometric wave equations (\ref{intro1}). This shows that there are interactions between the vorticity, entropy, sound waves, and acoustic geometry.
\end{itemize}
	\subsection{Comparision with Low-Regularity Results for 3D Non-relativistic Compressible Euler Equations}
	Recently, \cite{3DCompressibleEuler} and \cite{Roughsolutionsofthe3-DcompressibleEulerequations} proved low-regularity results for the 3D non-relativistic compressible Euler equations with the help of the geometric formulation in \cite{ANewFormulationofthe3DCompressibleEulerEquationswithDynamicEntropy}. In \cite{3DCompressibleEuler}, Disconzi-Luo-Mazzone-Speck showed that, if the ``wave-part" of the data is initially in $H^{2+}$, and the ``transport-part" $\modivort, \DivGradEnt$ are in $C^{0,\alpha}$, then the regularity of the solutions can be preserved for short times. For barotropic flow, Wang proved a similar result by removing the H\"older assumption for $\modivort,\DivGradEnt$ and assuming the $H^{2+}$ assumptions on the ``transport-part" in \cite{Roughsolutionsofthe3-DcompressibleEulerequations}.
	For the relativistic Euler equations, we prove a similar result as in \cite{3DCompressibleEuler} for the 3D compressible Euler equations. That is, we allow any equation of state and we have the same level of regularity assumptions on the initial data. Due to the geometric nature of the relativistic Euler flow, the vorticity $\vort^\alpha$ in this article is a space-time $v$-orthogonal vectorfield (see Definition \ref{vorticity}) which solves a space-time transport-div-curl system. In the 3D non-relativistic compressible Euler case, the geometry of vorticity is much simpler: it is a $\St$-tangent vector field and solves a div-curl system with constant coefficients on constant-time slices. 
	
	In the relativistic Euler equations, vorticity and entropy gradient satisfy transport equations in the $v$-direction. To control generic $v$-orthogonal (with respect to Minkowski metric) derivatives of vorticity and entropy gradient, we rely on a space-time div-curl system. Moreover, using the transport equations satisfied by the modified fluid variables $\modivort,\DivGradEnt$, we can control these quantities not only along constant-time slices, but also along null cones, which is fundamental in our work. To derive sufficient regularity for vorticity, entropy gradient, and modified fluid variables along constant-time slices, we rewrite the space-time div-curl system into a spatial div-curl system with source terms that can be controlled only due to the special structure of the equations. A crucial ingredient in our analysis is that the spatial divergence equation has the form $(G^{-1})^{ij}\p_i(\vort_\flat)_j=...$ where $G^{-1}:=G^{-1}(v)$ is the inverse of a Riemannian metric on constant-time hypersurfaces (see equation (\ref{D:Gab}) for the definition of $G^{-1}$).
	Because the coefficient metric $G$ of the divergence equation is Riemannian, by using the technique of freezing the spatial points, we are able to derive a localized div-curl system with constant coefficient principle terms, such that the Fourier transform of vorticity and entropy gradient is bounded in the frequency space by the source terms of the div-curl system. This allows us to control appropriate H\"older norms of $\p\vort,\p\gradEnt$ in terms of the same H\"older norms of $\modivort,\DivGradEnt,\pfour\vvariables$. The analysis relies on the Littlewood-Paley theory as well as the standard theorem in pseudodifferential operators. We take a similar approach when deriving the elliptic div-curl estimates in $L^2$ space, where we need to control derivatives of vorticity and the entropy gradient by $\pfour\vvariables$, the modified fluid variables $\modivort$ and $\DivGradEnt$. Finally, we use the transport equations (\ref{intro3.1})-(\ref{intro3.2}) and initial assumptions of $\modivort,\DivGradEnt$ to bound the H\"older norm of $\modivort,\DivGradEnt$ by $\p\vort,\p\gradEnt$ to close the estimates for $\p\vort,\p\gradEnt$.
	\subsection{Main Idea of the Proof of Theorem \ref{IntroMainTheorem}}\label{mainidea}
	
	Theorem \ref{IntroMainTheorem} provides a priori estimates for smooth solutions, which is needed for a full proof of local well-posedness. The remaining aspects of a full proof of local well-posedness could be shown by deriving uniform estimates for sequences of smooth solutions and their differences. We refer readers to \cite[Sections 2-3]{Sharplocalwell-posednessresultsforthenonlinearwaveequation} for the proof of local well-posedness based on a priori estimates.

In this subsection, we present the logic of proofs in this paper, that is, the bootstrap argument. The colored steps involve new ingredients, where we need to do analysis based on the special structure of the relativistic Euler equations (see Subsection \ref{standardenergyestimates} and \ref{introsectionschauder} for a discussion of these steps). The uncolored steps are introduced in the previous low-regularity problem works (see Subsections \ref{reductionofstrichartz}-\ref{introcontrol} for a discussion of these steps). We emphasized that, with the estimates we derive in the colored steps, the proofs of the uncolored steps are essentially the same as in \cite{3DCompressibleEuler,AGeoMetricApproach,ImprovedLocalwellPosedness,Equationsd'ondesquasilineairesetestimationdeStrichartz,Equationsd'ondesquasilinearireseteffetdispersif,Sharplocalwell-posednessresultsforthenonlinearwaveequation,Strichartzestimatesforsecondorderhyperbolicoperators,ACommutingVectorfieldsApproachtoStrichartz}. Hence, in this work, we provide all of the details for the colored steps, and give terse sketches for the uncolored steps with the appropriate citations.
\\

\begin{tikzpicture}
\tikzstyle{process} = [rectangle, minimum width=3cm, minimum height=1cm, text centered, text width=4cm, draw=black]
\node(ba)[process,xshift=3cm]{\hyperlink{BA}{Bootstrap assumptions} on the ``wave-part" and the ``transport-part"};
\node(ee)[process,fill={rgb:red,1,;yellow,1},below of=ba, xshift=-3cm, yshift=-1cm] {Energy estimates for fluid variables};
\node (cg)[process,below of=ee, yshift=-1cm]{Controlling of the acoustic geometry};
\node(cee)[process,below of=cg, yshift=-1cm] {Conformal energy estimates};
\node(de) [process,below of=cee, yshift=-1cm]{Decay estimates};
\node(se)[process,right of=de, xshift=5cm] {Strichartz estimates};
\node(end)[process,right of=se, xshift=4cm] {End of the bootstrap argument};
\node(te) [process,fill={rgb:red,1,;yellow,1},right of=ba, yshift=-2cm, xshift=2cm]{Transport-Schauder estimates for vorticity $\vort$ and entropy gradient $\gradEnt$};
\node(iw) [process,above of=se, yshift=2cm]{Improvements of bootstrap assumptions on the ``wave-part"};
\node(it) [process,right of=iw, xshift=4cm]{Improvements of bootstrap assumptions on the ``transport-part"};
\draw[thick,->](ba)--(ee);
\draw[thick,->](ee)--(cg);
\draw[thick,->](cg)--(cee);
\draw[thick,->](cee)--(de);
\draw[thick,->](de)--(se);
\draw[thick,->](ba)--(te);
\draw[thick,->](se)--(iw);
\draw[thick,-](6,-2.7)--(9.25,-3.82);
\draw[thick,->](iw.north) to[out=30, in=150] (it.north);
\draw[thick,-](10,-5.5)--(8.5,-6.5);
\draw[thick,->](iw)--(end);
\end{tikzpicture}

	\subsubsection{Overview of Elliptic and Energy Estimates}\label{standardenergyestimates}In this subsection, we provide an overview of how energy estimates work and are related to the bootstrap assumptions (\ref{boot1})-(\ref{boot2}). We provide representative energy estimates for wave variables, vorticity, and entropy gradient by using the basic energy estimates (see Section \ref{energymethodforwave}) and $L^2$ elliptic estimates (see Section \ref{EllipticDiv-Curl}). Then we leave the discussion of the key assumptions to future subsections and detailed estimates in Section \ref{Section4.1}.

	We first consider the energy estimates for the ``wave-part". Given any $2<N<5/2$, $0<t\leq\Tstar$ where $0<\Tstar\ll1$ denotes the bootstrap time. By the vectorfield multiplier method and Littlewood-Paley calculus applied to the equations (\ref{intro1}), we derive the following energy estimates for the ``wave-part":

	\begin{align}\label{introenergy1}
	\sobolevnorm{\pfour\vvariables}{N-1}{\Sigma_{t}}^2\lesssim\sobolevnorm{\pfour\vvariables}{N-1}{\Sigma_{0}}^2+\int_{0}^{t}\left(\norm{\pfour\vvariables}_{L^\infty_x(\Sigma_{\tau})}+1\right)\sobolevnorm{\pfour\vvariables, \vectvort,\DivGradEnt
	 }{N-1}{\Sigma_{\tau}}^2\diff\tau,
	\end{align}
	which is an analogue of (\ref{standardenergy}). We will provide a detailed expression and its proof in Section \ref{Discussionofprop5.1}.

	To control $\vectvort,\DivGradEnt$ on the right-hand side of (\ref{introenergy1}), we then consider the energy estimates for the ``transport-part". We first need an important $L^2$ elliptic div-curl estimate
	\begin{align}\label{introelliptic}
	\norm{(\p\vvort,\p\vgradientEnt)}_{L_x^2(\St)}\lesssim\norm{\pfour\vvariables,\vectvort,\DivGradEnt}_{L_x^2(\St)}.
	\end{align}
	
	We note that the proof of (\ref{introelliptic}) requires a rewriting div-curl system (\ref{Gab0})-(\ref{Hab0}) for vorticity and entropy gradient, where one has to exploit the structure of the relativistic Euler equations. By splitting the space-time div-curl systems into time and spatial directions of derivatives, taking the advantage of the transport equations (\ref{intro2.1})-(\ref{intro2.2}) for $\vort$ and $\gradEnt$, we write time derivative of vorticity and entropy gradient components as a combination of spatial derivatives of $\vort$ and $\gradEnt$. We obtain a new spatial div-curl system of the form:

\begin{subequations}\label{introGH}
		\begin{align}
	\label{GAB}(G^{-1})^{ab}\p_a((\vort_\flat),\gradEnt)_b&=F,\\
	\label{HAB}\p_a((\vort_\flat),\gradEnt)_b-\p_b((\vort_\flat),\gradEnt)_a&=H_{ab},
	\end{align}
\end{subequations}
	where $G^{-1}=G^{-1}(v)$ is Riemannian, $F=\DivGradEnt+l.o.t.$ and $H=\modivort+l.o.t.$. (\ref{introGH}) is a PDE system on constant-time slices. Notice that equation (\ref{GAB}) is a quasilinear divergence equation while the analogue in \cite{3DCompressibleEuler} is a constant-coefficient equation.
	Then, by Littlewood-Paley estimates and a partition of unity argument, we prove (\ref{introelliptic}) in Proposition \ref{EllipticestimatesinL2}.
	
	As in Proposition \ref{EllipticestimatesinL2}, by (\ref{introelliptic}) and Littlewood-Paley calculus, we also have, for $2<N<5/2$,
	\begin{align}\label{introelliptic2}
	\sobolevnorm{(\p\vvort,\p\vgradientEnt)}{N-1}{\St}\lesssim\sobolevnorm{\pfour\vvariables,\vectvort,\DivGradEnt}{N-1}{\St}.
	\end{align}
	
	By applying energy estimates and Littlewood-Paley calculus on evolution equations (\ref{intro3.1})-(\ref{intro3.2}) for $\vectvort,\DivGradEnt$, we have the following energy estimate for $\vectvort,\DivGradEnt$:
	\begin{align}\label{introenergy2}
	\sobolevnorm{\vectvort,\DivGradEnt}{N-1}{\Sigma_{t}}^2\lesssim\sobolevnorm{\vectvort,\DivGradEnt}{N-1}{\Sigma_{0}}^2+\int_{0}^{t}\left(\norm{\pfour\vvariables,\pfour\vvort,\pfour\vgradientEnt}_{L^\infty_x(\Sigma_{\tau})}+1\right)\sobolevnorm{\pfour\vvariables,\pfour\vvort,\pfour\vgradientEnt,\vectvort,\DivGradEnt
	}{N-1}{\Sigma_{\tau}}^2\diff\tau.
	\end{align}
	By elliptic estimates (\ref{introelliptic2}), (\ref{introenergy2}) and (\ref{introenergy1}), we have 
	\begin{align}\label{introenergy3}
	\sobolevnorm{\pfour\vvort,\pfour\vgradientEnt}{N-1}{\Sigma_{t}}^2&\lesssim\sobolevnorm{\pfour\vvariables,\vectvort,\DivGradEnt}{N-1}{\Sigma_{t}}^2\\
	\notag&\lesssim\sobolevnorm{\pfour\vvariables,\vectvort,\DivGradEnt}{N-1}{\Sigma_{0}}^2+\int_{0}^{t}\left(\norm{\pfour\vvariables,\pfour\vvort,\pfour\vgradientEnt}_{L^\infty_x(\Sigma_{\tau})}+1\right)\sobolevnorm{\pfour\vvariables,\pfour\vvort,\pfour\vgradientEnt,\vectvort,\DivGradEnt
	}{N-1}{\Sigma_{\tau}}^2\diff\tau.
	\end{align}
	
	The results of energy estimates are obtained in Subsection \ref{Discussionofprop5.1}.
	
	From (\ref{introenergy1}), (\ref{introenergy2}), (\ref{introenergy3}) and Gr\"onwall's inequality, we see that if $\twonorms{\pfour\vvariables,\pfour\vvort,\pfour\vgradientEnt}{t}{1}{x}{\infty}{\region}$ is  bounded (note that $\modivort\simeq\text{vort}(\vort)+l.o.t,\DivGradEnt\simeq\text{div}\gradEnt+l.o.t$), the Sobolev regularity of the data can be propogated by the solution. This drives us to set up a bootstrap argument with the bootstrap assumptions as introduced in the next subsection.
	\subsubsection{Bootstrap Assumptions}
	As we made it clear in the previous subsection, our argument crucially relies on the boundness of the term $\twonorms{\pfour\vvariables,\pfour\vvort,\pfour\vgradientEnt}{t}{1}{x}{\infty}{\region}$. 
We prove the boundness of this via a bootstrap argument that we now describe:\\
	
	\hypertarget{BA}{Throughout the paper,} $0<\Tstar\ll1$ denotes the bootstrap time. We assume that $\vvariables,\vvort,\vgradientEnt$ is a smooth solution to the relativistic Euler equations. For $\delta_0>0$ defined as in Section \ref{ChoiceofParameters}, we assume the following estimates hold:
\begin{subequations}
	\begin{align}
	\label{boot1}\twonorms{\pfour\vvariables}{t}{2}{x}{\infty}{[0,\Tstar]\times\mathbb{R}^3}^2+\sum\limits_{\upnu\geq2}\upnu^{2\delta_0}\twonorms{\littlewood\pfour\vvariables}{t}{2}{x}{\infty}{[0,\Tstar]\times\mathbb{R}^3}^2&\leq1,\\
	\label{boot2}\twonorms{\p\vvort,\p\vgradientEnt}{t}{2}{x}{\infty}{[0,\Tstar]\times\mathbb{R}^3}^2+\sum\limits_{\upnu\geq2}\upnu^{2\delta_0}\twonorms{\littlewood\p\vvort,\littlewood\p\vgradientEnt}{t}{2}{x}{\infty}{[0,\Tstar]\times\mathbb{R}^3}^2&\leq1,
	\end{align}
\end{subequations}
where $\littlewood$ is the Littlewood-Paley projection (see Section \ref{definitionlittlewood} for definition).
	 The Littlewood-Paley terms in the assumptions are needed for establishing the dyadic Strichartz estimate in order to improve the bootstrap assumptions.\\
	 \indent
	 In the classical local well-posedness problem of the relativistic Euler equations, the regularity assumptions are $(\lnenth, v, \ent)\in H^{(5/2)+}(\Sigma_{0})$. This gives the $\pfour\vvariables\in H^{(3/2)+}(\Sigma_{0})$. One can recover the boundness assumption $\norm{\pfour\vvariables}_{L_{x}^{\infty}(\Sigma_{t})}$ by standard energy estimates and Sobolev embedding $H^{3/2+}\hookrightarrow L^\infty$ at any constant-time hypersurface $\Sigma_{t}$. The lack of Sobolev embedding in the low-regularity level forces one to find a new machinery to improve the reasonable bootstrap assumptions (see Section \ref{Overviewofpreviousresults} for introduction of previous results). Recovering the bootstrap assumptions occupies a large part of this article.
	 
	 \subsubsection{Transport-Schauder Estimates for the Transport-div-curl System}\label{introsectionschauder}
	 In this subsection, we explain how to improve the bootstrap assumption (\ref{boot2}). In particular, by applying H\"older's inequality in time, this will show $\twonorms{\pfour\vvort,\pfour\vgradientEnt}{t}{1}{x}{\infty}{[0,\Tstar]\times\mathbb{R}^3}$ is small. This improvement is conditional on (\ref{IBA}), which we explain how to derive in Subsection \ref{reductionofstrichartz}.
	 
	 Because of the lack of tools in Hodge estimates in $L^\infty$ space, we have assumed a slight bit of extra regularity for the ``transport-part". That is, we propagate the H\"older boundness of ``transport-part" with given initial data $\holdernorm{\vectvort,\DivGradEnt}{x}{0}{\alpha}{\Sigma_0}\leq D$ where $0<\alpha<1$ and $D\in\mathbb{R}$. Besides using the transport equations (\ref{intro2.1})-(\ref{intro2.2}) which exhibit source terms with surprisingly good structures, we also rely on the following Schauder-type estimate:
	 \begin{align}
	 \label{introschauder}\holdernorm{\p\vvort,\p\vgradientEnt}{x}{0}{\alpha}{\Sigma_{t}}\lesssim\holdernorm{\pfour\vvariables,\vectvort,\DivGradEnt}{x}{0}{\alpha}{\Sigma_{t}}.
	 \end{align}
	 In order to control $\modivort, \DivGradEnt$ on the right-hand side of (\ref{introschauder}), we use the transport equations (\ref{intro3.1})-(\ref{intro3.2}) of $\modivort$ and $\DivGradEnt$, which are coupled to $\p\omega,\p\gradEnt$ (see equation (\ref{intro3.1})-(\ref{intro3.2})). By combining the two, under bootstrap assumptions (\ref{boot1}), we can apply Gr\"onwall's inequality to bound the vorticity and entropy gradient of the relativistic Euler flow. We will discuss this approach in more details below. 
	 
	 To derive Schauder estimates, we split the derivative of the vorticity $\vort$ and entropy gradient $\gradEnt$ (see Definition \ref{vorticity} and \ref{D:entgradient} for the definition of $\vort$ and $\gradEnt$) into $v$-tangent direction and the $\St$-tangent directions ($v$ is transversal to $\St$ with respect to Minkowski metric). Now we highlight the following two features in our analysis:
	 \begin{itemize}
	 	\item The $v$-tangent direction derivatives of vorticity and entropy gradient are favorable because of the transport phenomena. That is, by using transport equations (\ref{intro2.1})-(\ref{intro2.2}), we are able to obtain the H\"older bound for $v$-tangent direction of vorticity and entropy gradient by using bootstrap assumptions.
	 	\item To control the $\St$-tangent directional derivatives of vorticity and entropy gradient, we rely on a space-time transport-div-curl system for $\vort$ and $\gradEnt$. We note that it is qualitatively distinct from the case in \cite{3DCompressibleEuler} for 3D non-relativistic compressible Euler equations where the div-curl equations are spatial with constant coefficients.
	 \end{itemize}
	  We now explain how we derive Schauder estimates (\ref{introschauder}). We use the same div-curl system (\ref{GAB})-(\ref{HAB}) as in the $L^2$ elliptic estimates. By partition of unity, Fourier transform, Littlewood-Paley theory and properties of pseudodifferential operators, we are able to bound $\holdernorm{\pfour\vvort,\pfour\vgradientEnt}{x}{0}{\alpha}{\Sigma_{t}}$ by $\holdernorm{\pfour\vvariables,\vectvort,\DivGradEnt}{x}{0}{\alpha}{\Sigma_{t}}$ as in (\ref{introschauder}), see Lemma \ref{Schaudertypeestimate} for the detailed proof. Then we bound $\holdernorm{\vectvort,\DivGradEnt}{x}{0}{\alpha}{\Sigma_{t}}$ by applying transport equations (\ref{intro3.1})-(\ref{intro3.2}) as follows:
	  \begin{align}\label{introtrans}
	  \holdernorm{\vectvort,\DivGradEnt}{x}{0}{\delta_1}{\Sigma_{t}}
	  \lesssim1+\int_0^t\left(\holdernorm{\pfour\vvariables}{x}{0}{\delta_1}{\Sigma_\tau}+1\right)\holdernorm{\pfour\vvariables,\pfour\vvort,\pfour\vgradientEnt}{x}{0}{\delta_1}{\Sigma_\tau}\diff\tau.
	  \end{align}
	 Finally, combining (\ref{introschauder}) and (\ref{introtrans}), we use Gr\"onwall's inequality and bootstrap assumptions to close the transport-Schauder type estimates 
	 \begin{align}\label{introschauder2}
	 \holdernorm{\pfour\vvort,\pfour\vgradientEnt}{x}{0}{\alpha}{\Sigma_{t}}\lesssim\holdernorm{\pfour\vvariables}{x}{0}{\alpha}{\Sigma_{t}}.
	 \end{align}
	 
	 We emphasize that later in the argument, we will integrate (\ref{introschauder2}) in time and combine it with the improved Strichartz estimate (\ref{IBA}) (which is obtained independently of (\ref{introschauder2})). These lead to a strict improvement of the bootstrap assumptions (\ref{boot2}). We provide full details of the Schauder estimates in Section \ref{section4}.
	 \subsubsection{Reductions of the Strichartz Type Estimates}
	 \label{reductionofstrichartz}Our argument above crucially relies on bounding $\twonorms{\pfour\vvariables}{t}{1}{x}{\infty}{\region}$. In this subsection, we explain how we derive strict improvements of bootstrap assumptions (\ref{boot1}), that is, we describe how to derive Strichartz estimates (\ref{IBA}). By taking the advantage of the smallness of bootstrap time interval $[0,\Tstar]$, we will improve our bootstrap assumptions to the following Strichartz estimates:
	 \begin{align}\label{IBA}
	 \twonorms{\pfour\vvariables}{t}{2}{x}{\infty}{[0,\Tstar]\times\mathbb{R}^3}^2+\sum\limits_{\upnu\geq2}\upnu^{2\delta_1}\twonorms{\littlewood\pfour\vvariables}{t}{2}{x}{\infty}{[0,\Tstar]\times\mathbb{R}^3}^2\lesssim\Tstar^{2\delta},
	 \end{align}
	 where $\delta>0$ is sufficiently small as in Subsection \ref{ChoiceofParameters} and $8\delta_0<\delta_1<N-2$, where $\delta_0$ is from the bootstrap assumptions (\ref{boot1}) and (\ref{boot2}). Notice that if $\Tstar$ is small, (\ref{IBA}) is a strict improvement of (\ref{boot1}). We reduce the proof of (\ref{IBA}) to the proof of estimates on the acoustic geometry by adopting the geometric approach of \cite{AGeoMetricApproach}. This reduction is done through the following steps: Improvement of bootstrap assumptions$\longleftarrow$Strichartz estimates$\longleftarrow$Decay estimates$\longleftarrow$Conformal energy estimates$\longleftarrow$Controlling of the acoustic geometry, where the left arrow indicates that the latter estimate implies the former. We remind readers of the logic diagram at the beginning of Section \ref{mainidea}.
	 \begin{itemize}
	 	\item \textit{Reduction to dyadic Strichartz estimates.} The first step in the proof of (\ref{IBA}) is to reduce Strichartz estimates to a dyadic Strichartz estimate. Specifically, for a fixed large frequency $\lambda$, we partition $[0,\Tstar]$ into disjoint union of sub-intervals $I_{k}:=[t_{k-1},t_k]$ of total number $\lesssim\lambda^{8\varepsilon_0}$ with $|I_k|\lesssim\lambda^{-8\varepsilon_0}$ (see Section \ref{ChoiceofParameters} for the definition of $\varepsilon_0$). By Littlewood-Paley decomposition and Duhamel principle, the proof of (\ref{IBA}) can be reduced to a dyadic Strichartz estimate
	 	\begin{align}\label{DyadicStrichartz}
	 	\twonorms{P_\lambda\pfour\varphi}{t}{q}{x}{\infty}{[\tau,t_{k+1}]\times\mathbb{R}^3}\lesssim\lambda^{\frac{3}{2}-\frac{1}{q}}\norm{\pfour\varphi}_{L_x^2(\Sigma_{\tau})},
	 	\end{align}
	 	where $\varphi$ is a solution of geometric equation
	 	\begin{align}
	 	\boxg\varphi=0,
	 	\end{align}
	 	on the time interval $I_k$. In (\ref{DyadicStrichartz}), $q>2$ is any real number which is sufficiently close to $2$ and $\tau\in[t_{k},t_{k+1}]$. Here, we focus on large frequencies since control of small frequency is easier due to Berstein inequalities.
	 	\item \textit{Reduction to decay estimates.} For a large frequency $\lambda$, by rescaling the coordinates (see Section \ref{SelfRescaling} for the rescaling) and using an abstract $\mathcal{T}\mathcal{T}^*$ argument, we can reduce the dyadic Strichartz estimates (\ref{DyadicStrichartz}) to  $L^2-L^\infty$ decay estimates at any $t\in[0,\Trescale]$ where $\Trescale$ is the rescaled bootstrap time (see Section \ref{SelfRescaling} for the definition of $\Trescale$):
	 	\begin{align}\label{introdecayestimates}
	 	\norm{P_1\Timelike\varphi}_{L_x^\infty(\St)}\lesssim\left(\frac{1}{(1+\abs{t-1})^{\frac{2}{q}}}+d(t)\right)\left(\norm{\pfour\varphi}_{L_x^2(\Sigma_1)}+\norm{\varphi}_{L_x^2(\Sigma_1)}\right),
	 	\end{align}
	 	where the timelike vectorfied $\Timelike$ is $\gfour$-unit normal to $\St$ (that is defined in Definition \ref{timelike}) and  $\varphi$ is an arbitrary solution to the equation $\boxg\varphi=0$ on the time interval $[0,\Trescale]\times\mathbb{R}^3$ with $\varphi(1,x)$ supported in the Euclidean ball $B_{R}$ (see Theorem \ref{Spatiallylocalizeddecay} for detailed definition of $R$).
	 	Moreover, the function $d(t)$ satisfies
	 	\begin{align}
	 	\norm{d}_{L_t^{\frac{q}{2}}([0,\Trescale])}\lesssim1,
	 	\end{align}
	 	for $q>2$ sufficiently close to 2.
	 	\item \textit{Reduction to conformal energy estimates.} By product estimates and Littlewood-Paley theory, we reduce the proof of (\ref{introdecayestimates}) to a proof of following estimates for the conformal energy $\Energyconformal[\varphi](t)$ (see Section \ref{definitionofconformalenergy} for the definition of the conformal energy $\Energyconformal[\varphi](t)$) at time $t\in[1,\Trescale]$:
	 	\begin{align}\label{introconformalestimates}
	 	\Energyconformal[\varphi](t)\lesssim_{\varepsilon}(1+t)^{2\varepsilon}\left(\norm{\pfour\varphi}^2_{L_x^2(\Sigma_{1})}+\norm{\varphi}^2_{L_x^2(\Sigma_{1})}\right),
	 	\end{align}
	 	where $\varphi$ is an arbitrary solution to the equation $\boxg\varphi=0$ on $[0,\Trescale]\times\mathbb{R}^3$ with $\varphi(1)$ supported in $B_R\subset\intregion\bigcap\Sigma_{1}$ (see Section \ref{opticalfunction} for the definition of $\intregion$) where $\varepsilon>0$ is an arbitrary small number.
	 \end{itemize}
 
\textbf{ We emphasize that both the reduction of (\ref{introdecayestimates}) to (\ref{introconformalestimates}), as well as the very definition of $\Energyconformal[\varphi]$ require the acoustic geometry, where its sharp control is needed for deriving (\ref{introconformalestimates}).} We describe how to obtain such control in Subsection \ref{introgeometry}-\ref{introcontrol}.
	 We provide an overview of the structure over the reductions in Section \ref{sectionstructure} and more detailed discussions in Section \ref{sectionreduction}.
	\subsubsection{Structures for the Causal Geometry of the Acoustic Space-time}\label{introgeometry}
	In order to reduce the decay estimates to conformal energy estimates (see Subsection \ref{introconformal} and \ref{structurenullenergy} for introduction and Section \ref{subsectionconformalenergy} for details), one needs sharp information about the acoustic geometry. In this subsection, we discuss the geometric framework that is crucial for our analysis. This part of the result is well-known and standard (see Section \ref{Overviewofpreviousresults} for the introduction of the previous results). The central object of our geometric framework is the acoustical function $u$, which is defined as a solution of the acoustical eikonal equation $\left(\gfour^{-1}\right)^{\alpha\beta}\p_\alpha u\p_\beta u=0$, where $\gfour^{-1}$ is the inverse of the normalized acoustical metric. We denote the level sets of $u$ by $\coneu$, which are forward truncated null cones (defined in Section \ref{opticalfunction}).
	
	We construct a null frame, which consists of a null pair $\Lunit, \uLunit$ and two spherical vectorfields $\{e_A\}_{A=1,2}$ (see Section \ref{geometricquantities} for detailed definitions). We derive transport and Hodge type equations for the Ricci coefficients. An important example is the Raychaudhuri equation (see Section \ref{connectioncoefficientsandpdes} for definitions of connection coefficients and PDEs verified by geometric quantities):
	\begin{equation}\label{introRaychau}
	\Lunit\gtr\upchi+\frac{1}{2}(\gtr\upchi)^2=-\sgabs{\hchi}^2-k_{\spherenormal\spherenormal}\gtr\upchi-\Ricfour{\Lunit}{\Lunit}.
	\end{equation}
	With the help of a remarkable decomposition of the Ricci curvature tensor
	\begin{align}\label{introRicci}
	\Ricfour{\alpha}{\beta}=-\frac{1}{2}\boxg \gfour_{\alpha\beta}(\vvariables)+\frac{1}{2}\left(\Dfour_\alpha\Chfour_\beta+\Dfour_\beta\Chfour_\alpha\right)+\quadsmoothfunction(\vvariables)[\pfour\vvariables,\pfour\vvariables],
	\end{align}
	and the Bianchi identities, where we can substitute $\boxg \gfour_{\alpha\beta}(\vvariables)$ in (\ref{introRicci}) by using relativistic Euler equations (\ref{intro1}), we are able to obtain some estimates for the Ricci coefficients, which we will utilize in the following subsections. We emphasize that it is important to have exactly $\modivort,\DivGradEnt$ on the right-hand side of (\ref{intro1}). $\modivort,\DivGradEnt$ satisfies the transport equations (\ref{intro3.1})-(\ref{intro3.2}), which allows us to derive estimates along constant-time slices and null hypersurfaces (in Subsection \ref{standardenergyestimates}) and control the acoustic geometry.\\
	\indent
	The advantage of using the acoustic geometry in this low regularity setting is that it reveals the dispersive properties of solutions to the wave equations. That is, for a solution $\phi$ of wave equation $\boxg\phi=0$, the derivatives which are tangent to the characteristic null cones $\coneu$ have better decay than the transversal derivatives parallel to the $\uLunit$ direction. We have to control some geometric quantities for several reasons: 
	\begin{itemize}
		\item The acoustic geometry that we set up must be well-defined. In particular, we have to rule out short-time shock formation due to the intersection of distinct null cones.
		\item To bound a suitably constructed weighted energy in order to derive decay estimates, a multiplier vectorfield method needs to be introduced. The multipliers we use are related to $\Lunit$ and $\Sigma_{t}$-tangent sphere normal vector $\spherenormal$. Since $\Lunit$ and $\spherenormal$ depend on the wave variables $(\lnenth,\ent,v)$, the acoustical eikonal function $u$, and their first derivatives, to control the weighted energy, one needs to control relative derivatives of the above quantities.  
	\end{itemize}

\subsubsection{Control of the Conformal Energy}\label{introconformal}
A crucial part in the reduction of Strichartz estimates is to derive the decay estimates. As we discussed in Section \ref{reductionofstrichartz}, we use the conformal energy method that was introduced by Wang in \cite{AGeoMetricApproach}. We need to consider both of equation $\boxg\varphi=0$ and the conformal wave equation $\square_{\rescaledgfour}(e^{-\conformalfactor}\varphi)=\cdots$ (see Definition \ref{conformalchangemetric} for the definition of $\conformalfactor$ and $\rescaledgfour$) to control various terms via the energy method. We are interested in such equations because we have reduced the Strichartz estimates for solution $\varphi$ of geometric equation $\boxg\varphi=0$ (see Subsection \ref{reductionofstrichartz} for the reduction).

Since the metric $\rescaledgfour$ is only smoother along null hypersurfaces, we have to first use the original wave equation and choose $X=f\spherenormal$ (see Section \ref{BoundnessTheorem} for the definition of $f$ and Section \ref{geometricquantities} for sphere normal vector $\spherenormal$) as the multiplier. By using the divergence theorem for a modified current on an appropriate region, we get a Morawetz-type energy estimate where we obtain a uniform bound for the standard energy of $\varphi$ along a union of a portion of the constant-time hypersurfaces and null cones. 

Then we consider the conformal wave equation. We use the multiplier approach with $\rgeo^p\Lunit$ type vectorfields in the region $\{\tau_1\leq u\leq\tau_2\}\bigcap\{\rgeo\geq R\}$ where $1\leq\tau_1<\tau_2<\Trescale$ to control the conformal energy in the exterior region and to provide energy decay for each null slice. Finally, we control the conformal energy in the interior region with the help of the argument in \cite{Anewphysical-spaceapproachtodecayforthewaveequationwith} by obtaining energy decay in each spatial-null slice.

The very definition of the conformal energy, as well as its analysis, requires delicate and precise control of the acoustic geometry. We state the boundness theorem of conformal energy in Theorem \ref{BoundnessTheorem}. By using our estimates on the Ricci coefficients, one could follow the steps listed in \cite[Section 11]{3DCompressibleEuler} to prove Theorem \ref{BoundnessTheorem}. One could go through the details of the argument in \cite[Section 7]{AGeoMetricApproach}. Also, readers could look into \cite[Section 3]{ACommutingVectorfieldsApproachtoStrichartz} for initial ideas. We omit these details because the exact same arguments hold in our case. 

\subsubsection{Control of the Acoustic Geometry} \label{introcontrol}To control the conformal energy, we need to control the acoustic geometry. Klainerman-Rodnianski \cite{ImprovedLocalwellPosedness} and Wang \cite{AGeoMetricApproach} developed an approach of controlling the geometric quantities in the low regularity setting. In this paper, we control the acoustic geometry by following the approach in Wang \cite{AGeoMetricApproach}. 

First, we provide the PDEs verified by the geometric quantities in \cite[Section 2]{ImprovedLocalwellPosedness}. We write down the geometric transport equations and div-curl system for the connection coefficients. These equations depend on our geometric formalism and are independent of the relativistic Euler equations. Secondly, we use the estimates for certain Ricci and Riemann curvature tensor components by using the decomposition of the Ricci curvature (\ref{introRicci}) in \cite[Lemma 2.1]{ImprovedLocalwellPosedness} and the Bianchi identities. It is at this step that the structure of the relativistic Euler equations is used. Specifically, we can substitute $\boxg \gfour_{\alpha\beta}(\vvariables)$ in (\ref{introRicci}) by using relativistic Euler equations (\ref{intro1}). $\modivort,\DivGradEnt$ on the right-hand side of (\ref{intro1}) satisfies the transport-div-curl system, which allows us to derive elliptic estimates and control the acoustic geometry.

Then by combining the geometric transport equations and the aforementioned Ricci and Riemann curvature tensor components estimates, one can derive and analyze the equations for many acoustic variables. These include the important mass aspect function $\mass$ and the conformal factor $\conformalfactor$, which introduce the rough geometry. Finally, we derive mixed space-time norm estimates for all the quantities, which are needed in the conformal energy estimates. We omit the details of the proof of controlling the geometry since the argument follows the same as in \cite[Section 10]{3DCompressibleEuler}.

In the following few paragraphs, we show why the standard Morawetz energy estimates are insufficient. The deformation tensor $\deform{K}:=\Lie_K\gfour$ is present in the standard Morawetz-type energy estimates, where $K:=\frac{1}{2}\left(u^2\uLunit+(2t-u)^2\Lunit\right)$ and $\Lie$ is the Lie derivative. $\deform{K}$ can be expressed by the connection coefficients of the null frame. We need to control them with the help of the transport and Hodge type equations for geometric quantities. \\
\indent
After integrating by parts to obtain the Morawetz-type energy identity, one needs to control various derivatives of $\deform{K}$. In particular, $\angnabla\gtr\upchi$, which is the angular derivative of the expansion scalar, as well as the mass aspect function $\mass:=\uLunit\gtr\upchi+\frac{1}{2}\gtr\upchi\gtr\uchi$ (see Definition \ref{D:DEFSOFCONNECTIONCOEFFICIENTS} for the definition of these geometric quantities). In order to control $\angnabla\gtr\chi$, we rely on the Raychaudhuri equation (\ref{introRaychau}) commuted with $\angnabla$. To control the Ricci term in (\ref{introRaychau}), we use (\ref{introRicci}) contracted with $\Lunit^\alpha\Lunit^\beta$. After commuting with $\angnabla$, we have to control the error term $\Lunit(\angnabla\Chfour_{\Lunit})$. It turns out that it is difficult to control $\angnabla\gtr\chi$ and $\angnabla\Chfour_{\Lunit}$ separately. Therefore, standard Morawetz energy estimates are insufficient in our case.

Therefore, to control all terms at a consistent level of regularity, we use the approach of Wang \cite{AGeoMetricApproach}, which relies on renormalized quantities and a metric that is conformal to the acoustical metric, where the conformal factor is carefully constructed so that the null expansion scalar associated to the conformal metric $\Restrace\reschi$ is precisely $\gtr\upchi+\Chfour_{\Lunit}$. We are able to obtain the regularity theory of $\gtr\upchi+\Chfour_{\Lunit}$, while it seems impossible to treat them independently.

The conformal wave equation attempts to resolve the issues described above, but introduces the difficult conformal factor $\conformalfactor$. Thus, in order to obtain sufficient regularity for the conformal factor $\conformalfactor$, we must in fact control the modified mass aspect function $\modmass$: 
\begin{align}
\modmass:=2\anglap\conformalfactor+\mass-\gtr\upchi k_{\spherenormal\spherenormal}+\frac{1}{2}\gtr\upchi\Chfour_{\uLunit}.
\end{align}
as well as the modified torsion $\angnabla\conformalfactor+\upzeta$. These quantities satisfy favorable transport and div-curl systems, i.e., the source terms have sufficient regularity, moreover, they have good decay properties. We stress that this analysis relies on obtaining careful control over the top derivatives of the specific vorticity and entropy gradient, since the modified fluid variables $\modivort, \DivGradEnt$ enter as source terms in various geometric equations, such as the Raychaudhuri equation.
\subsection{Paper Outline}
The structure of this article will follow the non-relativistic 3D compressible Euler in \cite{3DCompressibleEuler}. The logical graph of this paper is in Section \ref{mainidea}. 
\begin{itemize}
	\item In Section \ref{section2}, we first state the notations that we are going to use throughout the paper. Then we define the fluid variables and tensor fields, including the acoustical metric. We also introduce the geometric formulation of the relativistic Euler equations.
	\item In Section \ref{sectionmainthm}, we define the Littlewood-Paley projections, which are frequently used in our analysis. We provide the frequency-projected versions of the evolution equations (i.e., the relativistic Euler equations) in Lemma \ref{lwpequations}. We state the main theorem and the bootstrap assumptions in Theorem \ref{maintheorem} and Section \ref{bootstrap}.
	\item In Section \ref{sectionstructure}, we discuss the structure of the proofs which we will follow in the rest of the paper.
	\item In Section \ref{section4}, we use the bootstrap assumptions to derive the energy, $L^2$ elliptic, and Schauder estimates for the fluid variables along constant-time hypersurfaces. Note that Schauder estimates will improve the bootstrap assumption (\ref{transportba}) in Section \ref{bootstrap} after the bootstrap assumption (\ref{waveba}) is improved by Strichartz estimates. 
	\item In Section \ref{sectionreduction}, following the approach of Tataru \cite{Strichartzestimatesforsecondorderhyperbolicoperators} and Wang \cite{AGeoMetricApproach}, we rescale the fluid solution and reduce the proof of the Strichartz estimates to the proof of a spatially localized decay estimate. The reduction is essentially the same as one used by Wang \cite{AGeoMetricApproach}, which we refer to for various details.
	\item In Section \ref{sectiongeometrysetup}, we implement nonlinear geometric optics by constructing an acoustical function $u$ and setting up its geometry, including constructing an appropriate null frame. Finally, we define the conformal energy and state the boundness theorem of the conformal energy in Theorem \ref{BoundnessTheorem}, which plays a crucial role in deriving the decay estimates that were stated in Theorem \ref{Spatiallylocalizeddecay}.
	\item In Section \ref{connectioncoefficientsandpdes}, we prove the energy estimates for the fluid variables along the acoustical null hypersurfaces in Section \ref{energyalongnullhypersurfaces}. We define additional geometric quantities, including the connection coefficients of the null frame, conformal factors, mass aspect functions, and curvature tensor components.  Then we restate the estimates from \cite{AGeoMetricApproach} that yield control over geometry along the initial data hypersurface. We state the bootstrap assumptions satisfied by the rescaled fluid variables as well as the bootstrap assumptions for geometric quantities in Subsection \ref{BAGeo}. Then we state the main estimates for the geometric quantities in Prop.\ref{mainproof}, followed by a discussion of its proof in Subsection \ref{finalproof2}.
	
\end{itemize}
\section{Acknowledgments}The author gratefully acknowledges support from NSF grant \#2054184. The author wants to express his gratitude to his advisor, Jared Speck for suggesting this problem and discussions. He also wishes to thank Leonardo Abbrescia for his useful suggestions.

\section{The Relativistic Euler Equations and Its Geometric Formulation}\label{section2}
In this section, we provide the standard first-order relativistic Euler equations and its geometric formulation. The latter will be used throughout our analysis.
\subsection{Notations}\label{Notations}

Greek ``space-time" indices take on the values $0,1,2,3$, while Latin ``spatial" indices take on the values $1,2,3$. In this article, for Greek and Latin indices, for any vectorfield or one-form $V$, we lower and raise indices with the Minkowski metric $\minkowski_{\alpha\beta}:=\text{diag}(-1,1,1,1)$ and its inverse by using the notation  $(V_\flat)_\beta:=\minkowski_{\alpha\beta}V^\alpha$ and $(V^\sharp)^\beta:=(\minkowski^{-1})^{\alpha\beta}V_\alpha$. Similar notations apply to all tensorfields. Moreover, $\antisymmetic_{\alpha\beta\gamma\delta}$ denotes the fully antisymmetric symbol normalized by $\antisymmetic_{0123}=1$. Note that $(\antisymmetic^\sharp)^{0123}=-1$. We use Einstein summation throughout the paper. 

We denote $\St:=\{(t^\prime,x^1,x^2,x^3)\in\mathbb{R}^{1+3}|t^\prime\equiv t\}$ as the standard constant-time slice.\\

We denote the spatial partial derivatives by $\p$ and the space-time partial derivatives by $\pfour$.
\begin{remark}
	Since $\minkowski_{\alpha\beta}:=\text{diag}(-1,1,1,1)$, for any vectorfield or one-form $V$, we have the identities $\pfour(V_\flat)_a=\pfour(V^\sharp)^a$ and $\pfour(V_\flat)_0=-\pfour(V^\sharp)^0$.
\end{remark}
\subsection{Definitions of the Fluid Variables and Related Quantities}\label{deffluid}
\subsubsection{The Basic Fluid Variables} \label{thebasicfluidvariables}

The fluid velocity $v^\alpha$ is a future-directed four vector and normalized by $(v_\flat)_\alpha v^\alpha=-1$. $\pressure$ denotes the pressure, $\density$ denotes the proper energy density, $n$ denotes the proper number density, $\ent$ denotes the entropy per particle, $\temp$ denotes the temperature, and 
\begin{equation}
\enth=(\density+\pressure)/n
\end{equation}
is the enthalpy per particle. Thermodynamics supplies the following laws:
\begin{align}
\enth&=\left.\frac{\p\density}{\p n}\right\vert_\ent,&
\temp&=\left.\frac{1}{n}\frac{\p\density}{\p\ent}\right\vert_n,&
\diff\enth&=\frac{\diff \pressure}{n}+\temp\diff\ent,
\end{align}
where $\left.\frac{\p}{\p n}\right\vert_\ent$ denotes partial differentiation with respect to $n$ at fixed $\ent$ and $\left.\frac{\p}{\p \ent}\right\vert_n$ denotes partial differentiation with respect to $\ent$ at fixed $n$.
\subsubsection{$v$-orthogonal Vorticity}\label{v-ortho}
\begin{definition}[The $v$-orthogonal vorticity of a one form]\label{vorthogonalvorticity}
	Given a space-time one-form $V$, we define the corresponding $v$-orthogonal (with respect to Minkowski metric) vorticity vectorfield as follows:
	\begin{equation}
	\text{vort}^\alpha(V):=-(\antisymmetic^\sharp)^{\alpha\beta\gamma\delta}(v_\flat)_\beta\p_\gamma V_\delta.
	\end{equation}
\end{definition}

\begin{definition}[The $v$-orthogonal vorticity vector field]\label{vorticity}
	We define the vorticity vector field $\vort^\alpha$ as follows:
	\begin{align}
	\vort^\alpha:=\text{vort}^\alpha(\enth v)=-(\antisymmetic^\sharp)^{\alpha\beta\gamma\delta}(v_\flat)_\beta\p_\gamma(\enth (v_\flat)_\delta).
	\end{align}
\end{definition}
\subsubsection{Auxiliary Fluid Variables}
\begin{definition}[Logarithmic enthalpy]\label{D:loglnen}
	Let $\average{H}>0$ be a fixed constant value of the background enthalpy. We define the logarithmic enthalpy $\lnenth$ as follows:
	\begin{equation}
	\lnenth:=\ln(\enth/\average{H}).
	\end{equation}
\end{definition}
\begin{definition}[Temperature over enthalpy]\label{D:Tempeatureoverenth}
	We define the quantity $\tempoverEnth$ as follows:
	\begin{equation}
	\tempoverEnth:=\frac{\temp}{\enth}.
	\end{equation}
\end{definition}
\begin{definition}[Entropy gradient one-form]\label{D:entgradient}
	We define the entropy gradient one-form $\gradEnt_\alpha$ as follows:
	\begin{equation}
	\gradEnt_\alpha:=\p_\alpha\ent.
	\end{equation}
\end{definition}
\subsubsection{Equation of State and Speed of Sound}\label{D:speedofsound}
\begin{definition}[Partial derivatives with respect to $\lnenth$ and $\ent$]
	If $Q$ is a quantity that can be expressed as a function of $(\lnenth,\ent)$, then
	\begin{align}
	Q_{;\lnenth}&=\left.Q_{;\lnenth}(\lnenth,\ent):=\frac{\p Q}{\p\lnenth}\right\vert_\ent,\\
	Q_{;\ent}&=\left.Q_{;\ent}(\lnenth,\ent):=\frac{\p Q}{\p\ent}\right\vert_\lnenth.
	\end{align}
\end{definition}
We assume an equation of state of the form $\pressure=\pressure(\density,\ent)$. The speed of sound $\speed$ is defined as follows:
\begin{equation}
\speed:=\sqrt{\left.\frac{\p \pressure}{\p\density}\right\vert_\ent}.
\end{equation}
In the rest of the article, we view the speed of sound be a function of $\lnenth$ and $\ent$: $\speed=\speed(\lnenth,\ent)$.\\
We restrict to the physically relevant regime where the speed of sound does not exceed the speed of light:
	\begin{align}
	0<\speed\leq 1.
	\end{align}
\subsection{Standard First-Order Equations}\label{RelEuler}
Considering $\ent$, $\lnenth$ and $\{v^\alpha\}_{\alpha=0,1,2,3}$ to be the fundamental unknowns, as in \cite[Section 3]{TheRelativistivEulerEquations}, the relativistic Euler equations take the form of a quasilinear hyperbolic system:
\begin{subequations}\label{RelEuler1}
	\begin{align}
v^\kappa\p_\kappa h+c^2\p_\kappa v^\kappa&=0,\\
v^\kappa\p_\kappa (v_\flat)_\alpha+\p_\alpha h+(v_\flat)_\alpha v^\kappa\p_\kappa h-q\p_\alpha s&=0,\\
\label{svothogonal}v^\kappa\p_\kappa s&=0.
\end{align}
\end{subequations}
\subsection{Modified Fluid Variables and the Geometric Wave-Transport Formulation}In this subsection we define several variables followed by the new formulation of relativistic Euler equations.\label{newformulation}
\begin{definition}\cite[Definition 2.8. Modified fluid variables]{TheRelativistivEulerEquations}\label{modifiedfluidvariables} We define the modified fluid variables as follows:
	\begin{align}
	\label{Def:C}\modivort^\alpha:=&\text{vort}^\alpha(\vort_\flat)+\speed^{-2}\antisymmetic^{\alpha\beta\gamma\delta}(v_\flat)_\beta(\p_\gamma\lnenth)(\vort_\flat)_\delta+(\temp-\temp_{;\lnenth})(\gradEnt^\sharp)^\alpha(\p_\kappa v^\kappa)\\
	\notag&+(\temp-\temp_{;\lnenth})v^\alpha\left\{(\gradEnt^\sharp)^\kappa\p_\kappa\lnenth\right\}-(\temp-\temp_{;\lnenth})(\gradEnt^\sharp)^\kappa\left\{(\minkowski^{-1})^{\alpha\lambda}\p_\lambda (v_\flat)_\kappa\right\},\\
	\label{Def:D}\DivGradEnt:=&\frac{1}{n}\left\{\p_\kappa(\gradEnt^\sharp)^\kappa\right\}+\frac{1}{n}\left\{(\gradEnt^\sharp)^\kappa\p_\kappa\lnenth\right\}-\frac{1}{n}\speed^{-2}\left\{(\gradEnt^\sharp)^\kappa\p_\kappa\lnenth\right\}.
	\end{align}
\end{definition}
\begin{definition}[Acoustical metric and its inverse]\label{D:acousticunrescaled}
	Let $\minkowski$ be the Minkowski as defined in Section \ref{Notations}, we define the acoustical metric $\gfour_{Acou_{\alpha\beta}}$ and its inverse $(\gacoustical^{-1})^{\alpha\beta}$ as follows\footnote{For convience, we write $\gacoustical^{\alpha\beta}$ instead of $(\gacoustical^{-1})^{\alpha\beta}$ in this paper.}:
\begin{subequations}
	\begin{align}
	\gfour_{Acou_{\alpha\beta}}&:=\speed^{-2}\minkowski_{\alpha\beta}+(\speed^{-2}-1)(v_\flat)_\alpha (v_\flat)_\beta,\\
	(\gacoustical^{-1})^{\alpha\beta}
	&:=\speed^2(\minkowski^{-1})^{\alpha\beta}+(\speed^2-1)v^\alpha v^\beta.
	\end{align}
\end{subequations}	
\end{definition}
\begin{definition}[Adjusted acoustical metric and its inverse]\label{D:acousticmetric}
We define the adjusted acoustical metric $\gfour_{\alpha\beta}=\gfour_{\alpha\beta}(\vvariables)$ and its inverse $(\gfour^{-1})^{\alpha\beta}=(\gfour^{-1})^{\alpha\beta}(\vvariables)$ as follows\footnotemark:
\begin{subequations}\label{gacoustic}
	\begin{align}
\label{gacoustic1}\gfour_{\alpha\beta}&=\gfour_{Acou_{\alpha\beta}}(-\gfour_{Acou}^{00})=\left\{\speed^{-2}\minkowski_{\alpha\beta}+(\speed^{-2}-1)(v_\flat)_\alpha (v_\flat)_\beta\right\}\left\{c^2-(c^2-1)(v^0)^2\right\},\\
\label{gacoustic2}(\gfour^{-1})^{\alpha\beta}&=\frac{\gfour_{Acou}^{\alpha\beta}}{-\gfour_{Acou}^{00}}=\frac{\speed^2\minkowski^{\alpha\beta}+(\speed^2-1)v^\alpha v^\beta}{c^2-(c^2-1)(v^0)^2}.
	\end{align}
\end{subequations}
 We emphasize that $(\gfour^{-1})^{00}=-1$. This helps us to simplify some of our formulas. Our bootstrap assumption will be so that $0<\speed\leq1$, so $(\gacoustical^{-1})^{00}<0$.
\end{definition}
 We lower and raise indices with the acoustic metric $\gfour$ and its inverse by using the notation $V_\beta=\gfour_{\alpha\beta}V^\alpha$ and $V^\beta=(\gfour^{-1})^{\alpha\beta}V_\alpha$. Note the difference between with raising and lowering indices with $\gfour$ versus $\minkowski$, see Section \ref{Notations}.

\footnotetext{For convience, we write $\gfour^{\alpha\beta}$ instead of $(\gfour^{-1})^{\alpha\beta}$ in this paper.}	

\begin{definition}\label{timelike}
	We define the future-directed $\gfour$-timelike vectorfield 
	\begin{align}
	\label{timelike2}\Timelike^\alpha:=-\gfour^{\alpha0}.
	\end{align}	
	We note that $\gfour(\Timelike,\Timelike)=-1$. We note that $\Timelike_\alpha=-\delta_\alpha^0$ where $\delta_\alpha^0$ is the Kronecker delta.
\end{definition}

\begin{definition}\label{D:inducegonconstanttime}
In Cartesian coordinates, the induced metric $g$ and its inverse on constant-time hypersurface $\St$ from $\gfour$ are as follows:
\begin{subequations}
	\begin{align}
g_{ab}&:=\gfour_{ab}+\Timelike_a\Timelike_b,\\
\label{gsigmatab}\left(g^{-1}\right)^{ab}&:=\gfour^{ab}+\Timelike^a\Timelike^b.
\end{align}
\end{subequations}
By (\ref{gacoustic}) and (\ref{timelike2}), one can compute that  $g_{ab}\left(g^{-1}\right)^{bc}=\delta^c_a$. We note that $g$ can be also viewed as a spacetime tensor, that is, 
\begin{subequations}
	\begin{align}
g_{\alpha\beta}&:=\gfour_{\alpha\beta}+\Timelike_\alpha\Timelike_\beta,\\
\left(g^{-1}\right)^{\alpha\beta}&:=\gfour^{\alpha\beta}+\Timelike^\alpha\Timelike^\beta.
\end{align}
\end{subequations}
Notice that by (\ref{timelike2}) and $(\gfour^{-1})^{00}=-1$, $\Sigmatproject^\delta_\alpha:=\gfour^{\beta\delta}g_{\alpha\beta}$ can be viewed as $\gfour-$orthogonal projection operator from whole spacetime onto $\St$.
\end{definition}

\begin{proposition}
	 The metric $g$ and its inverse $g^{-1}$ have the following spatial components relative to rectangular coordinates:
\begin{subequations}
	\begin{align}\label{gab}
(g^{-1})^{ab}&=\left\{\speed^2-(\speed^2-1)(v^0)^2\right\}^{-2}\left\{\speed^2\delta^{ab}\left[\speed^2-(\speed^2-1)(v^0)^2\right]+\speed^2(\speed^2-1)v^av^b\right\},\\
\label{gab2}g_{ab}&=\left\{\speed^{-2}\delta_{ab}+(\speed^{-2}-1)(v_\flat)_a (v_\flat)_b\right\}\left\{c^2-(c^2-1)(v^0)^2\right\}.
\end{align}	
\end{subequations}
Using (\ref{gacoustic2}), (\ref{timelike2}) and (\ref{gsigmatab}), (\ref{gab}) is obtained by direct computation. (\ref{gab2}) is obtained by the fact that $\Timelike_a=0$ for $a=1,2,3$.
\end{proposition}
\begin{definition}[Differential operators defined by $\gfour$]\label{LeviCivita}
	$\Dfour$ denotes the Levi-Civita connection
	of $\gfour$ and $\boxg:=\gfour^{\alpha\beta}\Dfour_\alpha\Dfour_\beta$ denotes the corresponding covariant wave operator. In Cartesian coordinates, for scalar function $\varphi$, $\boxg\varphi=\frac{1}{\sqrt{\abs{\gfour}}}\p_\alpha\left(\sqrt{\abs{\gfour}}\gfour^{\alpha\beta}\p_\beta\varphi\right)$, where $\abs{\gfour}$ is the determinant of $\gfour$.
\end{definition}
\begin{definition}[Arrays of variables] For convevience in presenting the formulations and analysis, we define the following arrays of solution variables:\label{D:cartesiancomponentsofvariables}
	\begin{align}
	\vec{v}:&=(v^0,v^1,v^2,v^3),& \vvort:&=(\vort^0,\vort^1,\vort^2,\vort^3),& \vgradientEnt:&=((\gradEnt^\sharp)^0,(\gradEnt^\sharp)^1,(\gradEnt^\sharp)^2,(\gradEnt^\sharp)^3),& \vectvort:=(\modivort^0,\modivort^1,\modivort^2,\modivort^3).
	\end{align}
 We also define the array $\vvariables$ of wave variables, as follows:
	\begin{align}
	\vvariables:=(v^0,v^1,v^2,v^3,h,s).
	\end{align}
\end{definition}
\begin{definition}[Material derivative]
We define the material derivative\footnote{We note that in \cite{3DCompressibleEuler}, $\materialderivative=\p_t+v^i\p_i$.} $\materialderivative=\materialderivative(\vvariables)$ as follows:
\begin{align}
\materialderivative:=\frac{v^\alpha}{v^0}\p_\alpha.
\end{align}
\end{definition}
\subsubsection{The geometric wave-transport formulation of the relativistic Euler equations}\label{geometriceq}
We use the following schematic notations throughout the paper where $A,B,C$ are arrays of variables:
\begin{itemize}
	\item $\linsmoothfunction[A](B)$ denotes any scalar-valued function that is linear in the components of $B$ with coefficients that are a function of the components of $A$.
	\item $\quadsmoothfunction[A](B,C)$ denotes any scalar-valued function that is quadratic in the components of $B$ and $C$ with coefficients that are a function of the components of $A$.
\end{itemize}
\begin{proposition}\label{WaveEqunderoringinalacoustical}
	\cite[(3.1)-(3.12b). The geometric wave-transport formulation of the relativistic Euler equations]{TheRelativistivEulerEquations} If $\variables\in\{v^0,v^1,v^2,v^3,h,s\}$ solves the relativistic Euler equations (\ref{RelEuler1}), then $\variables, \vort, \gradEnt, \modivort, \DivGradEnt$ also satisfy the following:\\
\textbf{\underline{Wave equations}}	
\begin{align}\label{boxacoustical}
\square_{\gacoustical}\variables=\linsmoothfunction(\vvariables)[\vectvort,\DivGradEnt]+\quadsmoothfunction(\vvariables)[\pfour\vvariables,\pfour\vvariables].
\end{align}
\textbf{Transport equations}
\begin{subequations}
	\begin{align}
\label{Bvort}&\materialderivative\vort^\alpha=\linsmoothfunction(\vvariables,\vvort,\vgradientEnt)[\pfour\vvariables],\\ &\label{BS}\materialderivative(\gradEnt^\sharp)^\alpha=\linsmoothfunction(\vvariables,\vgradientEnt)[\pfour\vvariables].
\end{align}
\end{subequations}
\textbf{Transport-Div-Curl system}

\begin{subequations}\label{DivCurlSystem}
	\begin{align}
	\label{Bmodivort}&\materialderivative \modivort^\alpha=\inhom{\modivort^\alpha}:=\quadsmoothfunction(\vvariables)[\pfour\vvariables,(\pfour\vort,\pfour\vgradientEnt,\pfour\vvariables)]+\quadsmoothfunction(\vgradientEnt)[\pfour\vvariables,\pfour\vvariables]+\linsmoothfunction(\vvariables,\vvort,\vgradientEnt)[\pfour\vvariables,\pfour\vgradientEnt],\\
	\label{BDiventGrad}&\materialderivative\DivGradEnt=\inhom{\DivGradEnt}:=\linsmoothfunction(\vvariables,\vgradientEnt)[\pfour\vort]+\quadsmoothfunction(\vvariables)[\pfour\vvariables,(\pfour\vgradientEnt,\pfour\vvariables)]+\quadsmoothfunction(\vgradientEnt)[\pfour\vvariables,\pfour\vvariables)]+\linsmoothfunction(\vvariables,\vvort,\vgradientEnt)[\pfour\vvariables],\\
    &\label{vortS}\textnormal{vort}^\alpha(S)=0,\\
	&\label{Divvort}\p_\alpha\vort^\alpha=\linsmoothfunction(\vvort)[\pfour\vvariables].
\end{align}
\end{subequations}
\end{proposition}

\begin{remark}
	The div-curl system (\ref{DivCurlSystem}) in the geometric formulation of the relativistic Euler equations is a space-time div-curl system. This feature causes difficulties as we want to derive estimates for vorticity and entropy gradient along the constant-time hypersurface $\St$. To solve this issue, we rewrite the div-curl system into a dynamic spatial system along constant-time slices and apply theory in Littlewood-Paley decomposition as well as pseudodifferential operators in Section \ref{section4}. These difficulties are not present in the non-relativistic 3D compressible Euler equations because the analogue of (\ref{vortS})-(\ref{Divvort}) is already a spatial div-curl system.
\end{remark}

We now provide some useful identities, which we are going to use throughout the rest of the paper.
\begin{lemma}[Identities involving vorticity and entropy gradient]\label{indentities}
	We list some useful identities in \cite[Section 4]{TheRelativistivEulerEquations} as follows:
\begin{subequations}
	 \begin{align}
\label{identities1} \vort^\kappa (v_\flat)_\kappa&=0,\\
\label{identities2} v^\kappa\p_\alpha(\vort_\flat)_\kappa&=-\vort^\kappa\p_\alpha (v_\flat)_\kappa,\\
\label{identities2.5} v^\kappa\p_\alpha\gradEnt_\kappa&=-(\gradEnt^\sharp)^\kappa\p_\alpha (v_\flat)_\kappa,\\
\label{identities3} \p_\gamma(\vort_\flat)_\delta-\p_\delta(\vort_\flat)_\gamma&=\antisymmetic_{\gamma\delta\kappa\lambda}v^\kappa \textnormal{vort}^\lambda(\vort)-(v^\kappa\p_\kappa(\vort_\flat)_\delta)(v_\flat)_\gamma+v^\kappa(\p_\delta(\vort_\flat)_\kappa)(v_\flat)_\gamma\\
\notag&+(v^\kappa\p_\kappa(\vort_\flat)_\gamma)(v_\flat)_\delta-v^\kappa(\p_\delta(\vort_\flat)_\kappa)(v_\flat)_\delta,\\
\label{identities4} \p_\gamma\gradEnt_\delta-\p_\delta\gradEnt_\gamma&=\antisymmetic_{\gamma\delta\kappa\lambda}v^\kappa \textnormal{vort}^\lambda(\gradEnt)-(v^\kappa\p_\kappa\gradEnt_\delta)(v_\flat)_\gamma\\
\notag&+v^\kappa(\p_\delta\gradEnt_\kappa)(v_\flat)_\gamma+(v^\kappa\p_\kappa\gradEnt_\gamma)(v_\flat)_\delta-v^\kappa(\p_\delta\gradEnt_\kappa)(v_\flat)_\delta.
 \end{align}
\end{subequations}
\begin{proof}[Discussion of the proof]
	(\ref{identities1}) follows from Definition \ref{vorticity}. (\ref{identities2}) follows from taking $\p_\alpha$ derivative of (\ref{identities1}). (\ref{identities2.5}) follows from taking $\p_\alpha$ derivative of (\ref{svothogonal}). 
	
	To prove (\ref{identities3}) and (\ref{identities4}), we use Definition \ref{vorticity} to express $\textnormal{vort}^\lambda(V)$ for $V=\vort_\flat$ and $V=\gradEnt$ respectively. Then, using the fact that 
	\begin{align}
	-\antisymmetic_{\alpha\beta\gamma\delta}\antisymmetic^{\delta\theta\kappa\lambda}=\delta_\alpha^\theta\delta_\gamma^\kappa\delta_\beta^\lambda-\delta_\alpha^\theta\delta_\gamma^\lambda\delta_\beta^\kappa+\delta_\alpha^\lambda\delta_\gamma^\theta\delta_\beta^\kappa-\delta_\alpha^\lambda\delta_\gamma^\kappa\delta_\beta^\theta+\delta_\alpha^\kappa\delta_\gamma^\lambda\delta_\beta^\theta-\delta_\alpha^\kappa\delta_\gamma^\theta\delta_\beta^\lambda,
	\end{align}
	we rearrange the terms and obtain (\ref{identities3}) and (\ref{identities4}). 
	
	We refer readers to \cite[Lemma 4.1]{TheRelativistivEulerEquations} for detailed proofs.
\end{proof}
\end{lemma}

In the following proposition, we provide the geometric wave equation with respect to the rescaled acoustical metric $\gfour$, which is more convenient for us to derive energy estimates and construct geometry. We will use this equation in the rest of the paper. 
\begin{proposition}[Wave equations after rescaling the acoustical metric]
	Let $\gfour$ be as defined in Definition \ref{D:acousticmetric}. If $\variables\in\{v^0,v^1,v^2,v^3,h,s\}$ solves the relativistic Euler equations (\ref{RelEuler1}), we have the following equation holds:
\label{Waveequationsafterrescalingtheacousticalmetric}	
\begin{align}
\label{waveEQ}\boxg\variables=\inhom{\variables}:=\linsmoothfunction(\vvariables)[\vectvort,\DivGradEnt]+\quadsmoothfunction(\vvariables)[\pfour\vvariables,\pfour\vvariables].
\end{align}
\end{proposition}
\begin{proof}[Proof of Prop \ref{Waveequationsafterrescalingtheacousticalmetric} using Prop \ref*{WaveEqunderoringinalacoustical} given]
\begin{align}\label{proofofboxg}
\boxg\variables&=\frac{1}{\sqrt{\abs{\gfour}}}\p_\alpha\left(\sqrt{\abs{\gfour}}\gfour^{\alpha\beta}\p_\beta\variables\right)\\
\notag&=\frac{1}{\sqrt{\abs{\gacoustical}}}\left(-\gacoustical^{00}\right)^{-2}\p_\alpha\left(\sqrt{\abs{\gacoustical}}\left(-\gacoustical^{00}\right)\gacoustical^{\alpha\beta}\p_\beta\variables\right)\\
\notag&=\frac{1}{\sqrt{\abs{\gacoustical}}}\left(-\gacoustical^{00}\right)^{-1}\p_\alpha\left(\sqrt{\abs{\gacoustical}}\gacoustical^{\alpha\beta}\p_\beta\variables\right)+\left(-\gacoustical^{00}\right)^{-2}\gacoustical^{\alpha\beta}\p_\alpha\left(\gacoustical^{00}\right)\p_{\beta}\variables\\
\notag&=\left(-\gacoustical^{00}\right)^{-1}\square_{\gacoustical}\variables+\quadsmoothfunction(\vvariables)[\pfour\vvariables,\pfour\vvariables].
\end{align}
Note that $\gacoustical^{\alpha\beta}$ is smooth function of $\vvariables$ and $\gacoustical^{00}\neq0$. Therefore by combining (\ref{boxacoustical}) and (\ref{proofofboxg}), we obtain the desired equation.
\end{proof}
\section{Norms, Littlewood-Paley Projections, Statement of Main Results and Bootstrap Assumptions}\label{sectionmainthm}
In this section, we define the norms, define the standard Littlewood-Paley projections that we use in the analysis, state our main results of the paper and bootstrap assumptions.
\subsection{Norms}

In this article, for functions $f,g$ on a normed space $(X,\norm{\cdot}_X)$, we use the notation $\norm{f,g}_X:=\norm{f}_X+\norm{g}_X$. Similarly, for an array of functions $\vec{U}=(U^1,U^2,\dots,U^k)$, we have $\norm{\vec{U}}_X:=\sum\limits_{a=1}^k\norm{U^a}_X$. In particular, we use $\abs{\vec{U}}:=\sqrt{\sum\limits_{i=1}^{k}(U^i)^2}$. For functions $f$ and arrays $\vec{g}$, we also use $\norm{f,g}_X:=\norm{f}_X+\norm{\vec{g}}_X$.

Since the volume form on the constant-time hypersurface $\St$ induced by Minkowski metric $\minkowski$ is $\diff x^1\diff x^2\diff x^3$, and by identifying $(t,x^1,x^2,x^3)\in\St$ with $(x^1,x^2,x^3)\in\mathbb{R}$, we define the standard Sobolev norm on $\St$ for $s\in\mathbb{R}$: $\sobolevnorm{F}{s}{\St}:=\norm{\langle\xi\rangle^s\hat{f}(\xi)}_{L_x^2(\St)}$, where $\langle\xi\rangle:=(1+\abs{\xi}^2)^{1/2}$. 

We denote the standard H\"older semi-norm $\dot{C}_x^{0,\beta}$ and H\"older norm $C_x^{0,\beta}$, where $0<\beta<1$, of a function $F$ with respect to flat metric on constant-time hypersurface $\St$ by 
\begin{align}
\semiholdernorm{F}{x}{0}{\beta}{\St}:=&\sup\limits_{x\neq y\in\St}\frac{\abs{F(x)-F(y)}}{\abs{x-y}^\beta},\\
\holdernorm{F}{x}{0}{\beta}{\St}:=&\sup\limits_{x\in\St}\abs{F(x)}+\sup\limits_{x\neq y\in\St}\frac{\abs{F(x)-F(y)}}{\abs{x-y}^\beta}.
\end{align}

We also use the following mixed norms for function $F: \mathbb{R}^3\rightarrow\mathbb{R}$, where $1\leq q_1<\infty$, $1\leq q_2\leq\infty$, and $I$ is an interval of time:
\begin{align}
\twonorms{F}{t}{q_1}{x}{q_2}{I\times\St}&:=\left\{\int_I\onenorm{F}{x}{q_2}{\Sigma_{\tau}}^{q_1}\diff\tau\right\}^{1/q_1},&\twonorms{F}{t}{\infty}{x}{q_2}{I\times\St}&:=\text{ess}\sup\limits_{\tau\in I}\onenorm{F}{x}{q_2}{\Sigma_{\tau}},\\
\holdertwonorms{F}{t}{q_1}{x}{0}{\beta}{I\times\St}&:=\left\{\int_I\holdernorm{F}{x}{0}{q_2}{\Sigma_{\tau}}^{q_1}\diff\tau\right\}^{1/q_1},&\holdertwonorms{F}{t}{\infty}{x}{0}{q_2}{I\times\St}&:=\text{ess}\sup\limits_{\tau\in I}\holdernorm{F}{x}{0}{q_2}{\Sigma_{\tau}}.
\end{align}

If $\{F_\lambda\}_{\lambda\in 2^\mathbb{N}}$ is a dyadic-indexed sequence of functions on $\St$, we define
\begin{align}
\norm{F_{\upnu}}_{l^2_\upnu L_x^2(\St)}:=\left(\sum\limits_{\upnu\geq 1}\onenorm{F_{\upnu}}{x}{2}{\St}^2\right)^{1/2}.
\end{align}
\subsection{Littlewood-Paley Projections}\label{definitionlittlewood} We fix a smooth function $\psi=\psi(\abs{\xi}): \mathbb{R}^3\rightarrow [0,1]$ supported on the frequency space annulus $\{\xi\in\mathbb{R}^3|1/2\leq\abs{\xi}\leq2\}$ such that for $\xi\neq0$, we have $\sum\limits_{k\in\mathbb{Z}}\psi(2^k\xi)=1$. For dyadic frequencies $\upnu=2^k$ with $k\in\mathbb{Z}$, we define the standard Littlewood-Paley projection $\littlewood$, which acts on scalar functions $F:\mathbb{R}\rightarrow\mathbb{C}$, as follows:
\begin{align}
\littlewood F(x):=\int_{\mathbb{R}^3}e^{2\pi ix\cdot\xi}\psi(\upnu^{-1}\xi)\hat{F}(\xi)\diff\xi,
\end{align}
where $\hat{F}(\xi):=\int_{\mathbb{R}^3}e^{-2\pi ix\cdot\xi}F(x)\diff x$ is the Fourier transform of $F$. If $I\subset 2^\mathbb{Z}$ is an interval of dyadic frequencies, then $P_IF:=\sum\limits_{\upnu\in I}\littlewood F$, and $P_{\leq\upnu}F:=P_{[-\infty,\upnu]}F$.
For functions $f,g$, we use the schematic notation that $\littlewood(f,g)$ as the linear combination of $\littlewood f$ and $\littlewood g$, namely, $\littlewood f+\littlewood g$.
\begin{proposition}
	For a function $F$, standard results in Littlewood-Paley theory give the following:
	\begin{align}
	\label{Littlewoodsobolev}\sobolevnorm{F}{s}{\St}&\approx\onenorm{F}{x}{2}{\St}+\left(\sum\limits_{\upnu>1}\upnu^{2s}\onenorm{\littlewood F}{x}{2}{\St}\right)^{1/2},\\
	\label{Littlewoodholder}\holdernorm{F}{x}{0}{s}{\St}&\approx\norm{ F}_{L_x^\infty(\St)}+\sup\limits_{\upnu\geq2}\upnu^s\norm{\littlewood F}_{L_x^\infty(\St)},
	\end{align}
	where $H^s$ is the standard Sobolev norm and $C^{0,s}$ is the standard H\"older norm. One can refer \cite[Section 1]{AGeometricApproachtotheLittlewoodPaleyTheory} and \cite[A.1]{pseudodifferentialoperatorsandnonlinearpde} for above results.
\end{proposition}

The following two Lemmas consist of a commuted version of the equations. Lemma \ref{lemmacommute} commutes $\boxg$ and $\materialderivative$ with $\pfour$ and is needed for below-top order estimates. Lemma \ref{lwpequations} commutes $\boxg$ and $\materialderivative$ with $\littlewood\pfour$ and is needed for the top order estimates.
\begin{lemma}[Commuted equations satisfied by one derivative of the solution variables]\label{lemmacommute}
	
	We consider the solutions to the equations of Proposition \ref{WaveEqunderoringinalacoustical}, that is, if  $\variables\in\{v^0,v^1,v^2,v^3,\lnenth,\ent\}$ solves the relativistic Euler equations (\ref{RelEuler1}), the following equations hold:
	\begin{align}
\label{boxg2}&\boxg\pfour\variables=\linsmoothfunction(\vvariables)[\pfour\vectvort,\pfour\DivGradEnt]+\quadsmoothfunction(\vvariables)[\pfour^2\vvariables,\pfour\vvariables]+\linsmoothfunction(\vvariables)[(\pfour\vvariables)^3],\\
\label{bc2}&\materialderivative\pfour\modivort^\alpha=\quadsmoothfunction(\vvariables)[\pfour\vvariables,(\pfour^2\vort,\pfour^2\vgradientEnt,\pfour^2\vvariables)]+\quadsmoothfunction(\vvariables)[(\pfour^2\vvariables,\pfour\vvort,\pfour\vvariables,\pfour\vgradientEnt),(\pfour\vvort,\pfour\vgradientEnt,\pfour\vvariables)]+\linsmoothfunction(\vvariables)[(\pfour\vvariables)^2\cdot(\pfour\vvariables,\pfour\vgradientEnt,\pfour\vvort)],\\
\label{bd2}&\materialderivative\pfour\DivGradEnt=\linsmoothfunction(\vvariables,\vvort,\vgradientEnt)[\pfour^2\vort,\pfour^2\vvariables,\pfour^2\vgradientEnt]+\quadsmoothfunction(\vvariables)[(\pfour^2\vvariables,\pfour\vvort,\pfour\vvariables,\pfour\vgradientEnt),(\pfour\vgradientEnt,\pfour\vvariables)]+\linsmoothfunction(\vvariables)[(\pfour\vvariables)^2\cdot(\pfour\vvariables,\pfour\vgradientEnt)].
\end{align}
\end{lemma}
\begin{proof}[Sketch of the proof of Lemma \ref{lemmacommute}]
	By commuting (\ref{waveEQ}), (\ref{Bmodivort}) and (\ref{BDiventGrad}) with $\pfour$, using relations (by Definition \ref{modifiedfluidvariables}) $\modivort=\linsmoothfunction(\vvariables,\vvort,\vgradientEnt)[\pfour\vvort,\pfour\vvariables]$ and $\DivGradEnt=\linsmoothfunction[\pfour\vgradientEnt]+\linsmoothfunction(\vgradientEnt)[\pfour\vvariables]$, (\ref{boxg2})-(\ref{bd2}) are derived by straightforward computations.
\end{proof}
The following Lemma provides the commuted equations with the Littlewood-Paley projections.
\begin{lemma}\cite[Lemma 5.2, Lemma 5.4. Equations satisifed by the frequency-projected solution variables]{3DCompressibleEuler}\label{lwpequations}
	For solutions to the equations of Proposition \ref{WaveEqunderoringinalacoustical}, the following equations hold:
	\begin{align}
	\label{lwpboxg}\boxg\littlewood\pfour\variables&=\remainder{\pfour\vvariables},\\
	\label{lwpbc}\materialderivative\littlewood\pfour\modivort^\alpha&=\remainder{\pfour\modivort^\alpha},\\
	\label{lwpbd}\materialderivative\littlewood\pfour\DivGradEnt&=\remainder{\pfour\DivGradEnt},
	\end{align}
	where
	\begin{align}
	\label{remainder1}\remainder{\pfour\vvariables}=&\littlewood\pfour\inhom{\variables}-\sum\limits_{(\alpha,\beta)\neq(0,0)}\littlewood\left[\pfour\gfour^{\alpha\beta}\p_\alpha\p_\beta\variables\right]-\Chfour^{\alpha}\littlewood\p_\alpha\pfour\variables\\
	\notag&+\sum\limits_{(\alpha,\beta)\neq(0,0)}\left[\gfour^{\alpha\beta}-P_{\leq\upnu}\gfour^{\alpha\beta}\right]\littlewood\p_\alpha\p_{\beta}\pfour\variables\\
	\notag&+\sum\limits_{(\alpha,\beta)\neq(0,0)}\left\{P_{\leq\upnu}\gfour^{\alpha\beta}\littlewood\p_\alpha\p_{\beta}\pfour\variables-\littlewood[\gfour^{\alpha\beta}\p_\alpha\p_\beta\pfour\variables]\right\},\\
	\label{remainder2}\remainder{\pfour\modivort^\alpha}=&\littlewood\pfour\inhom{\modivort^\alpha}-\littlewood\left[\pfour \left(\frac{v^a}{v^0}\right)\p_a\modivort^\alpha\right]+\left[\frac{v^a}{v^0}-P_{\leq\upnu}\left(\frac{v^a}{v^0}\right)\right]\littlewood\p_a\pfour\modivort^\alpha\\
	\notag&+P_{\leq\upnu}\left(\frac{v^a}{v^0}\right)\littlewood\p_a\pfour\modivort^\alpha-\littlewood\left[\frac{v^a}{v^0}\p_a\pfour\modivort^\alpha\right],\\
	\label{remainder3}\remainder{\pfour\DivGradEnt}=&\littlewood\pfour\inhom{\DivGradEnt}-\littlewood\left[\pfour \left(\frac{v^a}{v^0}\right)\p_a\DivGradEnt\right]+\left[\frac{v^a}{v^0}-P_{\leq\upnu}\left(\frac{v^a}{v^0}\right)\right]\littlewood\p_a\pfour\DivGradEnt\\
	\notag&+P_{\leq\upnu}\left(\frac{v^a}{v^0}\right)\littlewood\p_a\pfour\DivGradEnt-\littlewood\left[\frac{v^a}{v^0}\p_a\pfour\DivGradEnt\right].
	\end{align}
	
	Moreover the following estimates hold for the remainders where $l^2_\upnu$-seminorm is taken over dyadic frequencies:
	\begin{align}
	\label{remainderestimates}\norm{\upnu^{N-2}\left(\remainder{\pfour\vvariables},\remainder{\pfour\modivort^\alpha},\remainder{\pfour\DivGradEnt}\right)}_{l^2_\upnu L_x^2(\St)}\lesssim&\sobolevnorm{\p(\vectvort,\DivGradEnt)}{N-2}{\St}\\
	\notag&+\left(\norm{\pfour(\vvariables,\vvort,\vgradientEnt)}_{L_x^\infty(\St)}+1\right)\left(\sobolevnorm{\pfour(\vvariables,\vvort,\vgradientEnt)}{N-1}{\St}+1\right).
	\end{align}
\end{lemma}
\begin{proof}[Discussion of the proof of Lemma \ref{lwpequations}]
	We omit the proof of equation (\ref{remainder1}),(\ref{remainder2}) and (\ref{remainder3}) since it follows from straightforward computations. We use bootstrap assumptions, product and commutator estimates for Littlewood-Paley calculus to prove estimates (\ref{remainderestimates}). We refer readers to \cite[Lemma 5.4]{3DCompressibleEuler} for the detailed proofs where the structure of the equations are the same as in this article. We note that we have $\materialderivative^a=\frac{v^a}{v^0}$ compare to $\materialderivative^a=v^a$ in \cite{3DCompressibleEuler}, which doesn't change the proof or result in the estimates (\ref{remainderestimates}).
\end{proof}
\subsection{Statement of Main Theorem}
Using the notations introduced in the previous subsections, we precisely provide our assumptions on the data and the statement of the main theorem
\begin{theorem}[Main theorem]\label{maintheorem}
	Consider a solution to the relativistic Euler equations whose initial data satisfies following assumptions for some real number $2<N<5/2$, $0<\alpha<1$, $c_1>0$ and $D$:\\
	1. $\sobolevnorm{h, v, \vort}{N}{\Sigma_0}+\sobolevnorm{s}{N+1}{\Sigma_0}\leq D$,\\
	2. $\holdernorm{\modivort,\DivGradEnt}{}{0}{\alpha}{\Sigma_0}\leq D$,\\
	3. The data functions are contained in the interior of $\Domain$ (See Definition \ref{Hyperbolicity} for definition of $\Domain$) and the enthalpy $\enth$ is strictly positive, i.e. $\enth\geq c_1>0$.\\
	Then the solution's time of classical existence $T>0$ can be bounded from below in terms of $D$ and $\mathcal{R}$. Moreover, the Sobolev and some H$\ddot{o}$lder regularity\footnote{We note that not entire H\"older regularity is propogated in the way that only $\holdernorm{\vvariables,\modivort,\DivGradEnt}{}{0}{\delta_1}{\Sigma_t}$ (see Section \ref{ChoiceofParameters} for the definition of $\delta_1$) is bounded by data.} of the data are propogated by the solution on the slab of classical existence.
\end{theorem}
\subsection{Choice of Parameters}\label{ChoiceofParameters}In this subsection, we introduce several parameters that each of them either measures the regularity or plays a role in our analysis. We denote the assumed Sobolev regularity of the ``wave-part" of the data and the H\"older regularity of the ``transport-part" of the data by, respectively, $2<N<5/2$ and $0<\alpha<1$. For the purpose of analysis, we choose positive numbers $q,\varepsilon_0, \delta_0,\delta$ and $\delta_1$ that satisfy the following conditions:
\begin{subequations}
	\begin{align}
&2<q<\infty,\\
&0<\varepsilon_0:=\frac{N-2}{10}<\frac{1}{10},\\
&\delta_0:=\min\{\varepsilon_0^2,\frac{\alpha}{10}\},\\
&0<\delta:=\frac{1}{2}-\frac{1}{q}<\varepsilon_0,\\
&\delta_1:=\min\{N-2-4\varepsilon_0-\delta(1-8\varepsilon_0),\alpha\}>8\delta_0>0.
\end{align}
\end{subequations}

More precisely, we consider $N,\alpha,\varepsilon_0$ and $\delta_0$ to be fixed throughout the paper, while $q,\delta$ and $\delta_1$ will be treated as parameters.

\subsection{Assumptions on the Initial Data}\label{Data}
In this subsection, we provide the bootstrap assumptions that will be used in the proof of Theorem \ref{maintheorem}.
\begin{definition}[Regime of hyperbolicity]\label{Hyperbolicity}
	We define $\Domain$ as follows
	\begin{align}\label{Hyperbolicity1}
	\Domain:=\left\{(\lnenth,\ent,\vec{v},\vvort,\vgradientEnt)\in\mathbb{R}^{14}|\ 0<\speed\leq 1\right\}.
	\end{align}
\end{definition}
With $N$ and $\alpha$ as in Section \ref{ChoiceofParameters}, we assume that 
\begin{align}
\textbf{``Wave-part"}&&\label{wavepart}\sobolevnorm{\lnenth,\vec{v}}{N}{\Sigma_{0}}&<\infty,\\
\textbf{``Transport-part"}&&\label{transportpart}\sobolevnorm{\ent}{N+1}{\Sigma_{0}}+\sobolevnorm{\vvort}{N}{\Sigma_{0}}+\holdernorm{\vectvort,\DivGradEnt}{}{0}{\alpha}{\Sigma_{0}}&<\infty.
\end{align}
Assumptions (\ref{wavepart}) and (\ref{transportpart}) correspond to regularity assumptions on the ``wave-part" and ``transport-part" of the data respectively. 

Let $\Int U$ denote the interior of the set $U$. We assume that there is a compact subset $\breve{\Domain}$ such that 
\begin{align}
\left(\vvariables,\vvort,\vgradientEnt\right)(\Sigma_{0})\subset\Int\breve{\Domain}\subset\breve{\Domain}\subset\Int\Domain,
\end{align}
where $\Domain$ is defined in (\ref{Hyperbolicity1}).
\subsection{Bootstrap Assumptions}\label{bootstrap} Throughout the article, $0<\Tstar\ll1$ denotes a bootstrap time that depends only on initial data. We assume that $\vvariables$ is a smooth\footnote{By smooth we mean as smooth as neccessary for the analysis arguments to go through. Meanwhile, all of our quantitative estimates depend only on the Sobolev and H\"older norms.} 
solution to the equation in Section \ref{RelEuler} and the following estimates hold:
\begin{subequations}
	\begin{align}
	\label{databa}\left(\vvariables,\vvort,\vgradientEnt\right)([0,\Tstar]\times\mathbb{R}^3)&\subset\Domain,\\
	\twonorms{\pfour\vvariables}{t}{2}{x}{\infty}{[0,\Tstar]\times\mathbb{R}^3}^2+\sum\limits_{\upnu\geq2}\upnu^{2\delta_0}\twonorms{\label{waveba}\littlewood\pfour\vvariables}{t}{2}{x}{\infty}{[0,\Tstar]\times\mathbb{R}^3}^2&\leq1,\\
	\label{transportba}\twonorms{\p\vvort,\p\vgradientEnt}{t}{2}{x}{\infty}{[0,\Tstar]\times\mathbb{R}^3}^2+\sum\limits_{\upnu\geq2}\upnu^{2\delta_0}\twonorms{\littlewood\p\vvort,\littlewood\p\vgradientEnt}{t}{2}{x}{\infty}{[0,\Tstar]\times\mathbb{R}^3}^2&\leq1.
	\end{align}
\end{subequations}

In Theorem \ref{Schauder0}, we derive an improvement of (\ref{transportba}). In Theorem \ref{MainEstimates}, we derive an improvement of (\ref{waveba}). By fundamental theorem of calculus, equations (\ref{Bvort}) and (\ref{BS}), (\ref{databa}) is a direct result of (\ref{waveba}) and (\ref{transportba}).

\section{Structure of the Proofs in the Rest of the Article}\label{sectionstructure}

In this section, we provide the structure of the proofs in the article. 
Our proofs rely on a bootstrap argument where the bootstrap assumptions are in Section \ref{bootstrap}. See Section \ref{mainidea} for the logic of the bootstrap argument. The main goal for us is to improve the bootstrap assumptions to the following Strichartz-type estimates:
\begin{subequations}
\begin{align}
\label{improvewaveba}\twonorms{\pfour\vvariables}{t}{2}{x}{\infty}{[0,\Tstar]\times\mathbb{R}^3}^2+\sum\limits_{\upnu\geq2}\upnu^{2\delta_1}\twonorms{\littlewood\pfour\vvariables}{t}{2}{x}{\infty}{[0,\Tstar]\times\mathbb{R}^3}^2\lesssim\Tstar^{2\delta},\\
\label{improvetransportba}\twonorms{\pfour\vvort,\pfour\vgradientEnt}{t}{2}{x}{\infty}{[0,\Tstar]\times\mathbb{R}^3}^2+\sum\limits_{\upnu\geq2}\upnu^{2\delta_1}\twonorms{\littlewood\pfour\vvort,\littlewood\pfour\vgradientEnt}{t}{2}{x}{\infty}{[0,\Tstar]\times\mathbb{R}^3}^2\lesssim\Tstar^{2\delta}.
\end{align}
\end{subequations}

We prove the (\ref{improvewaveba}) through the following series of reductions, see Subsection \ref{reductionofstrichartz} for an overview of the logic: Strichartz estimates$\longleftarrow$Decay estimates$\longleftarrow$Conformal energy estimates$\longleftarrow$Controlling of the acoustic null geometry.

To prove (\ref{improvetransportba}), we prove a transport-Schauder type estimate in Section \ref{section4}, which is independent of the proof of (\ref{improvewaveba}). In Theorem \ref{Schauder0}, we obtain (\ref{improvetransportba}) by combining the transport-Schauder estimate and (\ref{improvewaveba}).

\subsection{Similarities and Differences Compared to the 3D Compressible Euler Equations}
Broadly speaking,  we use the same machinery as in \cite{3DCompressibleEuler} to reduce the proof of the Strichartz estimates to geometric quantities that have to be controlled in order to derive a conformal energy estimate. This reduction was first introduced and developed in the context of low-regularity problems for quasilinear wave equations, as we discussed in Section 1.1.
Disconzi-Luo-Mazzone-Speck \cite{3DCompressibleEuler} and Wang \cite{Roughsolutionsofthe3-DcompressibleEulerequations} have exploited the remarkable structure of the non-relativistic 3D compressible Euler equations to derive similar low-regularity well-posedness results in the presence of vorticity and entropy.
The main purpose of this paper is to derive similar results for the relativistic Euler flow by using the remarkable structure of the equations derived by Disconzi-Speck in \cite{TheRelativistivEulerEquations}. 

Two main differences in the present paper compared to the non-relativistic case are \textbf{1)} the first fundamental form of $\St$ is no longer conformally flat in the relativistic case, leading to more complicated geometry and \textbf{2)} the $L^2$ elliptic and Schauder estimates that we need to handle the vorticity and entropy are more complicated because unlike in the non-relativistic case, the Hodge systems that we study are quasilinear (instead of constant-coefficient). 
\subsection{Energy, $L^2$ Elliptic and Schauder Estimates in Section \ref{section4}}
In Section \ref{section4}, first we prove the energy estimates for wave variables $\lnenth,\ent,v$ and transport variables $\vort,\gradEnt,\modivort,\DivGradEnt$ in Proposition \ref{EnergyandEllipticEstimates}. These estimates are essential to the local well-posedness theorem Theorem \ref{maintheorem}. We also need these estimates for controlling the acoustic geometry. We control the $H^{2+\varepsilon}$ norm of wave variables under the bootstrap assumptions by using the geometric energy method in Subsection \ref{energymethodforwave} and commuted equations in Lemma \ref{lemmacommute} and Lemma \ref{lwpequations}. We refer readers to \cite[Section 6]{RoughSolutionsofEinsteinVacuumEquationsinCMCSHGauge} for the commutator estimates involving LP projections in fractional Sobolev spaces. We note that the $L^2$ elliptic estimates for transport variables $\vort,\gradEnt,\modivort,\DivGradEnt$ in Proposition \ref{EllipticestimatesinL2} is proven based on a rewritten dynamic div-curl system in Proposition \ref{NewDivCurl}.

We then prove the transport-Schauder estimates in the H\"older space $C_x^{0,\delta_1}$ for the transport variables $\vort,\gradEnt,\modivort,\DivGradEnt$ in Theorem \ref{Schauder0}, which recovers the bootstrap assumptions (\ref{transportba}) in condition of (\ref{improvewaveba}). To prove these estimates, we use the div-curl system (\ref{newdivcurl}) and the transport equations (\ref{Bvort})-(\ref{BDiventGrad}). We prove Schauder estimates by some standard results in pseudodifferential operators as well as the Littlewood-Paley decomposition with the help of the equivalence between H\"older spaces and frequency spaces (\ref{Littlewoodholder}).

\subsection{Reduction of Strichartz Estimates to Decay Estimates in Section \ref{sectionreduction}}
We state the Strichartz estimates for wave variables in Theorem \ref{MainEstimates}, which improves the bootstrap assumptions (\ref{waveba}). Our reductions of the Strichartz estimates to the bounded conformal energy consists of several steps. 
In Section \ref{sectionreduction}, we list several reductions from Strichartz estimates to a spatially localized decay estimate in Theorem \ref{Spatiallylocalizeddecay}. We first use Duhamel's principle to reduce Theorem \ref{MainEstimates} to a frequency localized version of Strichartz estimates in Theorem \ref{FrequencyStrichartz}. Then after rescaling all the quantities with respect to the frequency in Subsection \ref{SelfRescaling}, we run a $\mathcal{T}\mathcal{T}^*$ argument to reduce the Theorem \ref{FrequencyStrichartz} to a decay estimate in Theorem \ref{DecayEstimates}. Finally, by Bernstein inequalities, partition of unity, and Sobolev embedding, we obtain the spatially localized version of decay estimates in Theorem \ref{Spatiallylocalizeddecay}. The reductions are by now standard, therefore we only state the reductions without proof. We refer to Wang \cite[Section 3]{AGeoMetricApproach} for the details.
It is crucial to derive the decay estimates (\ref{Decay}). To further reduce the decay estimates to conformal energy estimates, we need a geometric setup, which is in Section \ref{sectiongeometrysetup}, as we will discuss in the next subsection.

\subsection{Geometric Setup and Conformal Energy in Section \ref{sectiongeometrysetup}}\label{structurenullenergy}
To control the conformal energy, we use Wang's approach from \cite{AGeoMetricApproach}, which relies on analysis on a conformal changed acoustic geometry. We reduce the decay estimates to the conformal energy estimates Theorem \ref{BoundnessTheorem} in Section \ref{subsectionconformalenergy} via product estimates and Berstein inequality of Littlewood-Paley theory. 

In order to define conformal energy and do analysis based on the geometric structure of the 
relativistic Euler equations, in Section \ref{opticalfunction}-\ref{geometricquantities}, we construct the geometric null frame based on a solution $u$ to the acoustical eikonal equation:
\begin{align}
\gfour^{\alpha\beta}\p_\alpha u\p_\beta u=0.
\end{align}

Then we control the acoustic geometry based on the null frame that we just constructed. Note that given the control over the acoustic geometry, we have favorable estimates for the conformal energy as we discuss in Subsection \ref{discussionofconformalenergy}. We will discuss the control of the acoustic geometry in the next two subsections.

\subsection{Energy along Acoustic Null Hypersurfaces and Control of the Acoustic Geometry in Section \ref{connectioncoefficientsandpdes}}
In Section \ref{connectioncoefficientsandpdes}, We prove the energy estimates for fluid variables along the acoustic null cones in Subsection \ref{energyalongnullhypersurfaces}. This is important since we need to control fluid variables along the null cones.

Then we control the acoustic geometry. To start with, we list the connection coefficients in Definition \ref{D:DEFSOFCONNECTIONCOEFFICIENTS}. We define conformal factor for the metric in Definition \ref{conformalchangemetric}. We provide initial conditions for geometric cones in Proposition \ref{InitialCT1} and Proposition \ref{InitialCT2}.

We set up the bootstrap assumptions for geometric quantities in Section \ref{BAGeo} and list the main estimates for the geometric quantities in Proposition \ref{mainproof}. We give a discussion of the proof of Proposition \ref{mainproof} in Subsection \ref{finalproof2}, which improves the assumptions in Section \ref{BAGeo}. Proposition \ref{mainproof} is proven via doing analysis on the transport equations and the div-curl systems of the geometric quantities. We only give a brief discussion of the proof since it follows the same in \cite[Section 10]{3DCompressibleEuler}.

\section{Energy, $L^2$ Elliptic and Schauder Estimates}\label{section4}In this section, we first derive the energy and $L^2$ elliptic estimates along constant-time hypersurfaces. Note that we obtain the same results as in \cite[Section 4-5]{3DCompressibleEuler}, where $\density$ in \cite{3DCompressibleEuler} plays the same role as $\lnenth$ in this paper. Then we derive transport-Schauder type estimates for the vorticity and entropy gradient.
\subsection{Energy and $L^2$ Elliptic Estimates}\label{Section4.1}
The following Proposition is the main result of the energy estimates.
\begin{proposition}[Energy and elliptic estimates]\label{EnergyandEllipticEstimates}
	Under the initial data and bootstrap assumptions of Section \ref{sectionmainthm}, smooth solutions to the relativistic Euler equations satisfy the following estimates for $2<N<5/2$ and $t\in[0,\Tstar]$:

\begin{align}\label{energyestimates}
	&\sum\limits^2_{k=0}\sobolevnorm{\p_t^k(\lnenth,\vec{v},\vvort)}{N-k}{\St}+\sum\limits^2_{k=0}\sobolevnorm{\p_t^k\ent}{N+1-k}{\St}+\sum\limits^1_{k=0}\sobolevnorm{\p_t^k(\vectvort,\DivGradEnt)}{N-1-k}{\St}\\
	\notag&\lesssim\sobolevnorm{(\lnenth,\vec{v},\vvort)}{N}{\Szero}+\sobolevnorm{\ent}{N+1}{\Szero}+1.
\end{align}

\end{proposition}
\begin{remark}
	We note that $1$ on the right-hand side of (\ref{energyestimates}) is due to technical reasons. Specifically, as shown in Section \ref{Discussionofprop5.1}, $1$ can be replaced by  $\int_{0}^{t}\norm{\pfour(\vvariables,\vvort,\vgradientEnt)}_{L_x^\infty(\Sigma_{\tau})}\diff\tau$. By bootstrap assumptions (\ref{waveba})-(\ref{transportba}) and H\"older's inequality, we could actually bound this term by $\Tstar^{1/2}$. 
\end{remark}
We provide several key ingredients for proving Proposition \ref{EnergyandEllipticEstimates} in the next two subsections. We provide the basic energy inequality for wave equations and transport equations in Subsection \ref{energymethodforwave}. We prove a crucial elliptic div-curl estimate in Subsection \ref{EllipticDiv-Curl}. We give the proof of Proposition \ref{EnergyandEllipticEstimates} in Subsection \ref{Discussionofprop5.1}. We refer readers to \cite[Section 4-5]{3DCompressibleEuler} and \cite[Section 2]{AGeoMetricApproach} for the energy estimates in the non-relativistic 3D compressible Euler equations case and the quasilinear wave equations case respectively.
\subsubsection{The basic energy inequality for wave equations and transport equations}\label{energymethodforwave}
We provide the basic energy inequality for the wave equations in this subsection.
\begin{definition}[Energy-momentum tensor, energy current, and deformation tensor]\label{energymomentum} We define the energy-momentum tensor $Q_{\mu\nu}[\varphi]$ associated to a scalar function $\varphi$ to be the following tensorfield:
	\begin{align}
Q_{\mu\nu}[\varphi]:=\p_\mu\varphi\p_\nu\varphi-\frac{1}{2}\gfour_{\mu\nu}(\gfour^{-1})^{\alpha\beta}\p_\alpha\varphi\p_\beta\varphi.
\end{align} 

Given $\varphi$ and any multiplier vectorfield $\mathbf{X}$, we define the corresponding energy current $\Jenarg{\mathbf{X}}{\alpha}[\varphi]$ vectorfield as follows:
\begin{align}
\Jenarg{\mathbf{X}}{\alpha}[\varphi]:=Q^{\alpha\beta}[\varphi]\mathbf{X}_{\beta}-\varphi^2\mathbf{X}^\alpha.
\end{align}

We define the deformation tensor of $\mathbf{X}$ as follows:
\begin{align}\label{deformtensor}
\deform{\mathbf{X}}_{\alpha\beta}:=\Dfour_{\alpha}\mathbf{X}_\beta+\Dfour_{\beta}\mathbf{X}_\alpha,
\end{align}
where $\Dfour$ is the Levi-Civita connection with respect to $\gfour$.

We have the following well-known divergence identity:
\begin{align}\label{DalphaJalpha}
\Dfour_\alpha\Jenarg{\mathbf{X}}{\alpha}[\varphi]=\boxg\varphi(\mathbf{X}\varphi)+\frac{1}{2}Q^{\mu\nu}[\varphi]\deform{\mathbf{X}}_{\mu\nu}-2\varphi\mathbf{X}\varphi-\frac{1}{2}\varphi^2(\gfour^{-1})^{\mu\nu}\deform{\mathbf{X}}_{\mu\nu}.
\end{align}

We define the energy $\mathbb{E}[\varphi](t)$ as follows where $\Timelike^\alpha:=-\gfour^{\alpha0}$ is the future-directed $\gfour$-timelike vectorfield defined in Definition \ref{timelike}:
\begin{align}\label{Energy1}
\mathbb{E}[\varphi](t):=\int_{\St}\Jenarg{\mathbf{\Timelike}}{\alpha}[\varphi]\Timelike_\alpha\diff\tvol=\int_{\St}\left(Q^{00}[\varphi]+\varphi^2\right)\diff\tvol,
\end{align}
where $\diff\tvol$ is the volume form on $\St$ with respect to $g$ induced by $\gfour$.
\end{definition}

\begin{lemma}[Coerciveness of $\mathbb{E}$]\label{coeciveness} Under the bootstrap assumptions of Section \ref{bootstrap}, the following estimate holds for $t\in[0,\Tstar]$:
	\begin{align}\label{Coecive}
	\mathbb{E}[\varphi](t)\approx\norm{\varphi}_{H^1(\St)}^2+\norm{\p_t\varphi}_{ L_x^2(\St)}^2.
	\end{align}
\end{lemma}
\begin{proof}[Proof of Lemma \ref{coeciveness}]
	First recall that $\Timelike_a=0$, so $g_{ab}=\gfour_{ab}$.
Notice that since $0<\speed(\lnenth,\ent)\leq1$ and $(v^0)^2\geq1$, by direct computation and the bootstrap assumption (\ref{databa}), we have
	\begin{align}
	\label{closetoEuclidean}\diff\tvol=\sqrt{\det g}\diff x^1\diff x^2\diff x^3=\left\{\speed^2-(\speed^2-1)(v^0)^2\right\}^3\left\{\speed^{-6}+\speed^{-4}(\speed^{-2}-1)\left[(v^0)^2-1\right]\right\}\approx1.
	\end{align}	
	Then we compute $Q^{00}[\varphi]$. By (\ref{gsigmatab}) and (\ref{gab}), we have
	\begin{align}
\label{Q00}	Q^{00}[\varphi]=\frac{1}{2}\left\{(\Timelike\varphi)^2+(g^{-1})^{ab}\p_a\varphi\p_b\varphi \right\}=\frac{1}{2}\left\{(\Timelike\varphi)^2+\frac{\speed^2\delta^{ab}\left[\speed^2-(\speed^2-1)(v^0)^2\right]+\speed^2(\speed^2-1)v^av^b}{\left\{\speed^2-(\speed^2-1)(v^0)^2\right\}^{2}}\p_a\varphi\p_b\varphi\right\},
	\end{align}
	where $\delta^{ab}$ is the Kronecker delta.
	Note that 
	\begin{align}
	\label{T}\Timelike=\p_t+\frac{(\speed^2-1)v^a v^0}{c^2-(c^2-1)(v^0)^2}\p_a.
	\end{align}		
	Then, since the speed of sound satisfies $0<\speed\leq1$, it follows that (\ref{Q00}) is coercive in $\abs{\p\varphi}$, since $(v^0)^2=1+\sum\limits_{i=1,2,3}(v^i)^2$:
	\begin{align}\label{remiannianofg}
	&\left\{\speed^2\delta^{ab}\left[\speed^2-(\speed^2-1)(v^0)^2\right]+\speed^2(\speed^2-1)v^av^b\right\}\p_a\varphi\p_b\varphi\\
	\notag&=\speed^4\abs{\p\varphi}^2-\speed^2(\speed^2-1)\left\{\delta^{ab}(v^0)^2-v^av^b\right\}\p_a\varphi\p_b\varphi\\
	\notag&\geq\speed^4\abs{\p\varphi}^2.
	\end{align} 
	
	By bootstrap assumptions that $\abs{v^\alpha}$ are uniformly bounded and Young's inequality, we derive that $Q^{00}[\varphi]\lesssim\abs{\pfour\varphi}^2$. Combined with (\ref{remiannianofg}), the desired estimates (\ref{Coecive}) follows.
\end{proof}
\begin{lemma}[Basic energy inequality for the wave equations]\label{basicenergy}
	Let $\varphi$ be smooth on $[0,\Tstar]\times\mathbb{R}^3$. Under the bootstrap assumptions of Section \ref{bootstrap}, the following inequality holds for $t\in[0,\Tstar]$: 
	\begin{align}
	\mathbb{E}[\varphi](t)\lesssim &\mathbb{E}[\varphi](0)
	+\int_{0}^{t}\norm{\pfour\vvariables}_{L_x^\infty(\Sigma_{\tau})}\mathbb{E}[\varphi](\tau)\diff\tau\\
	\notag&+\int_{0}^{t}\norm{\boxg\varphi}_{L_x^2(\Sigma_{\tau})}\norm{\pfour\varphi}_{L_x^2(\Sigma_{\tau})}\diff\tau.
	\end{align}
\end{lemma}
\begin{proof}[Proof of Lemma \ref{basicenergy}]
	We apply the divergence theorem on the space-time region $[0,t]\times\mathbb{R}^3$ relative to the volume form $\diff\gvol=\sqrt{\det g}\diff x^1\diff x^2\diff x^3\diff\tau=\diff\tvol\diff\tau$. Note that $\Timelike$ is the future-directed $\gfour$-unit normal to $\St$. By (\ref{deformtensor}), (\ref{DalphaJalpha}), (\ref{Energy1}) and (\ref{Coecive}), with $\mathbf{X}:=\Timelike$, we have:
	\begin{align}
	\mathbb{E}[\varphi](t)&=\mathbb{E}[\varphi](0)-\int_{0}^{t}\int_{\Sigma_{\tau}}\left(\boxg\varphi(\Timelike\varphi)+\frac{1}{2}Q^{\mu\nu}[\varphi]\deform{\Timelike}_{\mu\nu}-2\varphi\Timelike\varphi-\frac{1}{2}\varphi^2(\gfour^{-1})^{\mu\nu}\deform{\Timelike}_{\mu\nu}\right)\diff\tvol\diff\tau.
	\end{align}
	
	By bootstrap assumptions, we have $\abs{\Timelike\varphi}\lesssim\abs{\pfour\varphi}$, $\abs{Q^{\mu\nu}[\varphi]}\lesssim\abs{\pfour\varphi}^2$ and $\abs{\deform{\Timelike}_{\mu\nu}}\lesssim\abs{\pfour\vvariables}$. Thus by Cauchy-Schwarz inequality along $\Sigma_{\tau}$, we get the desired estimate.
\end{proof}

\begin{lemma}[Basic energy inequality for the transport equations]\label{basicenergy2}
	Let $\varphi$ be smooth on $[0,\Tstar]\times\mathbb{R}^3$. Under the bootstrap assumptions of Section \ref{bootstrap}, the following inequality holds for $t\in[0,\Tstar]$: 
	\begin{align}
	\label{energytransport}\norm{\varphi}^2_{L^2_x(\St)}\lesssim\norm{\varphi}^2_{L^2_x(\Sigma_{0})}+\int_{0}^t\norm{\pfour\vvariables}_{L_x^\infty(\Sigma_{\tau})}\norm{\varphi}^2_{L^2_x(\Sigma_{\tau})}\diff\tau+\int_0^t\norm{\varphi}_{L^2_x(\Sigma_{\tau})}\norm{\materialderivative\varphi}_{L^2_x(\Sigma_{\tau})}\diff\tau.
	\end{align}
\end{lemma}
\begin{proof}[Proof of Lemma \ref{basicenergy2}]
	Let $\mathbf{J}^\alpha:=\varphi^2\materialderivative^\alpha$, then $\p_\alpha \mathbf{J}^\alpha=2\varphi\materialderivative\varphi+\left(\p_\alpha\materialderivative^\alpha\right)\varphi^2$. We apply the divergence theorem on the space-time region $[0,t]\times\mathbb{R}^3$ relative to the Cartesian coordinates. Note that $\mathbf{J}^0=\varphi^2$. By Cauchy-Schwarz inequality along $\Sigma_{\tau}$, we obtain the desired estimates.
\end{proof}
\begin{remark}
	We remark that for our implementation of the geometric energy method for wave equations, the timelike vectorfield $\Timelike$ (defined in Definition \ref{timelike}) plays the same role as $\materialderivative$ (Note that $\materialderivative=\p_t+v^a\p_a$ in \cite{3DCompressibleEuler} is not the same as $\materialderivative=\frac{v^\alpha}{v^0}\p_\alpha$ in this paper.) in \cite[section 4.1]{3DCompressibleEuler}. All the arguments for geometric energy method for wave equations go through in the same fashion as in \cite[section 4.1]{3DCompressibleEuler}.
\end{remark}
\subsubsection{Elliptic div-curl estiamtes in $L^2$ space}\label{EllipticDiv-Curl}
This subsection is dedicated to the proof of Proposition \ref{EllipticestimatesinL2}, which is a key ingredient in the proof of the energy estimates (\ref{energyestimates}) for the $\vvort,\vgradientEnt,\vectvort,\DivGradEnt$.
	\begin{proposition}[Elliptic div-curl estimates in $L^2$ space]\label{EllipticestimatesinL2}Under the bootstrap assumptions in Section \ref{bootstrap}, the following estimates holds for $\vort$ and $\gradEnt$:
		\begin{align}\label{elliptic1}
		\norm{(\p\vvort,\p\vgradientEnt)}_{L_x^2(\St)}\lesssim\norm{\vvort,\vgradientEnt,\vectvort,\DivGradEnt}_{L_x^2(\St)}.
		\end{align}
		Moreover, $H^{N-k}$ elliptic estimate also holds true for $k=1,2$, that is,
		\begin{align}\label{elliptic2}
		\sobolevnorm{(\p\vvort,\p\vgradientEnt)}{N-k}{\St}&\lesssim\sobolevnorm{\vvort,\vgradientEnt,\vectvort,\DivGradEnt}{N-k}{\St},&
		\text{for } &k=1,2.
		\end{align}
	\end{proposition}

	 Since we have to derive energy estimates on constant-time hypersurfaces, and the Hodge system (\ref{DivCurlSystem}) is a \textbf{space-time} div-curl system, we begin by deriving a \textbf{spatial} div-curl system for $\vort$ and $\gradEnt$.
\begin{proposition}[The div-curl system on constant-time hypersurfaces]\label{NewDivCurl}
	Given the div-curl system (\ref{DivCurlSystem}), following equations hold on $\St$ for vorticity $\vort$ and entropy gradient $\gradEnt$:
\begin{subequations}
	\begin{align}
	\label{Gab0}(G^{-1})^{ab}\p_a(\vort_\flat)_b&=F_\vort,&  (G^{-1})^{ab}\p_a\gradEnt_b&=F_\gradEnt,\\
	\label{Hab0}\p_a(\vort_\flat)_b-\p_b(\vort_\flat)_a&=\Harg{\vort}{ab},& \p_a\gradEnt_b-\p_b\gradEnt_a&=\Harg{\gradEnt}{ab},
	\end{align}
\end{subequations}
	where 
	\begin{subequations}
		\begin{align}
		F_\vort&=\linsmoothfunction(\vvariables,\vvort,\vgradientEnt)[\pfour\vvariables],& F_\gradEnt&=\linsmoothfunction(\vvariables)\DivGradEnt+\linsmoothfunction(\vvariables,\vgradientEnt)[\pfour\vvariables],\\ \Harg{\vort}{ab}&=\linsmoothfunction(\vvariables)\vectvort+\linsmoothfunction(\vvariables,\vvort,\vgradientEnt)[\pfour\vvariables],&
		\Harg{\gradEnt}{ab}&=\linsmoothfunction(\vvariables,\vgradientEnt)[\pfour\vvariables],\\
		\label{D:Gab}(G^{-1})^{ab}&=\delta^{ab}-\frac{v^av^b}{(v^0)^2}.
		\end{align}
	\end{subequations}

For convenience, we write the above 2 div-curl systems as follows where $(\upeta,F,H_{ab})$ which\footnote{We use the same notation throughout the remainder of the article.} is either $(\vort_\flat,F_\vort,\Harg{\vort}{ab})$ or $(\gradEnt,F_\gradEnt,\Harg{\gradEnt}{ab})$:
	\begin{subequations}\label{newdivcurl}
	\begin{align}
	\label{Gab}(G^{-1})^{ab}\p_a\upeta_b&=F,\\
	\label{Hab}\p_a\upeta_b-\p_b\upeta_a&=H_{ab}.
	\end{align}
\end{subequations}
\end{proposition}
	\begin{proof}[Proof of Prop \ref{NewDivCurl}]
		
		For the div part (\ref{Gab}), by the equations (\ref{svothogonal}) and (\ref{identities1}), we write
		\begin{align}
		\label{rewrite1}\upeta_0=-\frac{\upeta_bv^b}{v^0}.
		\end{align}	
		Also by using transport equation (\ref{Bvort}) and (\ref{BS}), we have
		\begin{align}
		\label{rewrite2}\p_0(\upeta^\sharp)^b=-\frac{v^a\p_a(\upeta^\sharp)^b}{v^0}+\linsmoothfunction(\vvariables,\vvort,\gradEnt)[\pfour\vvariables].
		\end{align}
		
		Using (\ref{Divvort}) for $\upeta=\vort_\flat$, we write $\p_0\vort^0+\p_a\vort^a=\linsmoothfunction(\vvort)[\pfour\vvariables]$. By lowering the index $\p_0\vort^0=-\p_0(\vort_\flat)_0$, equations (\ref{rewrite1}) and (\ref{rewrite2}), we prove (\ref{Gab0}) for $\vort$. Similarly, for $\upeta=\gradEnt$, by definition of $\DivGradEnt$ (\ref{Def:D}), we write $\p_0(\gradEnt^\sharp)^0+\p_a(\gradEnt^\sharp)^a=\linsmoothfunction(\vvariables)\DivGradEnt+\linsmoothfunction(\vvariables,\vgradientEnt)[\pfour\vvariables]$. Using equations (\ref{rewrite1}) and (\ref{rewrite2}), we obtain the equation (\ref{Gab0}) for $\gradEnt$. 
		
		Now we consider the curl part, note that we have the facts (\ref{identities3}) and (\ref{identities4}):
		\begin{align}
		\label{rewrite3}\p_\gamma\upeta_\delta-\p_\delta\upeta_\gamma=&\epsilon_{\gamma\delta\kappa\lambda}v^\kappa \text{vort}^\lambda(\upeta)-(v^\kappa\p_\kappa\upeta_\delta)(v_\flat)_\gamma+v^\kappa(\p_\delta\upeta_\kappa)(v_\flat)_\gamma\\
		\notag&+(v^\kappa\p_\kappa\upeta_\gamma)(v_\flat)_\delta-v^\kappa(\p_\delta\upeta_\kappa)(v_\flat)_\delta.
		\end{align}		
		Recall that
		\begin{align}
		\modivort^{\alpha}=\text{vort}^\alpha(\vort_\flat)+\linsmoothfunction(\vvariables,\vgradientEnt)[\pfour\vvariables].
		\end{align}		
		Hence for $\upeta=\vort$, the first term on the right-hand side of (\ref{rewrite3}) for $\vort$ is manifestly in $\Harg{\vort}{ab}$. Next, using  $v^\kappa\p_\kappa(\vort_\flat)_\delta=v^0\materialderivative(\vort_\flat)_\delta$ and (\ref{Bvort}), as well as (\ref{identities2}), we have that the right-hand side of (\ref{rewrite3}) for $\vort$ is $\Harg{\vort}{ab}$. Similarly, by (\ref{vortS}), $v^\kappa\p_\kappa\gradEnt_\gamma=v^0\materialderivative\gradEnt_\gamma$, (\ref{BS}) and (\ref{identities2.5}), we obtain the equation (\ref{Hab0}) for $\gradEnt$.
	\end{proof}

\begin{remark}
	In terms of elliptic estimates, there is a major difference in Proposition \ref{EllipticestimatesinL2} compared to the Hodge system of the non-relativistic 3D compressible Euler equations. In the non-relativistic case, the analogous elliptic equations are constant-coefficient div-curl equations along flat hypersurfaces of constant Cartesian time and, for example, the basic $L^2$ theory can be derived with the simple Hodge identity for $\St$ vectorfields $V\in H^1(\mathbb{R}^3;\mathbb{R}^3)$:
	\begin{align}
	\sum\limits_{a,b=1}^3\norm{\p_aV^b}^2_{L^2_x(\mathbb{R}^3)}=\norm{\text{div}V}^2_{L^2_x(\mathbb{R}^3)}+\norm{\curl V}^2_{L^2_x(\mathbb{R}^3)}.
	\end{align}
	In contrast, the divergence equation (\ref{Gab}) has dynamic, solution-dependent coefficients.
\end{remark}
In order to derive elliptic estimates and Schauder estimates from the div-curl system (\ref{newdivcurl}), we need Lemma \ref{schauderalgebraic} provided below, which allows us to do estimates via Littlewood-Paley theory. We provide a partition of unity before the Lemma \ref{schauderalgebraic}. 

	We want to apply the Fourier transform to a localized version of the div-curl system in Proposition \ref{EllipticestimatesinL2}. We consider the lattice $\mathcal{A}:=\delta_2\mathbb{Z}^3$, where $\delta_2$ is assumed to be small and will be determined in future analysis. Notice that $\{x_l\}_{l\in\mathbb{N}}:=\mathcal{A}\subset\St$ has points equally spread out, that is, for each $x_l$, there are 6 points in $\mathcal{A}$ such that the distance between $x_l$ and any of them is $\delta_2$. We define the family of functions $\{\psi_l\}_{l\in\mathbb{N}}$ as follows:
\begin{align}
\psi_l(x)&=
\begin{cases}
1&  x\in B(x_l,\frac{1}{8}\delta_2),\\
\exp(\frac{4}{3\delta_2^2})\exp(\frac{1}{\abs{x-x_l}^2-(\frac{7}{8}\delta_2)^2})&  x\in B(x_l,\frac{7}{8}\delta_2)-B(x_l,\frac{1}{8}\delta_2),\\
0&  x\notin B(x_l,\frac{7}{8}\delta_2),
\end{cases} 
\end{align}
and set
\begin{align}
\phi_l(x):=\frac{\psi_l(x)}{\sum\limits_k\psi_k(x)}.
\end{align}
We have constructed cut-off functions $\{\phi_l\}_{l\in\mathbb{N}}\subset C_0^\infty(\Sigma_t)$ such that $\phi_l=1$ in $ B(x_l,\frac{1}{8}\delta_2)$, $\text{supp}(\phi_l)\subset B(x_l,\frac{7}{8}\delta_2)$, $\sum\limits_l\phi_l(x)=1$ and for any $x_a,x_b\in\mathcal{A}$, $\phi_a(x)=\phi_b(x-x_a+x_b)$.
We want to apply Fourier transform on a localized region, where $(G^{-1})^{ab}(x_l)$ is a constant and $(G^{-1})^{ab}(x)-(G^{-1})^{ab}(x_l)$ will be shown to be a controllable error term in the future analysis.
\begin{lemma}\label{schauderalgebraic}
	Given the Proposition \ref{NewDivCurl}, with $x_l$ and $\phi_l$, $l\in\mathbb{N}$, defined as above, let $\upeta_i$ be the solution of equations (\ref{newdivcurl}), then the following identity holds in frequency space for $i=1,2,3$:
	\begin{equation}
	\label{Gxi}(G^{-1})^{ab}(x_l)\xi_a\xi_b\widehat{(\phi_l\upeta_i)}=C\xi_i\hat{\underline{F}}^l+\sum\limits_{k\neq i}C(G^{-1})^{ak}(x_l)\xi_a\hat{\underline{H}}_{ki}^l,
	\end{equation}
	where rewrite
	\begin{subequations}
	\begin{align}
	\label{Frep}\underline{F}^l=&(G^{-1})^{ab}(x_l)\p_a(\phi_l\upeta_b)=\phi_l F+(G^{-1})^{ab}(x_l)(\p_a\phi_l)\upeta_b\\
	\notag&-\left[(G^{-1})^{ab}(x)-(G^{-1})^{ab}(x_l)\right]\p_a(\phi_l\upeta_b)-\left[(G^{-1})^{ab}(x)-(G^{-1})^{ab}(x_l)\right](\p_a\phi_l)\upeta_b,\\
	\label{Habrep}\underline{H}_{ab}^l=&\p_a(\phi_l\upeta_b)-\p_b(\phi_l\upeta_a)=(\p_a\phi_l)\upeta_b-(\p_b\phi_l)\upeta_a+\phi_l H_{ab}.
	\end{align}
\end{subequations}
\end{lemma}
\begin{proof}[Proof of Lemma \ref{schauderalgebraic}]
	
By multiplying the div-curl system (\ref{Gab}) and (\ref{Hab}) by $\phi_l$, we rewrite the system as follows:
	\begin{subequations}
		\begin{align}
		\label{newGab}\left\{(G^{-1})^{ab}(x_l)+(G^{-1})^{ab}(x)-(G^{-1})^{ab}(x_l)\right\}\left\{\p_a(\phi_l\upeta_b)-(\p_a\phi_l)\upeta_b\right\}&=\phi_l F,\\
		\label{newHab}\p_a(\phi_l\upeta_b)-(\p_a\phi_l)\upeta_b-\p_b(\phi_l\upeta_a)+(\p_b\phi_l)\upeta_a&=\phi_l H_{ab}.
		\end{align}
	\end{subequations}
Taking the Fourier transform of (\ref{newGab}) and multiplying by $\xi_1$, we have
		\begin{align}
		\label{Algebra1}\xi_1\left\{(G^{-1})^{a1}(x_l)\xi_a\widehat{(\phi_l\upeta_1)}+(G^{-1})^{a2}(x_l)\xi_a\widehat{(\phi_l\upeta_2)}+(G^{-1})^{a3}(x_l)\xi_a\widehat{(\phi_l\upeta_3)}\right\}=C\xi_1\hat{\underline{F}}^l,
		\end{align}
			where $C=\frac{1}{2\pi i}$ is a constant from Fourier transform, and
		\begin{align}
		\underline{F}^l=&(G^{-1})^{ab}(x_l)\p_a(\phi_l\upeta_b)=\phi_l F+(G^{-1})^{ab}(x_l)(\p_a\phi_l)\upeta_b\\
		\notag&-\left[(G^{-1})^{ab}(x)-(G^{-1})^{ab}(x_l)\right]\p_a(\phi_l\upeta_b)-\left[(G^{-1})^{ab}(x)-(G^{-1})^{ab}(x_l)\right](\p_a\phi_l)\upeta_b.
		\end{align}
Similarly, taking the Fourier transform of (\ref{newHab}) and multiplying by $(G^{-1})^{a2}(x_l)\xi_a$ and $(G^{-1})^{a3}(x_l)\xi_a$, we have 
\begin{subequations}
	\begin{align}
	\label{Algebra2}&(G^{-1})^{a2}(x_l)\xi_a\left\{\xi_2\widehat{(\phi_l\upeta_1)}-\xi_1\widehat{(\phi_l\upeta_2)}\right\}=C(G^{-1})^{a2}(x_l)\xi_a\hat{\underline{H}}_{21}^l,\\
	\label{Algebra3}&(G^{-1})^{a3}(x_l)\xi_a\left\{\xi_3\widehat{(\phi_l\upeta_1)}-\xi_1\widehat{(\phi_l\upeta_3)}\right\}=C(G^{-1})^{a3}(x_l)\xi_a\hat{\underline{H}}_{31}^l,
	\end{align}
\end{subequations}
	where $C=\frac{1}{2\pi i}$ is a constant from Fourier transform, and
		\begin{align}
	\underline{H}_{ab}^l&=\p_a(\phi_l\upeta_b)-\p_b(\phi_l\upeta_a)=(\p_a\phi_l)\upeta_b-(\p_b\phi_l)\upeta_a+\phi_l H_{ab}.
		\end{align}
	Adding (\ref{Algebra1}), (\ref{Algebra2}) and (\ref{Algebra3}), we obtain
	\begin{equation}
	(G^{-1})^{ab}(x_l)\xi_a\xi_b\widehat{(\phi_l\upeta_1)}=C\xi_1\hat{\underline{F}}^l+C(G^{-1})^{a2}(x_l)\xi_a\hat{\underline{H}}^l_{21}+C(G^{-1})^{a3}(x_l)\xi_a\hat{\underline{H}}_{31}^l.
	\end{equation}
	
	We use the same argument for $\upeta_2$ and $\upeta_3$. Hence for $i=1,2,3$, we obtain
	\begin{equation}
	(G^{-1})^{ab}(x_l)\xi_a\xi_b\widehat{(\phi_l\upeta_i)}=C\xi_i\hat{\underline{F}}^l+\sum\limits_{k\neq i}C(G^{-1})^{ak}(x_l)\xi_a\hat{\underline{H}}_{ki}^l.
	\end{equation}
\end{proof}
\begin{lemma}[Positive definiteness of $G^{-1}$]\label{Positivedefinite}
	For any $\St$-tangent one-form $\xi$, we denote $\abs{\xi}^2:=\sum\limits_{i=1,2,3}\xi_i^2$. Then the following estimate holds for any $x_l\in\St$, where $G$ is defined in Proposition \ref{NewDivCurl}:
	\begin{align}\label{RiemannnianG}
	C|\xi|^2\leq (G^{-1})^{ab}(x_l)\xi_a\xi_b\leq|\xi|^2,
	\end{align}
	where $0<C<1$ is a constant depends only on $\norm{v}_{L_x^\infty(\St)}$, which is in turn controlled by the bootstrap assumption (\ref{databa}).
\end{lemma}
	\begin{proof}[Proof of Lemma \ref{Positivedefinite}]Using the definition of $(G^{-1})^{ab}$, we have
		\begin{equation}
		(G^{-1})^{ab}(x_l)\xi_a\xi_b
		=\left(\delta^{ab}-\frac{v^av^b}{(v^0)^2}\right)\xi_a\xi_b=|\xi|^2-\left(\frac{v^i\xi_i}{v^0}\right)^2.
		\end{equation}		
		Hence by the normalization $(v_\flat)_\alpha v^\alpha=-1$ in Section \ref{thebasicfluidvariables}, we have
		\begin{align}
		C|\xi|^2\leq (G^{-1})^{ab}(x_l)\xi_a\xi_b\leq|\xi|^2,
		\end{align}
		where $0<C<1$ is a constant depends only on $\norm{v}_{L_x^\infty(\St)}$.
	\end{proof}

	\begin{proof}[Proof of Proposition \ref{EllipticestimatesinL2}]
Throughout, $\upeta, F, G, H$ are the same as in Proposition \ref{NewDivCurl} and $\underline{F}, \underline{H}$ are defined in Lemma \ref{schauderalgebraic}. Since $\abs{\xi}\simeq2\upnu$ on support of $\widehat{\littlewood(\phi_l\upeta_i)}$, by Littlewood-Paley estimate (\ref{Littlewoodsobolev}), (\ref{RiemannnianG}) and (\ref{Gxi}), we remind $\psi$ is defined in Section \ref{definitionlittlewood}
	\begin{align}\label{Elliptic1}
	\norm{\p\left\{\phi_l\upeta_i\right\}}_{L_x^2(\St)}^2
	&\leq C\sum\limits_{\upnu}\norm{\upnu\littlewood(\phi_l\upeta_i)}_{L_x^2(\St)}^2\\
	\notag&\leq C\sum\limits_{\upnu}\|\mathcal{F}^{-1}\left(\psi\left(\frac{|\xi|}{\upnu}\right)|\upnu|^{-1}\left\{\xi_i\hat{\underline{F}}^l+\sum\limits_{k\neq i}(G^{-1})^{jk}(x_l)\xi_j\hat{\underline{H}}^l_{ki}\right\}\right)\|_{L_x^2(\St)}^2\\
	\notag&\leq C\left(\norm{\underline{F}^l}^2_{L_x^2(\St)}+\norm{\underline{H}^l}^2_{L_x^2(\St)}\right),
	\end{align}
	where $\mathcal{F}^{-1}$ is the Fourier inverse transform. By (\ref{Frep}), and the fundamental theorem of Calculus to $G^{ab}(x)-G^{ab}(x_l)$, we have,
	\begin{align}\label{Fanaysis0}
	\norm{\underline{F}^l}_{L_x^2(\St)}^2\leq&\norm{\phi_lF}_{L_x^2(\St)}^2+\norm{(G^{-1})}_{L_x^\infty(\St)}^2\norm{\p\phi_l}_{L_x^\infty(\St)}^2\norm{\upeta}_{L_x^2(\St)}^2\\
	\notag&+\sup\limits_{x\in B(x_l,\delta_2)}\abs{x-x_l}^{2}\norm{\p (G^{-1})}_{L_x^\infty(\St)}^2\norm{\p(\phi_l\upeta)}_{L_x^2(\St)}^2.
	\end{align}
	
	By (\ref{Habrep}),
	\begin{equation}
	\label{Hanalysis0}
	\norm{\underline{H}^l}_{L_x^2(\St)}^2\leq\norm{\phi_l H}_{L_x^2(\St)}^2+2\norm{\p\phi_l}_{L_x^\infty(\St)}^2\norm{\upeta}_{L_x^2(\St)}^2.
	\end{equation}
	
	Note that $\norm{\p (G^{-1})}_{L_x^\infty(\St)}\lesssim\norm{\pfour\vvariables}_{L_x^\infty(\St)}$. Let $\delta_2$ be small such that for all $x\in B(x_l,\delta_2)$,
	\begin{equation}\label{Schauder15}
	C\abs{x-x_l}^{2}\norm{\p (G^{-1})}_{L_x^\infty(\St)}^2<\frac{1}{4}.
	\end{equation}
	
	Hence by (\ref{Elliptic1}), (\ref{Fanaysis0}), (\ref{Hanalysis0}), (\ref{Schauder15}) and definition of $F,H$,  we have
	\begin{align}\label{Elliptic2}
	\norm{\p\left\{\phi_l\upeta_i\right\}}_{L_x^2(\St)}^2\lesssim\norm{\phi_l(\upeta,\modivort,\DivGradEnt)}_{L_x^2(\St)}^2.
	\end{align}
	
	Now we consider the $\norm{\p\upeta_i}_{L^2_x(\St)}^2$:
	\begin{align}\label{pupeta}
	\norm{\p\upeta_i}_{L_x^2(\St)}^2&=\norm{\sum\limits_l\p\left\{\phi_l\upeta_i\right\}}_{L_x^2(\St)}^2\\
	\notag&=\int_{\St}\sum\limits_l|\p\left\{\phi_l\upeta_i\right\}|^2+\sum\limits_{k,m}|\p\left\{\phi_k\upeta_i\right\}||\p\left\{\phi_m\upeta_i\right\}|\diff x.
	\end{align}
	
	By the construction of $\phi_l$, there are only finitely many pairs of $k,m$ such that $|\p\left\{\phi_k\upeta_i\right\}||\p\left\{\phi_m\upeta_i\right\}|\neq0$. Moreover, for each  $|\p\left\{\phi_k\upeta_i\right\}||\p\left\{\phi_m\upeta_i\right\}|\neq0$, $|\p\left\{\phi_k\upeta_i\right\}||\p\left\{\phi_m\upeta_i\right\}|\leq|\p\left\{\phi_n\upeta_i\right\}|^2$ for some $n$. Hence by (\ref{pupeta}) and (\ref{Elliptic2}) we have,
	\begin{align}
	\norm{\p\upeta_i}_{L_x^2(\St)}^2&\lesssim\sum\limits_l\norm{\p\left\{\phi_l\upeta_i\right\}}_{L_x^2(\St)}^2\lesssim\sum\limits_l\norm{\phi_l(\upeta,\modivort,\DivGradEnt)}_{L_x^2(\St)}^2\lesssim\norm{\upeta,\modivort,\DivGradEnt}_{L_x^2(\St)}^2.
	\end{align}
	
	We have proved (\ref{elliptic1}). Using (\ref{Littlewoodsobolev}), (\ref{elliptic2}) can be proved in a similar fashion by using the following estimates compared to (\ref{Elliptic1}) in $L^2$ elliptic estimate:
	\begin{subequations}
		\begin{align}\label{Elliptic3}
	\sobolevnorm{\p\left\{\phi_l\upeta_i\right\}}{N-k}{\St}
	\leq \norm{\p\left\{\phi_l\upeta_i\right\}}_{L^2_x(\St)}+C\sum\limits_{\upnu}\norm{\upnu^{N-k+1}\littlewood(\phi_l\upeta_i)}_{L_x^2(\St)}^2.
	\end{align}
	\begin{align}	
	C\sum\limits_{\upnu}\norm{\upnu^{N-k+1}\littlewood(\phi_l\upeta_i)}_{L_x^2(\St)}^2
	&\leq C\sum\limits_{\upnu}\|\mathcal{F}^{-1}\left(\psi\left(\frac{|\xi|}{\upnu}\right)|\upnu|^{N-k-1}\left\{\xi_i\hat{\underline{F}}^l+\sum\limits_{k\neq i}(G^{-1})^{jk}(x_l)\xi_j\hat{\underline{H}}^l_{ki}\right\}\right)\|_{L_x^2(\St)}^2\\
	\notag&\leq C\left(\sobolevnorm{\underline{F}^l}{N-k}{\St}^2+\sobolevnorm{\underline{H}^l}{N-k}{\St}^2\right).
	\end{align}
	\end{subequations}
\end{proof}  
\subsubsection{Proof of Proposition \ref{EnergyandEllipticEstimates}}\label{Discussionofprop5.1}
\begin{proof}
	For $N$ defined as in Section \ref{ChoiceofParameters}, we let
	\begin{align}\label{totalenergy}
P_N(t):=\sum\limits_{k=0}^2\sobolevnorm{\pfour^k(\vvariables,\vvort,\vgradientEnt)}{N-k}{\St}^2+\sum\limits_{k=0}^1\sobolevnorm{\pfour^k(\vectvort,\DivGradEnt)}{N-k-1}{\St}^2.
	\end{align}
	
	In this proof, we derive integral inequalities for $\sum\limits_{k=0}^2\sobolevnorm{\pfour^k\vvariables}{N-k}{\St}^2$ and $\sum\limits_{k=0}^1\sobolevnorm{\pfour^k(\vectvort,\DivGradEnt)}{N-k-1}{\St}^2$ in $P_N(t)$, namely (\ref{proof6.1.2}), (\ref{proof6.1.7}), (\ref{proof6.1.4}), (\ref{proof6.1.6}). We then use elliptic estimates (\ref{elliptic1}) and apply Gr\"onwall's inequality to all the terms in $P_N(t)$ collectively.
	
The proof of Prop.\ref{EnergyandEllipticEstimates} for $\vvariables$ combines the vectorfield multiplier method and Littlewood-Paley theory. That is, to derive the energy estimates at the top order, one integrate (\ref{DalphaJalpha}) and applies the divergence theorem using the energy current $\Jenarg{\Timelike}{\alpha}[\pfour\vvariables]:=Q^{\alpha\beta}[\pfour\vvariables]\Timelike_\beta-\Timelike^\alpha(\pfour\vvariables)^2$ and  $\Jenarg{\Timelike}{\alpha}[\littlewood\pfour\vvariables]:=Q^{\alpha\beta}[\littlewood\pfour\vvariables]\Timelike_\beta-\Timelike^\alpha(\littlewood\pfour\vvariables)^2$ on the space-time region bounded by $\Sigma_{0}$ and $\Sigma_{t}$. Then by Lemma \ref{basicenergy} with $\pfour\vvariables$ and $\littlewood\pfour\vvariables$ in a role of $\varphi$, we have, respectively,
\begin{subequations}
	\begin{align}
\mathbb{E}[\pfour\vvariables](t)\lesssim &\mathbb{E}[\pfour\vvariables](0)
+\int_{0}^{t}\norm{\pfour\vvariables}_{L_x^\infty(\Sigma_{\tau})}\mathbb{E}[\pfour\vvariables](\tau)\diff\tau\\
\notag&+\int_{0}^{t}\norm{\boxg\pfour\vvariables}_{L^2_x(\Sigma_{\tau})}\left[\mathbb{E}[\pfour\vvariables](\tau)\right]^{1/2}\diff\tau,\\
\label{proof6.1.1}\mathbb{E}[\littlewood\pfour\vvariables](t)\lesssim &\mathbb{E}[\littlewood\pfour\vvariables](0)
+\int_{0}^{t}\norm{\pfour\vvariables}_{L_x^\infty(\Sigma_{\tau})}\mathbb{E}[\littlewood\pfour\vvariables](\tau)\diff\tau\\
\notag&+\int_{0}^{t}\norm{\boxg\littlewood\pfour\vvariables}_{L^2_x(\Sigma_{\tau})}\left[\mathbb{E}[\littlewood\pfour\vvariables](\tau)\right]^{1/2}\diff\tau.
\end{align}
\end{subequations}
Then, we use equation (\ref{boxg2}) to substitute for $\boxg\pfour\vvariables$, and we use equation (\ref{lwpboxg}) $\boxg\littlewood\pfour\vvariables$ to substitute for the right-hand side of (\ref{proof6.1.1}). Multiplying (\ref{proof6.1.1}) by $\upnu^{2(N-2)}$, summing over $\upnu$ and using (\ref{Littlewoodsobolev}) , estimates (\ref{remainderestimates}) and H\"older's inequality, we have
\begin{align}\label{proof6.1.2}
\sobolevnorm{\pfour^2\vvariables}{N-2}{\St}^2\lesssim&\sobolevnorm{\pfour^2\vvariables}{N-2}{\Sigma_{0}}^2+\int_{0}^{t}\norm{\pfour\vvariables}_{L_x^\infty(\Sigma_{\tau})}\sobolevnorm{\pfour^2\vvariables}{N-2}{\Sigma_{\tau}}^2\diff\tau\\
\notag&+\int_{0}^{t}\left\{\sobolevnorm{\p(\vectvort,\DivGradEnt)}{N-2}{\Sigma_{\tau}}\sobolevnorm{\pfour^2\vvariables}{N-2}{\Sigma_{\tau}}\right.\\
\notag&+\left.\left(\norm{\pfour(\vvariables,\vvort,\vgradientEnt)}_{L_x^\infty(\Sigma_{\tau})}+1\right)\left(\sobolevnorm{\pfour(\vvariables,\vvort,\vgradientEnt)}{N-1}{\Sigma_{\tau}}+1\right)\sobolevnorm{\pfour^2\vvariables}{N-2}{\Sigma_{\tau}}\right\}\diff\tau\\
\notag&\lesssim P_N(0)+\int_{0}^{t}\norm{\pfour(\vvariables,\vvort,\vgradientEnt)}_{L_x^\infty(\Sigma_{\tau})}\diff\tau+\int_{0}^{t}\left(\norm{\pfour(\vvariables,\vvort,\vgradientEnt)}_{L_x^\infty(\Sigma_{\tau})}+1\right)P_N(\tau)\diff\tau.
\end{align}
For $\sobolevnorm{\pfour\vvariables}{N-1}{\St}$, we first have
\begin{align}\label{proof6.1.2.4}
\sobolevnorm{\pfour\vvariables}{N-1}{\St}^2\lesssim\norm{\pfour\vvariables}_{L^2_x(\St)}^2+\sobolevnorm{\pfour^2\vvariables}{N-2}{\St}^2.
\end{align}
Then, by the fundamental theorem of Calculus in time, Minkowski integral inequality, and smallness of $\Tstar$, we have 
\begin{align}\label{proof6.1.2.5}
\norm{\pfour\vvariables}_{L^2_x(\St)}^2=\int_{\St}\left\{\pfour\vvariables(0,x)+\int_0^t\p_t\pfour\vvariables(\tau,x)\diff\tau\right\}^2\diff x\lesssim P_N(0)+\sobolevnorm{\pfour^2\vvariables}{N-2}{\St}^2.
\end{align}
Similarly, we have
\begin{align}\label{proof6.1.2.6}
\sobolevnorm{\vvariables}{N}{\St}\lesssim P_N(0)+\sobolevnorm{\pfour^2\vvariables}{N-2}{\St}^2.
\end{align}
Therefore, by adding (\ref{proof6.1.2})-(\ref{proof6.1.2.5}), we have
\begin{align}\label{proof6.1.7}
\sum\limits_{k=0}^2\sobolevnorm{\pfour^k\vvariables}{N-k}{\St}^2\lesssim P_N(0)+\int_{0}^{t}\norm{\pfour(\vvariables,\vvort,\vgradientEnt)}_{L^\infty_x(\Sigma_{\tau})}\diff\tau+\int_{0}^{t}\left(\norm{\pfour(\vvariables,\vvort,\vgradientEnt)}_{L^\infty_x(\Sigma_{\tau})}+1\right)P_N(\tau)\diff\tau.
\end{align}
Now we derive top-order estimates for $\vectvort$ and $\DivGradEnt$. We apply the energy estimates (\ref{energytransport}) with $\pfour(\vectvort,\DivGradEnt)$ and $\littlewood\pfour(\vectvort,\DivGradEnt)$ in a role of $\varphi$ respectively to obtain,
\begin{subequations}
	\begin{align}
\norm{\pfour(\vectvort,\DivGradEnt)}^2_{L^2_x(\St)}\lesssim&\norm{\pfour(\vectvort,\DivGradEnt)}^2_{L^2_x(\Sigma_{0})}+\int_{0}^t\norm{\pfour\vvariables}_{L_x^\infty(\Sigma_{\tau})}\norm{\pfour(\vectvort,\DivGradEnt)}^2_{L^2_x(\Sigma_{\tau})}\diff\tau\\
\notag&+\int_0^t\norm{\pfour(\vectvort,\DivGradEnt)}_{L^2_x(\Sigma_{\tau})}\norm{\materialderivative\pfour(\vectvort,\DivGradEnt)}_{L^2_x(\Sigma_{\tau})}\diff\tau,\\
\label{proof6.1.3}\norm{\littlewood\pfour(\vectvort,\DivGradEnt)}^2_{L^2_x(\St)}\lesssim&\norm{\littlewood\pfour(\vectvort,\DivGradEnt)}^2_{L^2_x(\Sigma_{0})}+\int_{0}^t\norm{\pfour\vvariables}_{L^\infty_x(\Sigma_{\tau})}\norm{\littlewood\pfour(\vectvort,\DivGradEnt)}^2_{L^2_x(\Sigma_{\tau})}\diff\tau\\
\notag&+\int_0^t\norm{\littlewood\pfour(\vectvort,\DivGradEnt)}_{L^2_x(\Sigma_{\tau})}\norm{\materialderivative\littlewood\pfour(\vectvort,\DivGradEnt)}_{L^2_x(\Sigma_{\tau})}\diff\tau.
\end{align} 
\end{subequations}
We use equations (\ref{bc2})-(\ref{bd2}) to substitute for $\materialderivative\pfour(\vectvort,\DivGradEnt)$, and we use equations (\ref{lwpbc}) and (\ref{lwpbd}) for $\materialderivative\littlewood\pfour(\vectvort,\DivGradEnt)$ to substitute for the right-hand side of (\ref{proof6.1.3}). Multiplying (\ref{proof6.1.3}) by $\upnu^{2(N-2)}$, summing over $\upnu$ and using (\ref{Littlewoodsobolev}), using estimates (\ref{remainderestimates}) and elliptic estimates (\ref{elliptic1})-(\ref{elliptic2}), we have
\begin{align}\label{proof6.1.4}
\sobolevnorm{\pfour(\vectvort,\DivGradEnt)}{N-2}{\St}^2\lesssim&\sobolevnorm{\pfour(\vectvort,\DivGradEnt)}{N-2}{\Sigma_{0}}^2+\int_{0}^t\norm{\pfour\vvariables}_{L^\infty_x(\Sigma_{\tau})}\sobolevnorm{\pfour(\vectvort,\DivGradEnt)}{N-2}{\Sigma_{\tau}}^2\diff\tau\\
\notag&+\int_{0}^{t}\left\{\sobolevnorm{\p(\vectvort,\DivGradEnt)}{N-2}{\Sigma_{\tau}}\sobolevnorm{\pfour(\vectvort,\DivGradEnt)}{N-2}{\Sigma_{\tau}}\right.\\
\notag&+\left.\left(\norm{\pfour(\vvariables,\vvort,\vgradientEnt)}_{L^\infty_x(\Sigma_{\tau})}+1\right)\left(\sobolevnorm{\pfour(\vvariables,\vvort,\vgradientEnt)}{N-1}{\Sigma_{\tau}}+1\right)\sobolevnorm{\pfour(\vectvort,\DivGradEnt)}{N-2}{\Sigma_{\tau}}\diff\tau\right\}\\
\notag&\lesssim P_N(0)+\int_{0}^{t}\norm{\pfour(\vvariables,\vvort,\vgradientEnt)}_{L^\infty_x(\Sigma_{\tau})}\diff\tau+\int_{0}^{t}\left(\norm{\pfour(\vvariables,\vvort,\vgradientEnt)}_{L^\infty_x(\Sigma_{\tau})}+1\right)P_N(\tau)\diff\tau.
\end{align} 
For $\sobolevnorm{\vectvort,\DivGradEnt}{N-1}{\St}$, using the same method as in (\ref{proof6.1.2.4})-(\ref{proof6.1.7}), by (\ref{proof6.1.4}), we have
\begin{align}\label{proof6.1.6}
\sobolevnorm{\vectvort,\DivGradEnt}{N-1}{\St}^2\lesssim P_N(0)+\int_{0}^{t}\norm{\pfour(\vvariables,\vvort,\vgradientEnt)}_{L^\infty_x(\Sigma_{\tau})}\diff\tau+\int_{0}^{t}\left(\norm{\pfour(\vvariables,\vvort,\vgradientEnt)}_{L^\infty_x(\Sigma_{\tau})}+1\right)P_N(\tau)\diff\tau.
\end{align}
Combining (\ref{proof6.1.2}), (\ref{proof6.1.7}), (\ref{proof6.1.4}), (\ref{proof6.1.6}) and elliptic estimates (\ref{elliptic1})-(\ref{elliptic2}), we have,
\begin{align}\label{proof6.1.8}
P_N(t)\lesssim P_N(0)+\int_{0}^{t}\norm{\pfour(\vvariables,\vvort,\vgradientEnt)}_{L^\infty_x(\Sigma_{\tau})}\diff\tau+\int_{0}^{t}\left(\norm{\pfour(\vvariables,\vvort,\vgradientEnt)}_{L^\infty_x(\Sigma_{\tau})}+1\right)P_N(\tau)\diff\tau.
\end{align}
By bootstrap assumptions (\ref{waveba})-(\ref{transportba}), H\"older inequality in time and Gr\"onwall's inequality, we obtain the desired result.
\end{proof}
\subsection{Schauder Estimates}\label{Schauder05}
In this section, we bound the H\"older norms of the modified fluid variables $\modivort, \DivGradEnt$ and the derivatives of vorticity and entropy gradient. Moreover, in Proposition \ref{Schauder0}, we reduce the proof of the improvement of bootstrap assumption (\ref{transportba}) to the proof of the improvement of bootstrap assumption (\ref{waveba}).
\begin{theorem}[Improvements of the bootstrap assumptions for the vorticity and entropy gradient]\label{Schauder0}
	Let $\delta$ and $\delta_1$ be as in Section \ref{ChoiceofParameters}. Under the initial data and bootstrap assumptions of Section \ref{sectionmainthm}, assume the improved estimates in Theorem \ref{MainEstimates} holds for $\pfour\vvariables$, that is,
	\begin{align}
	\label{preimprove}\twonorms{\pfour\vvariables}{t}{2}{x}{\infty}{[0,\Tstar]\times\mathbb{R}^3}^2+\sum\limits_{\upnu\geq2}\upnu^{2\delta_1}\twonorms{\littlewood\pfour\vvariables}{t}{2}{x}{\infty}{[0,\Tstar]\times\mathbb{R}^3}^2\lesssim\Tstar^{2\delta},
	\end{align} 
	then the following estimates hold:

		\begin{align}
\label{improvedboot2}\sum\limits_{\upnu\geq2}\upnu^{\delta_1}\twonorms{\littlewood(\p\vvort,\p\vgradientEnt)}{t}{2}{x}{\infty}{[0,\Tstar]\times\mathbb{R}^3}^2\lesssim \Tstar^{2\delta}.
		\end{align}
\end{theorem}

We emphasize that the proof of the Strichartz estimates (\ref{estimateswavevariables}) is independent of (\ref{improvedboot2}). 

We will prove Theorem \ref{Schauder0} in Section \ref{Schauder6}. In this subsection, we derive a Schauder type estimate in the following lemma:
\begin{lemma}\label{Schaudertypeestimate}
		Let $\delta_1$ be as in Section \ref{ChoiceofParameters}. Under the initial data and bootstrap assumptions of Section \ref{sectionmainthm}, the following estimates hold:
\begin{align}
 \label{Schauder4}\holdernorm{\pfour\vvort,\pfour\vgradientEnt}{x}{0}{\delta_1}{\St}\lesssim\holdernorm{\pfour\vvariables,\vectvort,\DivGradEnt}{x}{0}{\delta_1}{\St}. 
\end{align}
\end{lemma}
\begin{proof}[Proof of Lemma \ref{Schaudertypeestimate}]
By Littlewood-Paley theory (\ref{Littlewoodholder}) and (\ref{Gxi}) where $\upeta, F, G, H$ are the same as in Proposition \ref{NewDivCurl} and $\underline{F}, \underline{H}$ are defined in Lemma \ref{schauderalgebraic},
	\begin{align}\label{Schauder1}
	\semiholdernorm{\p_s\left(\phi_l\upeta_i\right)}{x}{0}{\delta_1}{\St}
	\approx&\sup\limits_{\upnu\geq2}\norm{\upnu^{\delta_1}\littlewood(\p_s\phi_l\upeta_i)}_{L_x^\infty}\\
	\notag=&\sup\limits_{\upnu\geq2}\norm{C\upnu^{\delta_1}\mathcal{F}^{-1}\left\{\psi(\upnu^{-1}\xi) \frac{\xi_s}{(G^{-1})^{ab}(x_l)\xi_a\xi_b}\left(\xi_i\hat{\underline{F}}^l+\sum\limits_{k\neq i}(G^{-1})^{jk}(x_l)\xi_j\hat{\underline{H}}^l_{ki}\right)\right\}}_{L_x^\infty}\\
	\notag\leq& C\holdernorm{\mathcal{F}^{-1}\left\{ \frac{\xi_s}{(G^{-1})^{ab}(x_l)\xi_a\xi_b}\xi_i\hat{\underline{F}}^l\right\}}{x}{0}{\delta_1}{\St}
	\\
	\notag&+\sum\limits_{k\neq i}\holdernorm{\mathcal{F}^{-1}\left\{ \frac{\xi_s}{(G^{-1})^{ab}(x_l)\xi_a\xi_b}(G^{-1})^{jk}(x_l)\xi_j\hat{\underline{H}}^l_{ki}\right\}}{x}{0}{\delta_1}{\St},
	\end{align}
	where $\mathcal{F}^{-1}$ is the inverse Fourier transform operator and $\psi$ is defined in Section \ref{definitionlittlewood}.\\
	Now let's consider the first term in the right-hand side of last line of (\ref{Schauder1}). For each fixed $s,i=1,2,3$, we define function $p_{s,i}(x,\xi)$ as follows:
	\begin{align}
	p_{s,i}(x,\xi):=\frac{\xi_s\xi_i}{(G^{-1})^{ab}(x_l)\xi_a\xi_b}.
	\end{align} 	
	The associated pseudodifferential operator $p(x,D;s,i)$ is defined by using Fourier integral representation as follows:
	\begin{align}
	p_{s,i}(x,D)f(x):=\int p_{s,i}(x,\xi)\hat{f}(\xi)e^{ix\xi}\diff\xi.
	\end{align}	
	By direct computation and positive definiteness of $G$ which is showed in Lemma \ref{Positivedefinite}, we have 
	\begin{align}
	\abs{D^\beta_xD^\alpha_\xi p_{s,i}(x,\xi)}\leq C_{\alpha\beta}(1+\abs{\xi}^2)^{\frac{-|\alpha|}{2}}.
	\end{align}	
	So $p_{s,i}(x,\xi)$ is in the H\"ormander class $S^0_{1,0}$ and $p_{s,i}(x,D)$ belongs to $OPS^0_{1,0}$. By the theory of pseudodifferential operators, we have $p_{s,i}(x,D): C_x^{0,\delta_1}\rightarrow C_x^{0,\delta_1}$. We refer reader to \cite[Chapter 18]{Theanalysisoflinearpartialdifferentialoperatorsthree} for explicit definition of H\"ormander class and \cite[Proposition 2.1.D]{pseudodifferentialoperatorsandnonlinearpde} for the bounds of the operator $p_{s,i}(x,D)$. Therefore,
	\begin{align}\label{Fa1}
	\holdernorm{\mathcal{F}^{-1}\left\{ \frac{\xi_s}{(G^{-1})^{ab}(x_l)\xi_a\xi_b}\xi_i\hat{\underline{F}}^l\right\}}{x}{0}{\delta_1}{\St}\leq C\holdernorm{\underline{F}^l}{x}{0}{\delta_1}{\St}.
	\end{align}	
	Similarly, we have
	\begin{align}\label{Ha1}
	\sum\limits_{k\neq i}\holdernorm{\mathcal{F}^{-1}\left\{ \frac{\xi_s}{(G^{-1})^{ab}(x_l)\xi_a\xi_b}(G^{-1})^{jk}(x_l)\xi_j\hat{\underline{H}}^l_{ki}\right\}}{x}{0}{\delta_1}{\St}\leq C\sum\limits_{k\neq i}\holdernorm{\underline{H}^l_{ki}}{x}{0}{\delta_1}{\St}.
	\end{align}	
	Combining (\ref{Schauder1}), (\ref{Fa1}) and (\ref{Ha1}),  for any $l,s,i$, we have
	\begin{align}\label{holderbound}
	\semiholdernorm{\p_s\left(\phi_l\upeta_i\right)}{x}{0}{\delta_1}{\St}\leq C\left(\holdernorm{\underline{F}^l}{x}{0}{\delta_1}{\St}+\sum\limits_{k\neq i}\holdernorm{\underline{H_{ki}^l}}{x}{0}{\delta_1}{\St}\right),
	\end{align}
	where the constant $C$ is independent of $l,s,i$.
	By (\ref{Frep}),
	\begin{align}\label{Fanaysis}
	\holdernorm{\underline{F}^l}{x}{0}{\delta_1}{\St}\leq&\holdernorm{\phi_lF}{x}{0}{\delta_1}{\St}+\holdernorm{G^{-1}}{x}{0}{\delta_1}{\St}\holdernorm{\p\phi_l}{x}{0}{\delta_1}{\St}\holdernorm{\upeta}{x}{0}{\delta_1}{\St}\\
	\notag&+\sup\limits_{x,y\in B(x_l,\delta_2)}\abs{x-y}^{1-\delta_1}\onenorm{\p G^{-1}}{x}{\infty}{\St}\holdernorm{\p(\phi_l\upeta)}{x}{0}{\delta_1}{\St}.
	\end{align}	
	By (\ref{Habrep}),
	\begin{equation}
	\label{Hanalysis}
	\holdernorm{\underline{H}^l}{x}{0}{\delta_1}{\St}\leq\holdernorm{\phi_l H}{x}{0}{\delta_1}{\St}+2\holdernorm{\p\phi_l}{x}{0}{\delta_1}{\St}\holdernorm{\upeta}{x}{0}{\delta_1}{\St}.
	\end{equation}	
	Let $\delta_2$ be small such that for all $x\in B(x_l,\delta_2)$, for $C$ as in (\ref{holderbound}),
	\begin{equation}\label{Schauder2}
	C\cdot\sup\limits_{x,y\in B(x_l,\delta_2)}\abs{x-y}^{1-\delta_1}\onenorm{\p G^{-1}}{x}{\infty}{\St}<\frac{1}{4}.
	\end{equation}	
	Combining (\ref{holderbound}),(\ref{Fanaysis}), (\ref{Hanalysis}) and (\ref{Schauder2}), we have 
	\begin{align}
	C^{-1}\semiholdernorm{\p\left(\phi_l\upeta_i\right)}{x}{0}{\delta_1}{\St}\leq&\holdernorm{\phi_lF}{x}{0}{\delta_1}{\St}+\holdernorm{G^{-1}}{x}{0}{\delta_1}{\St}\holdernorm{\p\phi_l}{x}{0}{\delta_1}{\St}\holdernorm{\upeta}{x}{0}{\delta_1}{\St}+\holdernorm{\phi_l H}{x}{0}{\delta_1}{\St}\\
	\notag&+\holdernorm{\p\phi_l}{x}{0}{\delta_1}{\St}\holdernorm{\upeta}{x}{0}{\delta_1}{\St},
	\end{align}
	where constant $C$ is independent of $l,i$.
	Now notice that
	\begin{equation}
	\semiholdernorm{\p\upeta}{x}{0}{\delta_1}{\St}=\semiholdernorm{\p\left(\sum\limits_l\phi_l\upeta\right)}{x}{0}{\delta_1}{\St}\leq\sum\limits_l\semiholdernorm{\p(\phi_l\upeta)}{x}{0}{\delta_1}{\St}.
	\end{equation}	
	Notice that for each $x\in\St$, there are at most finitely many $l$'s such that $\phi_l(x)$ is non-zero.
	Hence, by definition of $G, F, H$,
	\begin{align}
	\semiholdernorm{\p\upeta}{x}{0}{\delta_1}{\St}&\lesssim\holdernorm{F}{x}{0}{\delta_1}{\St}+\holdernorm{G^{-1}}{x}{0}{\delta_1}{\St}\holdernorm{\upeta}{x}{0}{\delta_1}{\St}+\holdernorm{H}{x}{0}{\delta_1}{\St}+\holdernorm{\upeta}{x}{0}{\delta_1}{\St}\\
	\notag&\lesssim	\holdernorm{\pfour\vvariables,\vectvort,\DivGradEnt}{x}{0}{\delta_1}{\St}.
	\end{align}
	We now bound $\onenorm{\p\upeta}{x}{\infty}{\St}$. For any point $z\in\St$, there is a $y\in B(z,1)$ such that $\abs{\p\upeta(y)}\leq\onenorm{\p\upeta}{x}{2}{\St}$, thus
	\begin{align}
	\abs{\p\upeta(z)}\leq\abs{\p\upeta(y)}+2\semiholdernorm{\p\upeta}{x}{0}{\delta_1}{\St}\lesssim\onenorm{\p\upeta}{x}{2}{\St}+\semiholdernorm{\p\upeta}{x}{0}{\delta_1}{\St}.
	\end{align}
	Combining with Proposition \ref{EnergyandEllipticEstimates}, we have
	\begin{align}
	\holdernorm{\p\upeta}{x}{0}{\delta_1}{\St}\lesssim1+\holdernorm{\pfour\vvariables,\vectvort,\DivGradEnt}{x}{0}{\delta_1}{\St}.
	\end{align}
\end{proof}
	\subsection{Estimates for $\holdernorm{\modivort,\DivGradEnt}{x}{0}{\delta_1}{\St}$ via Transport Equations and the Proof of Theorem \ref{Schauder0}} \label{Schauder6}
	In this subsection, we first estimate the H\"older norm of the modified fluid variables. We then prove Theorem \ref{Schauder0}. We will use the following lemma, which is a standard estimate for transport equations with H\"older data and H\"older inhomogeneities.
\begin{lemma}\label{schauder3}
	Let $\phi$ be a scalar function. If $F\in C_x^{0,\alpha}(\Sigma_{\tau})$ with $\tau\in[0,t]$ and
	\begin{equation}
	\materialderivative\phi=F,
	\end{equation}	
	then 
	\begin{equation}
	\|\phi\|_{C_x^{0,\alpha}(\Sigma_{t})}\lesssim\|\phi\|_{C_x^{0,\alpha}(\Sigma_{0})}+\int^t_0\|F\|_{C_x^{0,\alpha}(\Sigma_{\tau})}\diff\tau.
	\end{equation}
\end{lemma}
	\begin{proof}

	Note that $\materialderivative^0=1$. Let $\gamma$ be the integral curve of $\materialderivative$ such that
	\begin{equation}
	\gamma^i(0,x)=x^i,
	\end{equation}
	\begin{equation}
	\gamma^0(t,x)=t,
	\end{equation}	
	and 
	\begin{equation}
	\p_t\gamma^i(t,x)=\materialderivative^i(t,\gamma).
	\end{equation}	
	Then 
	\begin{equation}
	\begin{aligned}
	|\gamma^i(t,x)-\gamma^i(t,y)|
	&\leq|x^i-y^i|+\int^t_0\materialderivative^i(\tau,\gamma(\tau,x))-\materialderivative^i(\tau,\gamma(\tau,y))\diff\tau\\
	&\leq|x^i-y^i|+\int^t_0\|\p \materialderivative^i\|_{L^\infty_x(\Sigma_\tau)}|\gamma(\tau,x)-\gamma(\tau,y)|\diff\tau.
	\end{aligned}
	\end{equation}	
	By Gr\"onwall's inequality, we have
	\begin{equation}\label{5}
	\frac{|\gamma(t,x)-\gamma(t,y)|}{|x-y|}\lesssim \exp\left(\int^t_0\|\p \materialderivative\|_{L^\infty_x(\Sigma_\tau)}\diff\tau\right).
	\end{equation}	
	Note that $\norm{\p\materialderivative}_{L^\infty_x(\Sigma_{\tau})}\approx\norm{\p\vvariables}_{L^\infty_x(\Sigma_{\tau})}$. Similarly, considering $\bar{\gamma}(\tau):=\gamma(t-\tau)$ from $\Sigma_t$ to $\Sigma_0$, we have
	\begin{equation}\label{6}
	\frac{|x-y|}{|\gamma(t,x)-\gamma(t,y)|}\lesssim \exp\left(\int^t_0\|\p \materialderivative\|_{L^\infty_x(\Sigma_\tau)}\diff\tau\right).
	\end{equation}
		
	Now we consider $\phi$:
	\begin{equation}
	\p_t(\phi\circ\gamma)=F\circ\gamma,
	\end{equation}	
	then
	\begin{equation}
	\phi\circ\gamma(t,x)-\phi\circ\gamma(t,y)=\phi\circ\gamma(0,x)-\phi\circ\gamma(0,y)+\int_0^tF(\tau,\gamma(\tau,x))-F(\tau,\gamma(\tau,y))\diff\tau.
	\end{equation}	
	By (\ref{5}), (\ref{6}) and bootstrap assumption in Section \ref{bootstrap}, we have
	\begin{equation}
	\begin{aligned}
	\frac{|F(\tau,\gamma(\tau,x))-F(\tau,\gamma(\tau,y))|}{|\gamma(t,x)-\gamma(t,y)|^\alpha}
	&=\frac{|F(\tau,\gamma(\tau,x))-F(\tau,\gamma(\tau,y))|}{|\gamma(\tau,x)-\gamma(\tau,y)|^\alpha}\frac{|\gamma(\tau,x)-\gamma(\tau,y)|^\alpha}{|\gamma(t,x)-\gamma(t,y)|^\alpha}\\
	&\lesssim\|F\|_{C_x^{0,\alpha}(\Sigma_{\tau})}.
	\end{aligned}
	\end{equation}
	\end{proof}
	\begin{proof}[Proof of Theorem \ref{Schauder0}]
	Now we consider equations (\ref{Bmodivort}) and (\ref{BDiventGrad}):
	\begin{subequations}
		\begin{align}
	&\materialderivative \modivort^\alpha=\quadsmoothfunction(\vvariables)[\pfour\vvariables,(\pfour\vort,\pfour\vgradientEnt,\pfour\vvariables)]+\quadsmoothfunction(\vgradientEnt)[\pfour\vvariables,\pfour\vvariables]+\linsmoothfunction(\vvariables,\vvort,\vgradientEnt)[\pfour\vvariables,\pfour\vgradientEnt],\\
	&\materialderivative\DivGradEnt=\linsmoothfunction(\vvariables,\vgradientEnt)[\pfour\vort]+\quadsmoothfunction(\vvariables)[\pfour\vvariables,(\pfour\vgradientEnt,\pfour\vvariables)]+\quadsmoothfunction(\vgradientEnt)[\pfour\vvariables,\pfour\vvariables]+\linsmoothfunction(\vvariables,\vvort,\vgradientEnt)[\pfour\vvariables].
	\end{align}
	\end{subequations}
	By the Lemma \ref{schauder3}, we have
	\begin{align}
	\holdernorm{\vectvort,\DivGradEnt}{x}{0}{\delta_1}{\Sigma_{t}}
	\lesssim1+\int_0^t\left(\holdernorm{\pfour\vvariables}{x}{0}{\delta_1}{\Sigma_\tau}+1\right)\holdernorm{\pfour\vvariables,\pfour\vvort,\pfour\vgradientEnt}{x}{0}{\delta_1}{\Sigma_\tau}\diff\tau.
	\end{align}
	By (\ref{Schauder4}), bootstrap assumptions and Gr\"onwall's inequality, we have
	\begin{align}
	\label{CDleq1}\holdernorm{\vectvort,\DivGradEnt}{x}{0}{\delta_1}{\St}\lesssim1.
	\end{align}	
	Integrating (\ref{Schauder4}) in time,  we have 
	\begin{align}
	\int_0^t\holdernorm{\pfour\vvort,\pfour\vgradientEnt}{x}{0}{\delta_1}{\St}^2\diff t\lesssim\int_0^t1+\holdernorm{\pfour\vvariables,\vectvort,\DivGradEnt}{x}{0}{\delta_1}{\St}^2\diff t.
	\end{align}
	If (\ref{preimprove}) holds, using (\ref{CDleq1}), the bootstrap assumption (\ref{waveba}) and the standard results in Littlewood-Paley (\ref{Littlewoodholder}), (\ref{improvedboot2}) is obtained by following estimate:
	\begin{align}
	\int_0^t\holdernorm{\pfour\vvort,\pfour\vgradientEnt}{x}{0}{\delta_1}{\St}^2\diff t\lesssim\int_0^t1+\holdernorm{\pfour\vvariables,\vectvort,\DivGradEnt}{x}{0}{\delta_1}{\St}^2\diff t\leq\Tstar^{2\delta}+\Tstar\lesssim\Tstar^{2\delta}.
	\end{align}
\end{proof}
\section{Reduction of Strichartz Estimates and the Rescaled Solution}\label{sectionreduction}
In this section, we state our main estimates as Theorem \ref{MainEstimates}, which are the improvement of the bootstrap assumption (\ref{waveba}). We then provide a series of analytic reductions from the Strichartz estimates of Theorem \ref{MainEstimates} to the decay estimates of Theorem \ref{Spatiallylocalizeddecay}. We are quite terse in this section since the full proofs of these reductions are lengthy and difficult, yet standard. We refer readers to \cite{3DCompressibleEuler,AGeoMetricApproach,ImprovedLocalwellPosedness} for the detailed proofs.
\begin{theorem}[Improvement of the Strichartz-type bootstrap assumptions for the wave variables]\label{MainEstimates}
	If $\delta>0$ is sufficiently small as in Subseciton \ref{ChoiceofParameters}, then under the initial data and bootstrap assumptions of Section \ref{sectionmainthm}, the following estimates holds with a number $8\delta_0<\delta_1<N-2$:
	\begin{align}
	\label{estimateswavevariables}\twonorms{\pfour\vvariables}{t}{2}{x}{\infty}{[0,\Tstar]\times\mathbb{R}^3}^2+\sum\limits_{\upnu\geq2}\upnu^{2\delta_1}\twonorms{\littlewood\pfour\vvariables}{t}{2}{x}{\infty}{[0,\Tstar]\times\mathbb{R}^3}^2\lesssim\Tstar^{2\delta}.
	\end{align}
\end{theorem}
	We first reduce the proof of Theorem \ref{MainEstimates} to the proof of Strichartz estimates on small time intervals.
\subsection{Partitioning of the Bootstrap Time Interval}\label{partitioning}
 Let $\lambda$ be a fixed large number and let $0<\varepsilon_0<\frac{N-2}{5}$ be a fixed number as mentioned in Subsection \ref{ChoiceofParameters}. By the bootstrap assumptions, we can\footnote{The existence of such partition easily follows from the bootstrap assumptions, see \cite[Remark 1.3]{ImprovedLocalwellPosedness}.} partition $[0,\Tstar]$ into disjoint union of sub-intervals $I_k:=[t_{k-1},t_k]$ of total number $\lesssim\lambda^{8\varepsilon_0}$ with the properties that $\abs{I_k}\leq\lambda^{-8\varepsilon_0}\Tstar$ and 
\begin{subequations}
	\begin{align}
	\twonorms{\pfour\vvariables}{t}{2}{x}{\infty}{I_k\times\mathbb{R}^3}^2+\sum\limits_{\upnu\geq2}\upnu^{2\delta_0}\twonorms{\littlewood\pfour\vvariables}{t}{2}{x}{\infty}{I_k\times\mathbb{R}^3}^2&\lesssim\lambda^{-8\varepsilon_0},\\
		\twonorms{\p\vvort,\p\vgradientEnt}{t}{2}{x}{\infty}{I_k\times\mathbb{R}^3}^2+\sum\limits_{\upnu\geq2}\upnu^{2\delta_0}\twonorms{\littlewood(\p\vvort,\p\vgradientEnt)}{t}{2}{x}{\infty}{I_k\times\mathbb{R}^3}^2&\lesssim\lambda^{-8\varepsilon_0}.
\end{align}
\end{subequations}

Now we reduce Theorem \ref{MainEstimates} to a frequency localized estimate.
\begin{theorem}[Frequency localized Strichartz estimate]\label{FrequencyStrichartz}
	Let $\varphi$ be a solution of
	\begin{align}
	\boxg\varphi=0
	\end{align}
	on the time interval $I_k$. Then for any $q>2$ sufficiently close to $2$ and any $\tau\in[t_{k},t_{k+1}]$, under the bootstrap assumptions, we have the following estimate:
	\begin{align}\label{Stri2}
	\twonorms{P_\lambda\pfour\varphi}{t}{q}{x}{\infty}{[\tau,t_{k+1}]\times\mathbb{R}^3}\lesssim\lambda^{\frac{3}{2}-\frac{1}{q}}\norm{\pfour\varphi}_{L_x^2(\Sigma_{\tau})}.
	\end{align}
\end{theorem}
\subsubsection{Dicussion of the reduction to Theorem \ref{FrequencyStrichartz} from Theorem \ref{MainEstimates}}
The proof reducing Theorem \ref{MainEstimates} to Theorem \ref{FrequencyStrichartz} is exactly the same as the proof of \cite[Theorem 7.2]{3DCompressibleEuler}. The reason is that the proof relies only on: \textbf{1)} Duhamel's principle; \textbf{2)} top order energy estimates (\ref{energyestimates}) which is the same as in \cite{3DCompressibleEuler}, and \textbf{3)} Littlewood-Paley estimates for the inhomogeneous terms in a frequency-projected version of the wave equations, and the wave equations in this paper have an identical schematic form to the ones in \cite{3DCompressibleEuler}.

\subsection{Rescaled Quantities and Rescaled Relativistic Euler Equations}\label{SelfRescaling}
In this subsection, in order to do further reductions, we consider the following coordinate change $(t,x)\mapsto(\lambda(t-t_k),\lambda x)$. Let
\begin{align}\label{Trescale}
\Trescale:=\lambda(t_{k+1}-t_k).
\end{align}
Note that by construction, 
\begin{align}
0\leq\Trescale\leq\lambda^{1-8\varepsilon_0}\Tstar.
\end{align}
\begin{definition}[Rescaled quantities]\label{rescaledquantities}
	First we define the following variables:
\begin{subequations}
		\begin{align}
	\vvariables_{(\lambda)}(t,x):=&\vvariables(t_k+\lambda^{-1}t,\lambda^{-1}x),\\
	\vvort_{(\lambda)}(t,x):=&\vvort(t_k+\lambda^{-1}t,\lambda^{-1}x),\\
	\vgradientEnt^\kappa_{(\lambda)}(t,x):=&\vgradientEnt^\kappa(t_k+\lambda^{-1}t,\lambda^{-1}x),\\
	\temp_{(\lambda)}:=&\temp(t_k+\lambda^{-1}t,\lambda^{-1}x),\\
	\speed_{(\lambda)}:=&\speed(t_k+\lambda^{-1}t,\lambda^{-1}x),\\
		\modivort_{(\lambda)}^\alpha:=&\text{vort}^\alpha((\vort_\flat)_{(\lambda)})+\speed_{(\lambda)}^{-2}\antisymmetic^{\alpha\beta\gamma\delta}((v_\flat)_{(\lambda)})_\beta(\p_\gamma\lnenth_{(\lambda)})(\vort_\flat)_{(\lambda)\delta}+(\temp_{(\lambda)}-(\temp_{;\lnenth})_{(\lambda)})(\gradEnt^\sharp)_{(\lambda)}^\alpha(\p_\kappa (v_\flat)_{(\lambda)}^\kappa)\\
	\notag&+(\temp_{(\lambda)}-(\temp_{;\lnenth})_{(\lambda)})(v_\flat)_{(\lambda)}^\alpha((\gradEnt^\sharp)_{(\lambda)}^\kappa\p_\kappa\lnenth_{(\lambda)})-(\temp_{(\lambda)}-(\temp_{;\lnenth})_{(\lambda)})(\gradEnt^\sharp)_{(\lambda)}^\kappa\left((\minkowski^{-1})^{\alpha\lambda}\p_\lambda (v_\flat)_{_{(\lambda)}\kappa}\right),\\
	\DivGradEnt_{(\lambda)}:=&\frac{1}{n}(\p_\kappa(\gradEnt^\sharp)_{(\lambda)}^\kappa)+\frac{1}{n}((\gradEnt^\sharp)_{(\lambda)}^\kappa\p_\kappa\lnenth_{(\lambda)})-\frac{1}{n}\speed^{-2}((\gradEnt^\sharp)_{(\lambda)}^\kappa\p_\kappa\lnenth_{(\lambda)}).
	\end{align}
\end{subequations}

	Then we define the following rescaled tensor fields:
\begin{subequations}
		\begin{align}
	\label{frequencyrescaledg}(\gfour_{(\lambda)})_{\alpha\beta}(t,x):=&\gfour_{\alpha\beta}(\vvariables_{(\lambda)}(t,x)),\\
	\label{frequencyrescaledgt}(g_{(\lambda)})_{\alpha\beta}(t,x):=&g_{\alpha\beta}(\vvariables_{(\lambda)}(t,x)),\\
	\materialderivative_{(\lambda)}^\alpha(t,x):=&\materialderivative^\alpha(\vvariables_{(\lambda)}(t,x)).
	\end{align}
\end{subequations}
\end{definition}

\begin{remark}\label{newsigma0}
	We note that after rescaling, the new initial constant-time hypersurface $\Sigma_{0}$ corresponds to the constant-time hypersurface that was denoted by $\Sigma_{t_k}$ for some $k$ in Section \ref{sectionintro}-\ref{section4}.
\end{remark}
The following proposition provides the equations satisfied by the rescaled quantities. We omit the straightforward proof.
\begin{proposition}[The rescaled geometric wave-transport formulation of the relativistic Euler equations]
	For $\variables\in\{v^0,v^1,v^2,v^3,h,s\}$, after rescaling, we have the following equations inherited from Prop.\ref{WaveEqunderoringinalacoustical}:\\
	\textbf{Wave equations}	
	\begin{align}
	\square_{\gfour_{(\lambda)}}\variables_{(\lambda)}=\lambda^{-1}\linsmoothfunction(\vvariables_{(\lambda)})[\vectvort_{(\lambda)},\DivGradEnt_{(\lambda)}]+\quadsmoothfunction(\vvariables_{(\lambda)})[\pfour\vvariables_{(\lambda)},\pfour\vvariables_{(\lambda)}].
	\end{align}
	\textbf{Transport equations}
	\begin{subequations}
		\begin{align}
		\label{Bvort2}&\materialderivative_{(\lambda)}\vort_{(\lambda)}^\alpha=\linsmoothfunction(\vvariables_{(\lambda)},\vvort_{(\lambda)},\vgradientEnt_{(\lambda)})[\pfour\vvariables_{(\lambda)}],\\ &\materialderivative_{(\lambda)}(\gradEnt^\sharp)_{(\lambda)}^\alpha=\linsmoothfunction(\vvariables_{(\lambda)},\vgradientEnt_{(\lambda)})[\pfour\vvariables_{(\lambda)}].
		\end{align}
	\end{subequations}
	\textbf{Transport-Div-Curl system}
\begin{subequations}\label{DivCurlSystem2}
		\begin{align}
	\label{Bmodivort2}\materialderivative_{(\lambda)} \modivort_{(\lambda)}^\alpha=&\quadsmoothfunction(\vvariables_{(\lambda)})[\pfour\vvariables_{(\lambda)},(\pfour\vort_{(\lambda)},\pfour\vgradientEnt_{(\lambda)},\pfour\vvariables_{(\lambda)})]+\linsmoothfunction(\vvariables_{(\lambda)},\vvort_{(\lambda)},\vgradientEnt_{(\lambda)})[\pfour\vvariables_{(\lambda)},\pfour\vgradientEnt_{(\lambda)}],\\
	\label{BDiventGrad2}\materialderivative_{(\lambda)}\DivGradEnt_{(\lambda)}=&\linsmoothfunction(\vgradientEnt_{(\lambda)},\vvariables_{(\lambda)})[\pfour\vort_{(\lambda)}]\\
	\notag&+\quadsmoothfunction(\vvariables_{(\lambda)})[\pfour\vvariables_{(\lambda)},(\pfour\vgradientEnt_{(\lambda)},\pfour\vvariables_{(\lambda)})]+\linsmoothfunction(\vvariables_{(\lambda)},\vvort_{(\lambda)},\vgradientEnt_{(\lambda)})[\pfour\vvariables_{(\lambda)},\pfour\vgradientEnt_{(\lambda)}],\\
	\textnormal{vort}^\alpha(S_{(\lambda)}^\sharp)=&0,\\
	\p_\alpha\vort_{(\lambda)}^\alpha=&\linsmoothfunction(\vort_{(\lambda)})[\pfour\vvariables_{(\lambda)}].
	\end{align}
\end{subequations}
\end{proposition}
\begin{remark}
	For notation convenience,  in the remainder of the article, we drop the sub and super scripts $(\lambda)$ except for the rescaled time $\Trescale$.
\end{remark}
\subsubsection{Consequences of the bootstrap assumptions}
After rescaling in Subsection \ref{SelfRescaling}, assuming bootstrap assumptions (\ref{waveba})-(\ref{transportba}), as in \cite[Section 10.2.1]{3DCompressibleEuler}, by standard computation based on Littlewood-Paley calculus, we have the following consequences of the bootstrap assumptions:
\\

\noindent\textbf{Estimates by using bootstrap assumptions of variables}
\begin{align}
\label{bt1}\twonorms{\pfour\vvariables,\pfour\vVort,\pfour\vgradientEnt,\vectvort,\DivGradEnt}{t}{2}{x}{\infty}{\region}+\lambda^{\delta_0}\sqrt{\sum\limits_{\upnu>2}\upnu^{2\delta_0}\twonorms{\littlewood\left(\gensmoothfunction(\vvariables,\vVort,\vgradientEnt)(\pfour\vvariables,\pfour\vVort,\pfour\vgradientEnt,\vectvort,\DivGradEnt)\right)}{t}{2}{x}{\infty}{\region}^2}\lesssim\lambda^{-1/2-4\varepsilon_0}
.\end{align}
\subsection{Further Reduction of the Strichartz Estimates}
By the rescaling in the Section \ref{SelfRescaling} and direct computation, to prove Theorem \ref{FrequencyStrichartz}, it is equivalent to show the following Strichartz estimate on $[0,\Trescale]$ with respect to LP projection on the frequency domain $\{1/2\leq\abs{\xi}\leq 2\}$.
\begin{theorem}\label{RescaledStrichartz}
Under the bootstrap assumptions. For any solution $\varphi$ of $\boxg\varphi=0$ on the slab $[0,\Trescale]\times\mathbb{R}^3$, the following estimates holds:
\begin{align}
\twonorms{P_1\pfour\varphi}{t}{q}{x}{\infty}{[0,\Trescale]\times\mathbb{R}^3}\lesssim\norm{\pfour\varphi}_{L_x^2(\Sigma_{0})},
1\end{align}
where $q>2$ is sufficiently close to 2 and $\gfour$ is the rescaled metric $\gfour_{(\lambda)}$ defined in (\ref{frequencyrescaledg}).
\end{theorem}
The proof of Theorem \ref{RescaledStrichartz} crucially relies on the following decay estimate. 
\begin{theorem}[Decay estimate]\label{DecayEstimates}
	There exists a large number $\Lambda$ such that for any $\lambda\geq\Lambda$ and any solution $\varphi$ of the equation $\boxg\varphi=0$ on $[0,\Trescale]\times\mathbb{R}^3$, there is a function $d(t)$ satisfying
	\begin{align}
	\norm{d}_{L_t^{\frac{q}{2}}([0,\Trescale])}\lesssim1
	\end{align}
	for $q>2$ sufficiently close to 2 such that for any $t\in[0,\Trescale]$, the following decay estimate holds\footnote{$\Timelike_{\flat}=-\diff t$, see Definition \ref{timelike} for the definition of $\Timelike$.}:
	\begin{align}
	\norm{P_1\Timelike\varphi}_{L_x^\infty(\St)}\lesssim\left(\frac{1}{(1+t)^{\frac{2}{q}}}+d(t)\right)\left(\sum\limits_{m=0}^3\norm{\p^m\varphi}_{L_x^1(\Sigma_{0})}+\sum\limits_{m=0}^2\norm{\p^m\p_t\varphi}_{L_x^1(\Sigma_{0})}\right).
	\end{align}
	
\end{theorem}

The proof of Theorem \ref{RescaledStrichartz} using Theorem \ref{DecayEstimates} is based on a $\mathcal{T}\mathcal{T}^*$ argument\footnote{This argument comes from functional analysis, which does not require the structure of the relativistic Euler equations.} (see \cite[Section 8.6]{ImprovedLocalwellPosedness} and \cite[Section 8.30]{RoughsolutionstotheEinsteinvacuumequations}).

Theorem \ref{DecayEstimates} can be further reduced to the following spatially localized version.
\begin{theorem}[Spatially localized version of decay estimate]\label{Spatiallylocalizeddecay}
	There exists a large number $\Lambda$ such that for any $\lambda\geq\Lambda$ and any solution $\varphi$ of the equation $\boxg\varphi=0$ on $[0,\Trescale]\times\mathbb{R}^3$ with $\varphi(1,x)$ supported in the Euclidean ball $B_{R}$, where radius $R$ is a fixed radius\footnote{The radius $R$ will be used in the following sections as well with the same definition. The existence of such an $R$ is guaranteed by properties of $g$, namely, (\ref{gab2}), that ensures $g$ is comparable to the Euclidean metric on $\St$ under the bootstrap assumptions.} such that
	\begin{align}
	B_R(p)&\subset B_{\frac{1}{2}}(p,g_{(\lambda)}),&
	&\forall p\in\Sigma_t, 0\leq t\leq\Trescale,
	\end{align}
	where $B_{\rho}(p,g_{(\lambda)})$ is the geodesic ball\footnote{The notation $B_{\rho}(p,g_{(\lambda)})$ will be used in the remainder of the article. This is consistent with the notation that is used in \cite{3DCompressibleEuler} and \cite{AGeoMetricApproach}} centered at $p$ with radius $\rho$ and $g_{(\lambda)}$ is the rescaled induced metric of $\gfour$ on $\St$ (defined in (\ref{frequencyrescaledgt})), there is a function $d(t)$ satisfying
	\begin{align}
	\norm{d}_{L_t^{\frac{q}{2}}([0,\Trescale])}\lesssim1
	\end{align}
	for $q>2$ sufficiently close to 2 such that for any $t\in[0,\Trescale]$, there holds
	\begin{align}\label{Decay}
	\norm{P_1\Timelike\varphi}_{L_x^\infty(\St)}\lesssim\left(\frac{1}{(1+\abs{t-1})^{\frac{2}{q}}}+d(t)\right)\left(\norm{\pfour\varphi}_{L_x^2(\Sigma_1)}+\norm{\varphi}_{L_x^2(\Sigma_1)}\right).
	\end{align}
\end{theorem}

	The proof of Theorem \ref{DecayEstimates} using Theorem \ref{Spatiallylocalizeddecay} can be done via an approach involving the Bernstein inequalities of LP projection, partition of unity\footnote{We take a sequence of Euclidean balls $\{B_I\}$ of radius $R$ such that their union covers $\mathbb{R}^3$. For each ball $B_R$, it is centered at $\upgamma_\tip(1)$ (defined in  Section \ref{solutionofu}) for some $\tip$.} of $\varphi$ and Sobolev embedding $W^{2,1}\hookrightarrow L^2$. We 
	refer readers to \cite[Section 8.5]{ImprovedLocalwellPosedness} for detailed proof.\\
	
To summarize, in this section we have reduced Theorem \ref{MainEstimates} to Theorem \ref{Spatiallylocalizeddecay}. To further reduce the estimates, we need to introduce the geometric setup. We will discuss the proof of Theorem \ref{Spatiallylocalizeddecay} in Section \ref{subsectionconformalenergy}.
\section{Geometric Setup and Conformal Energy}\label{sectiongeometrysetup}
In this section, we first construct acoustic geometry. It is deeply coupled with the relativistic Euler equations(via acoustic metric $\gfour$) and is crucial in our analysis. Then we provide the conformal energy and its estimates in Section \ref{subsectionconformalenergy}
\begin{definition}[Christoffel symbols]\label{Christoffelsymbols}
	We define the Christoffel symbols $\Chfour_{\alpha\kappa\lambda}$ and $\Chfour^\beta_{\kappa\lambda}$ with the rescaled metric $\gfour$ to be as follows:
	\begin{align}
	\Chfour_{\alpha\kappa\lambda}&:=\frac{1}{2}\left(\p_\kappa\gfour_{\alpha\lambda}+\p_\lambda\gfour_{\alpha\kappa}-\p_\alpha\gfour_{\kappa\lambda}\right),\\
	\Chfour^\beta_{\kappa\lambda}&:=\gfour^{\alpha\beta}\Chfour_{\alpha\kappa\lambda}.
	\end{align}
\end{definition}
\subsection{Construction of the Acoustical Function}\label{opticalfunction}

The goal of this subsection is to construct the geometry based on a solution $u$ to the acoustical eikonal equation\footnote{For more details of the geometric construction of $u$, we refer reader to \cite[Chapter 9 and Chapter 14]{GlobalStabiliityOfMinkowski}.}:
\begin{align}\label{eikonalequation}
\gfour^{\alpha\beta}\p_\alpha u\p_\beta u=0.
\end{align}

\subsubsection{Point $\tip$ and integral curve $\upgamma_\tip(t)$}\label{pointtip}
Let $\tip\in\Sigma_{0}$ be an arbitrarily fixed point\footnote{Note that in the original spacetime $[0,\Tstar]\times\mathbb{R}^3$, $\tip$ is a point at $\Sigma_{t_k}$ for some $k$.} in the rescaled space-time $[0,\Trescale]\times\mathbb{R}^3$, where $\Trescale$ is defined in (\ref{Trescale}). We let $\upgamma_\tip(t)$ denote the integral curve of the future-directed vectorfield $\Timelike$ emanating from the point $\tip$. We say $\upgamma_\tip(t)$ is the cone-tip axis. Specifically, the point $\tip$ depends on the partition of unity of $\Sigma_1$ used in the proof of Theorem \ref{DecayEstimates} using Theorem \ref{Spatiallylocalizeddecay}. More specifically, $\tip$ is going to be the tip of the cone (constructed in Subsection \ref{solutionofu}) such that $\upgamma_\tip(1)$ is the center of the Euclidean ball $B_R$ (as in Theorem \ref{Spatiallylocalizeddecay}). We note that the estimates, constants and parameters in Section \ref{sectiongeometrysetup}-\ref{connectioncoefficientsandpdes} are independent of $\tip$.
\subsubsection{The interior and exterior solution $u$}\label{solutionofu}

\paragraph{The interior solution $u$.}
We let $\{\left.\Lunit_\sangle\right|_\tip\}_{\sangle\in\Stwo}$ be the family of null vectors (parameterized by $\Stwo$, where the parameterization will be uniquely determined in the paragraph of the exterior solution $u$) in the tangent space $T_{\tip}\region$ and $\langle\left.\Lunit_\sangle\right|_\tip,\Timelike\rangle=-1$. To propagate $\Lunit_\sangle$ along the cone-tip axis $\upgamma_\tip(t)$, for any $\textbf{p}\in\upgamma_\tip(t)$ (as defined in Subsection \ref{pointtip}) and $\sangle\in\Stwo$, we define $\left.\Lunit_\sangle\right|_{\textbf{p}}$ by solving the parallel transport equation $\Dfour_\Timelike\Lunit_{\sangle}=0$. We note that for any $\textbf{p}\in\upgamma_\tip(t)$, $\langle\left.\Lunit_\sangle\right|_{\textbf{p}},\Timelike\rangle=-1$ since $\Timelike$ is geodesic. Then, for each $u\in[0,\Trescale]$ and $\sangle\in\Stwo$, there exists a unique null geodesic $\Upsilon_{u,\sangle}(t)$, where $t\in[u,\Trescale]$, emanating from $\textbf{p}=\upgamma_\tip(u)$ with $\left.\frac{\diff}{\diff t}\Upsilon_{u,\sangle}\right|_{t=u}=\Lunit_\sangle$ and $\Upsilon_{u,\sangle}^0(t)=t$. Specifically, $\Upsilon_{u,\sangle}(t)$ is constructed by solving the following ``geodesic" ODE system\footnote{See \cite[Section 9.4.1]{3DCompressibleEuler} for detailed explanation of this ODE system where $\Timelike$ coincide with $\materialderivative$.}:
\begin{subequations}
	\begin{align}
\label{geodesiceq}\ddot{\Upsilon}_{u,\sangle}^\alpha(t)=&-\left.\Chfour^\alpha_{\kappa\lambda}\right|_{\Upsilon_{u,\sangle}(t)}\dot{\Upsilon}_{u,\sangle}^\kappa(t)\dot{\Upsilon}_{u,\sangle}^\lambda(t)\\
\notag&+\frac{1}{2}\left.[\Lie_\Timelike\gfour]_{\kappa\lambda}\right|_{\Upsilon_{u,\sangle}(t)}(\dot{\Upsilon}_{u,\sangle}^\kappa(t)-\left.\Timelike^\kappa\right|_{\Upsilon_{u,\sangle}(t)})(\dot{\Upsilon}_{u,\sangle}^\lambda(t)-\left.\Timelike^\lambda\right|_{\Upsilon_{u,\sangle}(t)})\dot{\Upsilon}_{u,\sangle}^\alpha(t),\\
\Upsilon_{u,\sangle}^\alpha(u)=&\upgamma_\tip^\alpha(u), \ \ \ 
\dot{\Upsilon}_{u,\sangle}^\alpha(u)=\Lunit^\alpha_{\sangle},
\end{align} 
\end{subequations}
where $\ddot{\Upsilon}_{u,\sangle}^\alpha:=\frac{\diff^2}{\diff t^2}\Upsilon_{u,\sangle}^\alpha$, $\dot{\Upsilon}_{u,\sangle}^\alpha:=\frac{\diff}{\diff t}\Upsilon_{u,\sangle}^\alpha$, $\Chfour$ is the Christoffel symbol of $\gfour$ and $\Lie_\Timelike\gfour$ is the Lie derivative of $\gfour$ with respect to $\Timelike$. The curve $t\rightarrow\Upsilon_{u,\sangle}(t)$ is a non-affinely parameterized null geodesic such that the vectorfield $\Lunit_{u,\sangle}^\alpha:=\frac{\diff}{\diff t}\Upsilon_{u,\sangle}^\alpha$ is null and normalized by $\Lunit_{u,\sangle}^0=1$. In fact, this vectorfield coincides (in the interior region) with the vectorfield $\Lunit$ defined in (\ref{DefinitionofL}) below. We define the truncated null cone $\coneu$ to be $\coneu:=\bigcup\limits_{\sangle\in\Stwo,t\in[u,\Trescale]}\Upsilon_{u,\sangle}(t)$. We define the acoustical function $u$ by asserting that its level sets $\{u=u^\prime\}$ are $\mathcal{C}_{u^\prime}$. For $u\in[0,\Trescale]$ and $t\in[u,\Trescale]$, we let $\stu:=\coneu\bigcap\St$. For $u\neq t$, $\stu$ is a smooth surface diffeomorphic to $\Stwo$. We define $\intregion:=\bigcup\limits_{t\in[0,\Trescale], 0\leq u\leq t}\stu$. For each $\sangle\in\Stwo$ and $u\in[0,\Trescale]$, we define the angular coordinate functions $\{\sangle^A\}_{A=1,2}$ to be constant along each fixed null geodesics $\Upsilon_{u,\sangle}(t)$ and to coincide with standard angular coordinates on $\Stwo$ at the tip $\textbf{p}$, which corresponds to $t=u$.

\paragraph{The exterior solution $u$.}

Now we extend the foliation of space-time by null hypersurfaces to a neighborhood of $\bigcup\limits_{t\in[0,\Trescale], 0\leq u\leq t}\stu$ in $[0,\Trescale]\times\mathbb{R}^3$. Let $w_*=\frac{4}{5}\Trescale$. Using the arguments in \cite{NashMoser1}, we can guarantee\footnote{The existence of $w$-foliation for $w\in[0,\varepsilon]$ with a small $\varepsilon>0$ can be proven by Nash-Moser implicit function theorem and such foliation can be extended to $w_*$ by an argument of continuity (see \cite{NashMoser1}).} that there is a neighborhood $\neighborhood\in\Sigma_{0}$ contained in the geodesic ball $B_{\Trescale}(\tip,g_{(\lambda)})$, where $g_{(\lambda)}$ is the frequency rescaled induced metric of $\gfour$ on $\Sigma_{0}$ defined in (\ref{frequencyrescaledgt}), such that $\neighborhood$ can be foliated by the level sets $S_{0,-w}$ of a positive function $w$ taking all values in $[0,w_*]$ with $w(\tip)=0$ and each $S_{0,-w}$ for positive $w$ is diffeomorphic to $\mathbb{S}^2$. Fix a diffeomorphism $\sangle\rightarrow\Phi_\sangle(w_*)$ from $\Stwo$ to $S_{0,-w_*}$. Then, for each point $\textbf{p}=\Phi_\sangle(w_*)$, denoting the lapse $\lapse:=\left((g^{-1})^{cd}\p_cw\p_dw\right)^{-\frac{1}{2}}$ with $\lapse(\tip)=1$ and $\lapse\approx1$; see Proposition \ref{InitialCT1} (we note that the proof of Proposition \ref{InitialCT1} is independent of the construction of $u$), there is a unique integral curve\footnote{The existence and uniqueness of such integral curve are ensured by the estimates on $\Sigma_{0}$ in Proposition \ref{InitialCT1}. We refer readers to \cite[Subsection 9.4.2]{3DCompressibleEuler} for details.} 
$w\rightarrow\Phi_\sangle(w)$ of the vectorfield $\lapse^2(g^{-1})^{ic}\p_cw$ in $\Sigma_{0}$ with $\Phi_\sangle(w_*)=\textbf{p}$ and such that this integral curve can be extended to $\tip$, i.e. $\Phi_{\sangle}(0)=\tip$ (the extendibility of $\Phi_\sangle$ to $\tip$ follows by estimates (\ref{pNatSigma0}) and the fundamental theorem of Calculus). Denoting $\dot{\Phi}_\sangle:=\frac{\diff}{\diff w}\Phi_{\sangle}$, we then define $\spherenormal_\sangle|_\tip:=\dot{\Phi}_\sangle(0)$ and $\left.\Lunit_\sangle\right|_\tip:=\spherenormal_\sangle|_\tip+\Timelike|_\tip$. Note that the diffeomorphism $\sangle\rightarrow\left.\Lunit_\sangle\right|_\tip$ is uniquely determined by the vector field $\lapse^2(g^{-1})^{ic}\p_cw$, and it is precisely this diffeomorphism that appears in our construction of the interior solution described above, since $\langle\spherenormal_\sangle|_\tip+\Timelike|_\tip,\Timelike|_\tip\rangle=-1$ and $\langle\spherenormal_\sangle|_\tip+\Timelike|_\tip,\spherenormal_\sangle|_\tip+\Timelike|_\tip\rangle=0$ (because of $\lapse(\tip)=1$).
By such construction, for each $w\in(0,w_*]$ and any $\textbf{p}\in S_{0,-w}$, there exists a $\sangle\in\Stwo$ such that $\textbf{p}=\Phi_\sangle(w)$ and such that the outward unit normal (in $\Sigma_{0}$) to $S_{0,-w}$ at $\textbf{p}$ is $\spherenormal_{w,\sangle}:=\dot{\Phi}_\sangle(w)$. We set $\Lunit_{w,\sangle}:=\spherenormal_{w,\sangle}+\Timelike|_{\Phi_\sangle(w)}$, which is a null vector in $T_\textbf{p}\region$. Then, with $u=-w$, there is a unique (non-affinely parameterized) null geodesic $\Upsilon_{u,\sangle}$ emanating from $\textbf{p}$ and solving (\ref{geodesiceq}) with the initial condition $\left.\frac{\diff}{\diff t}\Upsilon_{u,\sangle}\right|_{t=0}=\Lunit_{-u,\sangle}$ and $\Upsilon_{u,\sangle}(0)=\textbf{p}$. We define the null cone $\coneu$ to be $\coneu:=\bigcup\limits_{\sangle\in\Stwo,t\in[0,\Trescale]}\Upsilon_{u,\sangle}(t)$. We define the acoustical function $u$ by asserting that its level sets $\{u=u^\prime\}$ are $\mathcal{C}_{u^\prime}$. Let $\stu:=\coneu\bigcap\St$. We note that $\coneu$'s are the outgoing truncated null cones, that is, $\coneu:=\bigcup\limits_{t\in[0,\Trescale]}\stu$. We define $\extregion:=\bigcup\limits_{t\in[0,\Trescale], u\in[-w_*,0)}\stu$.
For each $\sangle\in\Stwo$ and $u\in[-w_*,0)$, we define the angular coordinate functions $\{\sangle^A\}_{A=1,2}$ to be constant along null geodesics $\Upsilon_{u,\sangle}(t)$ and to coincide with the angular coordinates $\{\sangle^A\}_{A=1,2}$ on $\Sigma_{0}$ provided by the above construction; note that on $\Sigma_{0}\backslash\{\tip\}$, by construction, the angular coordinate functions $\{\sangle^A\}_{A=1,2}$ are constant along the integral curves of the vectorfield $\lapse^2(g^{-1})^{ic}\p_cw$.

We define the space-time region $\region$ as follows:
\begin{align}
\region=\intregion\bigcup\extregion.
\end{align}
By the constructions above, we have constructed the geometric coordinates $(t,u,\sangle^A)$ in $\region$.

See Figure \ref{figuregeometry} for the figure of the geometry.
\subsection{Geometric Quantities}\label{geometricquantities}

\begin{definition}[The radial variable]\label{D:rescaledquantities}
	Recall that
	\begin{align}
	\label{lengtht}0\leq\Trescale\leq\lambda^{1-8\varepsilon_0}\Tstar.
	\end{align}
	
	We define the geometric radial variable $\rgeo$ as follows:
	\begin{align}
	\label{DEFE:rgeo}\rgeo:=t-u.
	\end{align}
	
	Since we have that $t\in[0,\Trescale]$ and $u\in[-w_*,t]$ in $\region$, where $w_*:=\frac{4}{5}\Trescale$, we have
	\begin{align}\label{estimatesradial}
	0&\leq\rgeo<2\Trescale=2\lambda^{1-8\varepsilon_0}\Tstar,&
	-\frac{4}{5}\lambda^{1-8\varepsilon_0}T_*&\leq u\leq\lambda^{1-8\varepsilon_0}\Tstar.
	\end{align}
	We will silently use estimates (\ref{estimatesradial}) throughout the paper.
\end{definition}
\begin{definition}[Acoustic vectorfields and scalar functions]\label{D:Nullframe}
	
	We define the null vector field
	\begin{align}
	\Lgeo&:=-\gfour^{\alpha\beta}\partial_\beta u\partial_\alpha.
	\end{align}	
	Note that by (\ref{eikonalequation}), we have $\Dfour_{\Lgeo}\Lgeo=0$.	
	
	We define the null lapse $\nulllapse$ as follows:
	\begin{align}
	\nulllapse &:=(\sqrt{\gt^{ij}\partial_i u\partial_j u})^{-1}.
	\end{align}
	
	We define the vector field $\spherenormal$ as follows:
	\begin{align}
	\spherenormal&:=-\nulllapse\gt^{ij}\partial_iu\partial_j.
	\end{align}
	Note that $\left.\spherenormal\right|_{\Sigma_{0}}=\spherenormal_\sangle$, where $\spherenormal_\sangle$ is define in Section \ref{opticalfunction}.
	
	We define the principal null vector fields $\Lunit$ and $\uLunit$ as follows:
	\begin{align}\label{DefinitionofL}
	\Lunit&:=\Timelike+\spherenormal,
	&
	\uLunit&:=\Timelike-\spherenormal.
	\end{align}	
	Notice that by Definition \ref{D:inducegonconstanttime} and (\ref{eikonalequation}), we have identity $\Timelike u=\abs{\nabla u}_g=\frac{1}{\nulllapse}$. Then we have
	\begin{align}
	\Lunit=\nulllapse\Lgeo.
	\end{align}
	
	We have following basic properties:
	\begin{align}
	\langle\Timelike,\Timelike\rangle&=-1,
	&
	\langle\spherenormal,\spherenormal\rangle&=1,\\
	\langle\Timelike,\spherenormal\rangle&=0,
	&
	\langle\Lunit,\Lunit\rangle&=0,\\
	\langle\Lunit,\uLunit\rangle&=-2,
	&
	\langle\uLunit,\uLunit\rangle&=0.
	\end{align}
	
	We define $\gsphere$ to be as follows:
	\begin{align}\label{decompositionofmetric}
	(\gsphere^{-1})^{\alpha\beta}:=\gfour^{\alpha\beta}+\frac{1}{2}\Lunit^\alpha\uLunit^\beta+\frac{1}{2}\uLunit^\alpha\Lunit^\beta.
	\end{align}
	It is easy to check that $\gsphere$ is an induced metric of $\gfour$ on $\stu$. We denote a pair of unit orthogonal spherical vectorfields\footnotemark
	\footnotetext{In the rest of the paper, we automatically sum over $A$ if there are two $A$'s in the expression.} on $\stu$ by $\{e_A\}_{A=1,2}$ such that $(\gsphere^{-1})^{\alpha\beta}=\sum\limits_{A=1,2}e_A^\alpha e_A^\beta$. We call $\Lunit,\uLunit,e_1,e_2$ a null frame for the geometry. 
	 
	We denote the Levi-Civita connection on $\stu$ with respect to $\gsphere$ by $\angnabla$.
	
	 As we discussed in Section \ref{opticalfunction}, the angular coordinate functions $\{\sangle^A\}_{A=1,2}$ satisfy the equation $\Lunit(\sangle^A)=0$ along null cones. By this construction of angular coordinates, with respect to geometric coordinates $(t,u,\sangle^A)$, we have
	\begin{align}\label{framework}
	\Lunit&=\frac{\p}{\p t},&  \spherenormal&=-\nulllapse^{-1}\frac{\p}{\p u}+Y^A\deriasphere.
	\end{align}
	By construction in Section \ref{solutionofu}, $Y^A=0$ on $\Sigma_{0}$. 
\end{definition}

\label{figuregeometry}\begin{figure}
	\centering
	\includegraphics[width=12cm]{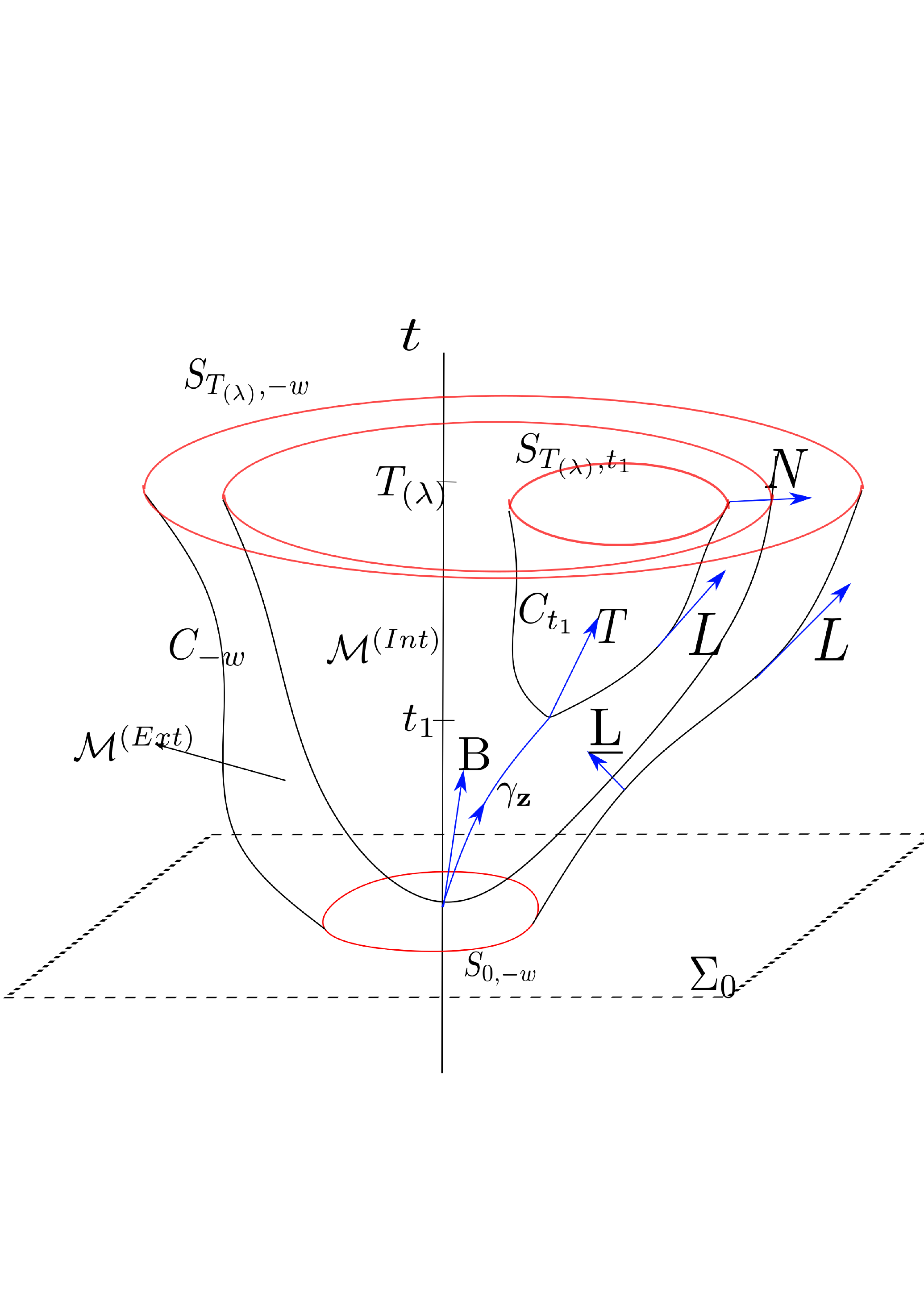}
	\caption{The geometric construction out of acoustical function}
\end{figure}

\begin{definition}[$\stu$-tangent tensorfields]\label{D:stutensorfields}
	
	We define the $\gfour$-orthogonal projection $\sphereproject$ onto $\stu$, where $\delta^\alpha_\beta$ is the Kronecker delta, as follows:
	\begin{align}
	\sphereproject^\alpha_\beta:=\delta^\alpha_\beta+\frac{1}{2}\Lunit^\alpha\uLunit_\beta+\frac{1}{2}\uLunit^\alpha\Lunit_\beta,
	\end{align}
	
	We use the notation $\sgabs{\upxi}$ to denote the norm of the $\stu$-tangent tensorfield $\upxi$ with respect to $\gsphere$, that is,
	\begin{align}
	\sgabs{\upxi}^2:=\gsphere(\upxi,\upxi).
	\end{align}
	
	We use $\gtr\upxi$ to denote the trace of a $\binom{0}{2}$ $\stu$-tangent tensorfield $\upxi$ with respect to $\gsphere$:
	\begin{align}
	\gtr\upxi:=\gsphere^{AB}\upxi_{AB}.
	\end{align}
	
	We define $\hat{\upxi}$ to be the trace-free part of the $\binom{0}{2}$ $\stu$-tangent tensorfield $\upxi$:
	\begin{align}
	\hat{\upxi}:=\upxi-\frac{1}{2}(\gtr\upxi)\gsphere.
	\end{align}
\end{definition}

\subsection{Conformal Energy}\label{subsectionconformalenergy}
We set up the conformal energy method that was introduced by Wang in \cite{AGeoMetricApproach}. In order to give the definition of our conformal energy, we define two smooth cut-off functions $\underline{\upvarpi}$ and $\upvarpi$ that depend only on $t$ and $u$ as follows:\\
\begin{align*}
\underline{\upvarpi}&=
\begin{cases}
1& \text{on } 0\leq u\leq t\\
0& \text{on } u\leq-\frac{t}{4}
\end{cases} , &
\upvarpi&=
\begin{cases}
1& \text{on } 0\leq u\leq \frac{1}{2}t\\
0& \text{if } u\geq\frac{3}{4}t\ \text{or}\ u\leq-\frac{t}{4}
\end{cases}.
\end{align*}
\begin{definition}\cite[section 4. Conformal energy]{AGeoMetricApproach}\label{definitionofconformalenergy}
	For any scalar $\varphi$ vanishing outside $\intregion$, we define the conformal energy $\mathfrak{C}[\varphi]$as follows:
	\begin{align}
	\Energyconformal[\varphi](t)=\iEnergyconformal[\varphi](t)+\eEnergyconformal[\varphi](t),
	\end{align}
	where
	\begin{subequations}
		\begin{align}
		\iEnergyconformal[\varphi](t)&=\int_{\St}(\underline{\upvarpi}-\upvarpi)t^2\left(\abs{\Dfour\varphi}^2+\abs{\rgeo^{-1}\varphi}^2\right)\diff\tvol,\\
		\eEnergyconformal[\varphi](t)&=\int_{\St}\upvarpi\left(\rgeo^2\abs{\Dfour_\Lunit\varphi}^2+\rgeo^{2}\abs{\angnabla\varphi}^2+\abs{\varphi}^2\right)\diff\tvol.
		\end{align}
	\end{subequations}
\end{definition}

The main result of $\Energyconformal[\varphi](t)$ is the following:
\begin{theorem}[Boundedness theorem]\label{BoundnessTheorem}
	Let $\varphi$ be any solution of $\boxg\varphi=0$ on $[0,\Trescale]\times\mathbb{R}^3$ with $\varphi(1)$ supported in $B_R\subset\intregion\bigcap\Sigma_{1}$. Then under bootstrap assumptions, for $t\in[1,\Trescale]$, the conformal energy of $\varphi$ satisfies the estimate:
	\begin{align}
	\Energyconformal[\varphi](t)\lesssim(1+t)^{2\varepsilon}\left(\norm{\pfour\varphi}^2_{L_x^2(\Sigma_{1})}+\norm{\varphi}^2_{L_x^2(\Sigma_{1})}\right),
	\end{align}
	where $\varepsilon>0$ is an arbitrary small number.
\end{theorem}
\subsubsection{Reduction of Theorem \ref{Spatiallylocalizeddecay}}
The proof of Theorem \ref{Spatiallylocalizeddecay} by using Theorem \ref{BoundnessTheorem} is done via product estimates and the Berstein inequality of Littlewood-Paley theory. We refer reader to \cite[Section 8]{ImprovedLocalwellPosedness} and \cite[Section 4]{AGeoMetricApproach} for the detailed proof.
\subsubsection{Discussion of the proof of Theorem \ref{BoundnessTheorem}}\label{discussionofconformalenergy}
Logically speaking, the proof of Theorem \ref{BoundnessTheorem} relies on the control of the acoustic geometry derived in Proposition \ref{mainproof}. However, Proposition \ref{mainproof} is logically independent of Theorem \ref{BoundnessTheorem}, so here we discuss the proof of Theorem \ref{BoundnessTheorem}, assuming that we have already controlled the acoustic geometry. 
In order to obtain an estimate for the conformal energy, we choose $\Theta$ to be as follows:
\begin{align}
\Theta:=\rgeo^{-1} f:=\rgeo^{-1}\left(\beta-\frac{\beta}{(1+\rgeo)^\alpha}\right),
\end{align}
where $\beta\alpha=2$ and $\rgeo$ is as in (\ref{DEFE:rgeo}).
We introduce the modified weighted energy:
\begin{align}
\tilde{Q}[\varphi](t)=\int_{\Sigma_t}\modifiedcurrent{X}_\mu[\varphi]\Timelike^\mu \diff x,
\end{align}
where $\modifiedcurrent{X}_\mu[\varphi]$ is the modified energy current:
\begin{align}
\modifiedcurrent{X}_\mu[\varphi]=Q_{\mu\nu}[\varphi]X^\mu+\frac{1}{2}\Theta\p_\mu(\varphi^2)-\frac{1}{2}\varphi^2\p_\mu\Theta.
\end{align}

By choosing $X=f\spherenormal$ and using the divergence theorem for the modified current on an appropriate region, one can derive a Morawetz-type energy estimate. This provides the uniform bounds for the standard energy of $\varphi$ along a union of a portion of the constant-time hypersurfaces and null cones. \\
\indent
Then we consider the conformal wave equation $\square_{\rescaledgfour}(e^{-\conformalfactor}\varphi)=\cdots$, where $\varphi$ satisfies $\boxg\varphi=0$, $\conformalfactor$ and $\rescaledgfour$ are defined in Definition \ref{conformalchangemetric}. We use the multiplier approach with $\rgeo^p\Lunit$ type vectorfields in the region $\{\tau_1\leq u\leq\tau_2\}\bigcap\{\rgeo\geq R\}$ where $1\leq\tau_1<\tau_2<\Trescale$ to control the conformal energy in the exterior region and to provide energy decay along null slices. Finally we control the conformal energy in the interior region with the help of the argument in \cite{Anewphysical-spaceapproachtodecayforthewaveequationwith}, where energy decay is obtained in each spatial-null slice.\\
\indent
The proof of Theorem \ref{BoundnessTheorem} closes the reduction of the Strichartz estimate. One could follow the steps listed in \cite[Section 11]{3DCompressibleEuler} to obtain the estimates of conformal energy with the control of Ricci coefficients given in Section \ref{connectioncoefficientsandpdes}. One could go through the details of the argument in \cite[Section 7]{AGeoMetricApproach}. Also reader could look into \cite[Section 3]{ACommutingVectorfieldsApproachtoStrichartz} for initial ideas.

\section{Energy along Acoustic Null Hypersurfaces and Control of the Acoustic Geometry}\label{connectioncoefficientsandpdes}
In this section, we prove the energy estimates along acoustic null hypersurfaces, which is necessary for obtaining the mixed-norm estimates in Proposition \ref{mainproof}. Then we introduce the notation for many geometric quantities, followed by their bootstrap assumptions. Their improved estimates are in Subsection \ref{sectionRestatement}.  The proof of estimates for geometric quantities is obtained by transport equation and div-curl estimates for the acoustic quantities, decomposition of Ricci curvature components, trace and Sobolev inequalities. We omit the proof of these estimates since they are the same as in \cite[Section 10]{3DCompressibleEuler}. 
\subsection{Energy Estimates along Acoustic Null Hypersurfaces}\label{energyalongnullhypersurfaces}
In this subsection, we define acoustic null fluxes and derive energy estimates along acoustic null hypersurfaces. These estimates are necessary for deriving the mixed-norm estimates for the acoustical function quantities in Proposition \ref{mainproof}.
\begin{definition}[Acoustic null fluxes]\label{D:acouticnullfluxes}
	For functions $\varphi$ defined on $\coneu$, we define the acoustic null fluxes $\mathcal{F}_{(wave)}[\varphi;\coneu]$ and $\mathcal{F}_{(transport)}[\varphi;\coneu]$ as follows:
	\begin{align}
	\mathcal{F}_{(wave)}[\varphi;\coneu]&:=\int_{\coneu}\left((\Lunit\varphi)^2+\sgabs{\angnabla\varphi}^2\right)\diff\spherevol\diff t,\\
	\mathcal{F}_{(transport)}[\varphi;\coneu]&:=\int_{\coneu}\varphi^2\diff\spherevol\diff t,
	\end{align}
	where $\diff\spherevol$ is the volume form of reduced metric $\gsphere$ on the $\stu$-sphere from $\gfour$.
\end{definition}
\begin{proposition}[Energy estimates along acoustic null hypersurfaces]\label{EnergyNullFluxes}
	Under the initial data, bootstrap assumptions and the standard energy estimates Proposition \ref{EnergyandEllipticEstimates}, we have the following estimates along null hypersurfaces $\coneu$ for $u\in[-w_*,\Trescale]$:
	\begin{align}
	\label{coneenergy1}\mathcal{F}_{(wave)}[\pfour\vvariables;\coneu]+\sum\limits_{\upnu>1}\upnu^{2(N-2)}\mathcal{F}_{(wave)}[\littlewood\pfour\vvariables;\coneu]\lesssim\lambda^{-1},\\
	\label{coneenergy2}\mathcal{F}_{(transport)}[\pfour(\vectvort,\DivGradEnt);\coneu]+\sum\limits_{\upnu>1}\upnu^{2(N-2)}\mathcal{F}_{(transport)}[\littlewood \pfour(\vectvort,\DivGradEnt);\coneu]\lesssim\lambda^{-1}.
	\end{align}
\end{proposition}

\begin{figure} \centering    
	\subfigure[When $t_0>0$.] {
		\label{fignullflux1}     
		\includegraphics[width=0.4\columnwidth]{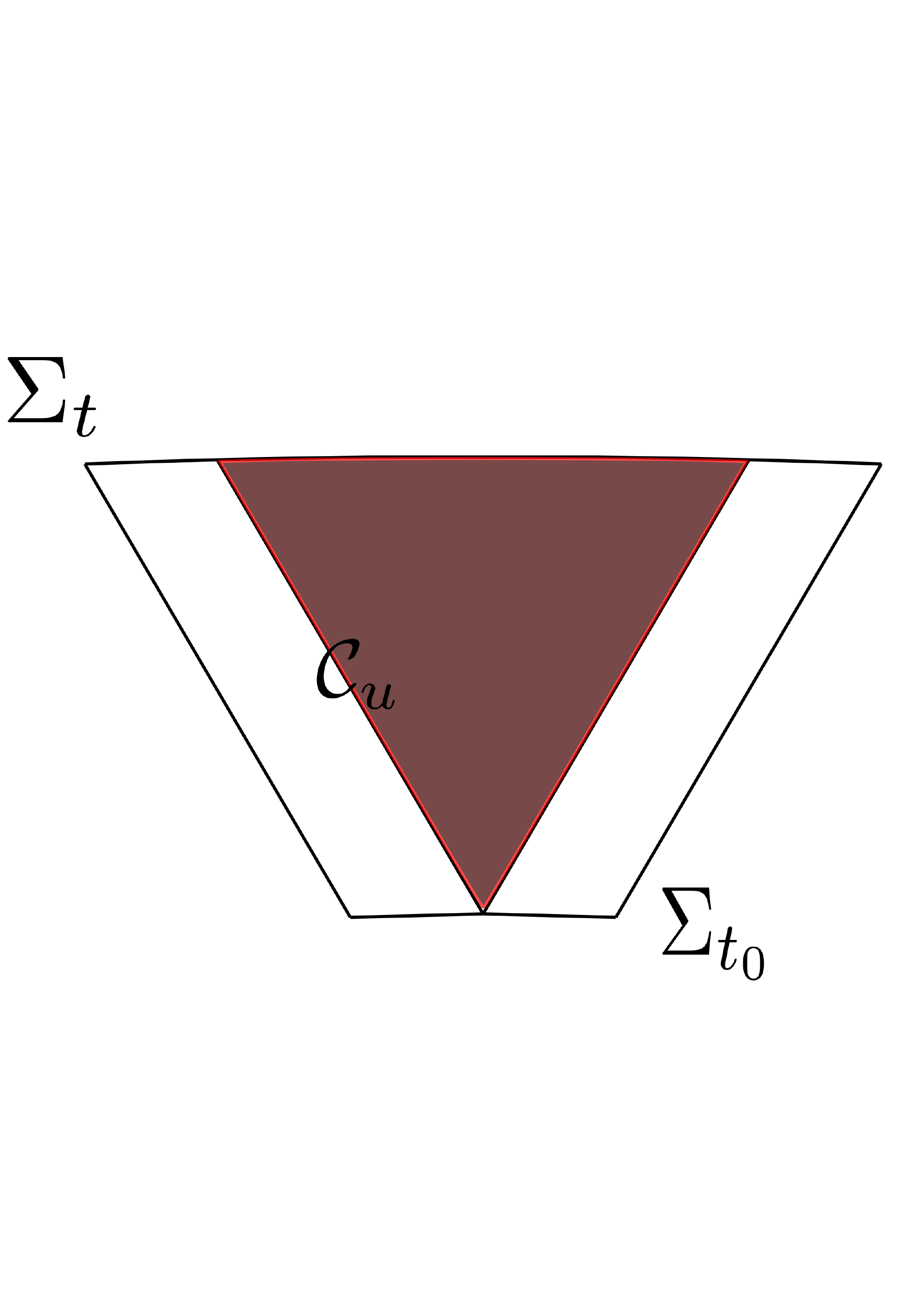}  
	}     
	\subfigure[When $t_0=0$.] { 
		\label{fignullflux2}     
		\includegraphics[width=0.4\columnwidth]{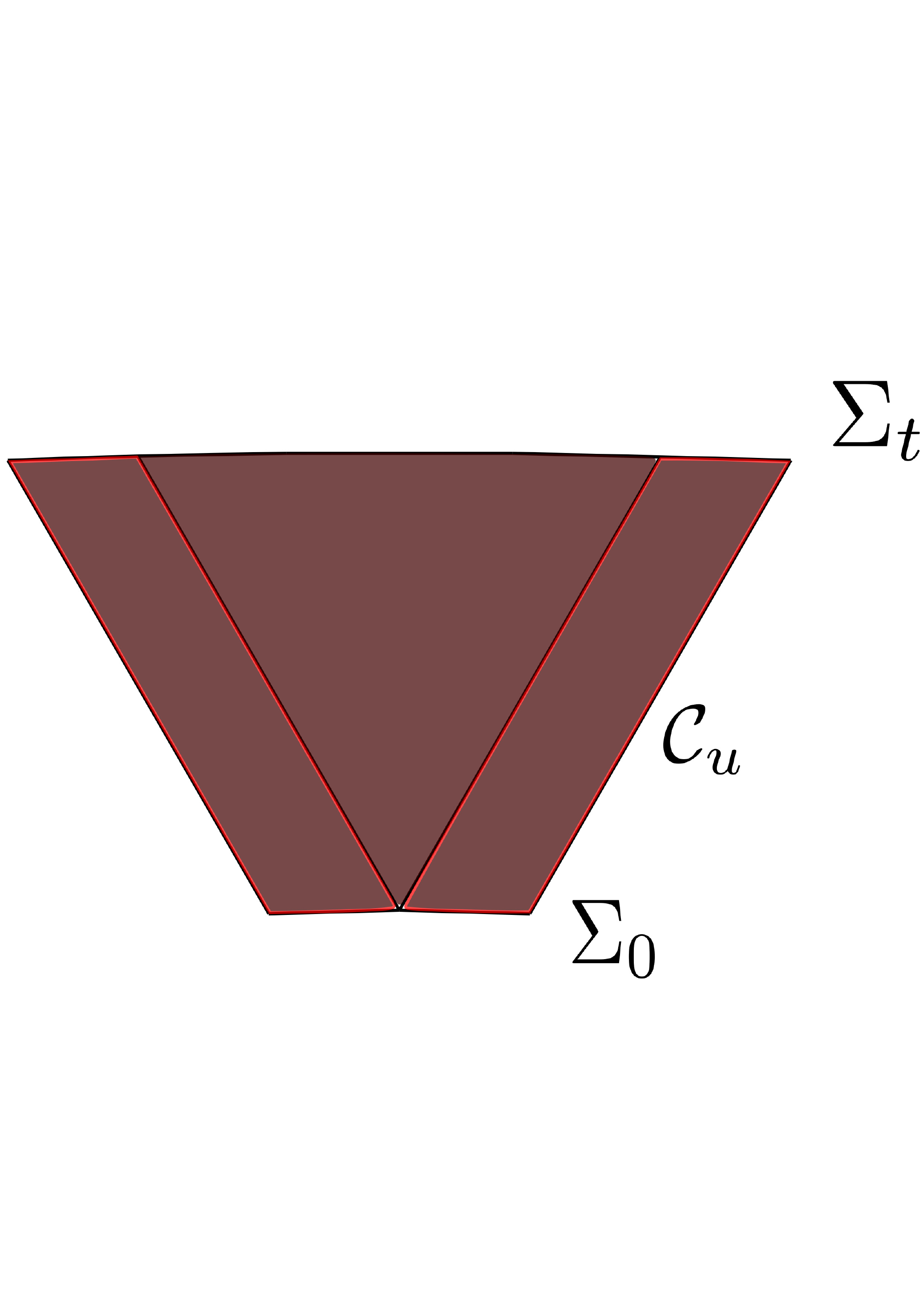}     
	}    
	\caption{The regions that the divergence theorem is applied on.}     
	\label{fignullflux}     
\end{figure}

\begin{proof}[Proof of Proposition \ref{EnergyNullFluxes}]
	
	We first prove (\ref{coneenergy1}). We apply the divergence theorem to the energy current  $\Jenwithlowerarg{\Timelike}{\alpha}[\varphi]:=Q^{\alpha\beta}[\varphi]\Timelike_\beta$, where $Q_{\alpha\beta}$ is the energy momentum tensor defined in (\ref{energymomentum}), over the region bounded by $\coneu$, $\Sigma_{t_0}$ and $\St$ where $t_0:=\max\{0,u\}$ (see figure \ref{fignullflux} for the region that divergence theorem applied). Notice that $\Dfour_\alpha\Jenarg{\Timelike}{\alpha}[\varphi]=\boxg\varphi(\Timelike\varphi)+\frac{1}{2}Q^{\mu\nu}[\varphi]\deform{\Timelike}_{\mu\nu}$ where $\deform{\Timelike}_{\mu\nu}$ is define in (\ref{deformtensor}). The same proof of Lemma \ref{coeciveness} reveals the coercivity $\Jenwithlowerarg{\Timelike}{\alpha}[\varphi]=Q^{00}\approx\abs{\pfour\varphi}^2$. It is straightforward to check that $\Jenwithlowerarg{\Timelike}{\alpha}[\varphi]\Lunit_\alpha=\left((\Lunit\varphi)^2+\sgabs{\angnabla\varphi}^2\right)$. Thus we have
	\begin{align}
	\mathcal{F}_{(wave)}[\varphi;\coneu]=\int_{\St}\abs{\pfour\varphi}^2\diff\tvol-\int_{\Sigma_{t_0}}\abs{\pfour\varphi}^2\diff\tvol+\int_{\bigcup\limits_{u^\prime\geq u}\mathcal{C}_{u^\prime}}\left(\boxg\varphi(\Timelike\varphi)+\frac{1}{2}Q^{\mu\nu}[\varphi]\deform{\Timelike}_{\mu\nu}\right)\diff\gvol.
	\end{align}
	
	Note that $\vvariables$ are rescaled quantities defined in Definition \ref{rescaledquantities}. By bootstrap assumptions (\ref{waveba}), we have $\abs{\Timelike\varphi}\lesssim\abs{\pfour\varphi}$, $\abs{Q^{\mu\nu}[\varphi]}\lesssim\abs{\pfour\varphi}^2$ and $\abs{\deform{\Timelike}_{\mu\nu}}\lesssim\abs{\pfour\vvariables}+1$. Next we substitute $\pfour\vvariables$ and $\littlewood\pfour\vvariables$ for $\varphi$,  and use equation (\ref{boxg2}) and equation (\ref{lwpboxg}) to substitute $\boxg\pfour\vvariables$ and $\boxg\littlewood\pfour\vvariables$. By Cauchy-Schwarz inequality, estimate $\int_{\bigcup\limits_{u^\prime\geq u}\mathcal{C}_{u^\prime}}\boxg\varphi(\Timelike\varphi)\diff\gvol\leq\int_0^{\Trescale}\int_{\Sigma_{\tau}}\abs{\boxg\varphi}\abs{\Timelike\varphi}\diff \tvol\diff\tau$, and energy estimates in Proposition \ref{energyestimates}, we obtain the desired estimates.
	
	To prove (\ref{coneenergy2}), for each $\beta=0,1,2,3$, we apply divergence theorem for energy fluxes
	$\Jenarg{\materialderivative}{\alpha}:=\abs{\pfour\modivort^\beta}^2\materialderivative^\alpha$ over the region bounded by $\coneu$, $\Sigma_{t_0}$ and $\St$ as follows:
	\begin{align}
	-\int_{\coneu}\abs{\pfour\modivort^\beta}^2\materialderivative^\alpha\Lunit_\alpha\diff\spherevol\diff t=\int_{\Sigma_{t_0}}\abs{\pfour\modivort^\beta}^2\materialderivative^\alpha\Timelike_\alpha\diff\tvol-\int_{\St}\abs{\pfour\modivort^\beta}^2\materialderivative^\alpha\Timelike_\alpha\diff\tvol-\int_{\bigcup\limits_{u^\prime\geq u}\mathcal{C}_{u^\prime}}\Dfour_\alpha(\abs{\pfour\modivort^\beta}^2\materialderivative^\alpha)\diff\gvol.
	\end{align}
	
	By the fact that $\materialderivative^\alpha=\Timelike^\alpha+\gensmoothfunction(\vvariables)$ and $\materialderivative$ is timelike, we have $\materialderivative^\alpha\Lunit_\alpha\approx-1$ and $\materialderivative^\alpha\Timelike_\alpha\approx-1$. By commuted equation (\ref{bc2}), we have 
	\begin{align}
	\Dfour_\alpha(\abs{\pfour\modivort^\beta}^2\materialderivative^\alpha)&=2\abs{\pfour\modivort^\beta}\materialderivative(\pfour\modivort^\beta)+\abs{\pfour\modivort^\beta}^2\left(\p_\alpha\materialderivative^\alpha-\Chfour_\beta\materialderivative^\beta\right)\\
	\notag&\lesssim\abs{\pfour\modivort^\beta}\materialderivative(\pfour\modivort^\beta)+\abs{\pfour\modivort^\beta}^2\norm{\pfour\vvariables}_{L^\infty_x(\Sigma_{\tau})}\\
	\notag&\lesssim\abs{\pfour\modivort^\beta}\pfour^2\vvariables(\pfour\vort,\pfour\gradEnt)+\abs{\pfour\modivort^\beta}\pfour\vvariables\cdot(\pfour^2\vvariables,\pfour^2\vort,\pfour^2\gradEnt)+(\abs{\pfour\modivort^\beta}^2+\abs{\pfour\vvariables}^2)\norm{\pfour\vvariables}_{L^\infty_x(\Sigma_{\tau})},
	\end{align}
	where $\Chfour$ is the contracted Christoffel symbols of $\gfour$ defined as $\Chfour_\alpha:=\gfour^{\kappa\lambda}\Gamma_{\alpha\kappa\lambda}=\gfour^{\kappa\lambda}\gfour_{\alpha\beta}\Gamma^\beta_{\kappa\lambda}$. Using the rescaled bootstrap assumptions (\ref{bt1}), energy estimates in Proposition \ref{EnergyandEllipticEstimates}, note that $\modivort,\DivGradEnt$ are rescaled quantities defined in Definition \ref{rescaledquantities}, we have 
	\begin{align}
	\mathcal{F}_{(transport)}[\pfour\vectvort;\coneu]\lesssim\lambda^{-1}.
	\end{align}
	
	The proofs for $\pfour\DivGradEnt$ and $\littlewood \pfour(\vectvort,\DivGradEnt)$ are of the same fashion. 
\end{proof}
\begin{remark}
	We refer readers to \cite[Proposition 6.1]{3DCompressibleEuler} for the energy estimate along null cones for 3D non-relativistic compressible Euler case where $\materialderivative$ coincides with $\Timelike$. We note that in \cite{3DCompressibleEuler,AGeoMetricApproach}, authors derive energy estimate along null cones before rescaling (with respect to $\lambda$). However, the results coincide with ours when energy estimates along null cones are rescaled. That is, our estimate is sufficient to obtain the mix-norm estimates for fluid variables, which serves the same as in \cite{3DCompressibleEuler,AGeoMetricApproach}.
	
\end{remark}
\subsection{Connection Coefficients}
In this subsection, we define connection coefficients. We will derive estimates for them in Proposition \ref{mainproof} as a part of controlling the acoustic geometry.
\subsubsection{Levi-Civita connections, angular operators and curvatures}\label{D:levi-civita2}
If $\upxi$ is a space-time tensor, then $\sphereproject\xi$ denotes its $\gfour$-orthogonal projection onto $\stu$. Then we define $\angD_V\upxi:=\sphereproject\Dfour_V\upxi$ where $V$ is a vector and $\Dfour_V\upxi$ is the covariant derivative of $\upxi$ in the $V$ direction. Note that $\angD_V\upxi=\angnabla_V\upxi$ when both $V$ and $\upxi$ are $\stu$-tangent. Also we define $\anglap:=\angnabla\cdot\angnabla$\\
\indent
We let $\Riemfour{\alpha}{\beta}{\gamma}{\delta}$ denote the Riemann curvature tensor of $\gfour$ and $\Ricfour{\alpha}{\beta}:=\gfour^{\gamma\delta}\Riemfour{\gamma}{\alpha}{\delta}{\beta}$. We use the notation that $\langle\Dfour_X\Dfour_YW-\Dfour_Y\Dfour_XW,Z\rangle=\Riemfour{Z}{W}{X}{Y}+\langle\Dfour_{[X,Y]}W,Z\rangle$ where $X,Y,W,Z$ are vectorfields, $[\cdot,\cdot]$ is the Lie bracket and $\langle\cdot,\cdot\rangle:=\gfour(\cdot,\cdot)$. 
\begin{definition}[Connection coefficients]
	\label{D:DEFSOFCONNECTIONCOEFFICIENTS}
	We define the second fundamental form $k$ of $\Sigma_t$
	to be the $\binom{0}{2}$ $\Sigma_t$-tangent tensor
	such that the following relation holds for all $\Sigma_t$-tangent 
	vectorfields $X$ and $Y$:
	\begin{align} 
	k(X,Y)
	& :=
	-
	\gfour(\Dfour_X \Timelike, Y).
	\label{E:SECONDFUNDOFSIGMAT} 
	\end{align}
	
	Let $\{e_A\}_{A=1,2}$ be the pair of unit orthogonal spherical vectorfields on $\stu$ from Definition \ref{D:Nullframe} .
	We define the second fundamental form $\spheresecondfund$ of $\stu$,
	the null second fundamental form $\upchi$ of $\stu$, 
	and $\underline{\upchi}$ 
	to be the following type $\binom{0}{2}$ $\stu$-tangent tensors:
	\begin{subequations}
		\begin{align}  \label{E:SECONDFUNDOFSPHERESINSIGMAT}
		\spheresecondfund_{AB}
		& := \gfour(\Dfour_A \spherenormal,e_B),
		\\
		\upchi_{AB}
		& := \gfour(\Dfour_A \Lunit, e_B),
		&
		\underline{\upchi}_{AB}
		& := \gfour(\Dfour_A \uLunit, e_B).
		\end{align}
	\end{subequations}
	
	We define the torsion $\upzeta$ and $\underline{\upzeta}$ to be the following
	$\stu$-tangent one-forms:
	\begin{align} \label{E:TORSION}
	\upzeta_A
	& := \frac{1}{2}
	\gfour(\Dfour_{\uLunit} \Lunit, e_A),
	&
	\underline{\upzeta}_A
	& := \frac{1}{2}
	\gfour(\Dfour_{\Lunit} \uLunit, e_A).
	\end{align}
	
\end{definition}

\subsection{Conformal Metric, Initial Conditions on $\Szero$ and on the Cone-tip Axis for the Acoustical Function $u$}
\begin{definition}[Conformal factor and the modified null second fundamental form and acoustical quantities]\label{conformalchangemetric}
	In interior region $\intregion$, we define $\conformalfactor$ to be the solution to the following transport initial value problem:
	\begin{subequations}
		\begin{align}
		\label{Lunitsigma}\Lunit\conformalfactor(t,u,\sangle)&=\frac{1}{2}[\Chfour_{\Lunit}](t,u,\sangle)&
		u\in[0,\Trescale],t\in[u,\Trescale],\sangle\in\Stwo,\\
		\label{initialsigma}\conformalfactor(u,u,\sangle)&=0,&
		u\in[0,\Trescale], \sangle\in\Stwo,
		\end{align}
	\end{subequations}
	where $\Chfour_{\Lunit}:=\Chfour_\alpha\Lunit^\alpha$ and $\Chfour_\alpha:=\gfour^{\kappa\lambda}\Chfour_{\alpha\kappa\lambda}=\gfour^{\kappa\lambda}\gfour_{\alpha\beta}\Chfour^\beta_{\kappa\lambda}$.
	\\ 	We define the conformal space-time metric and the Riemannian metric that induces on $S_{t,u}$ as follows:
	\begin{align}\label{rescaledgfour}
	\rescaledgfour&:=e^{2\conformalfactor}\gfour,&\congsphere:=e^{2\conformalfactor}\gsphere
	.\end{align}
	
	We define the null second fundamental forms of the metric to be the following symmetric $S_{t,u}$-tangent tensors:
	\begin{align}\label{rescaledchi}
	\reschi&:=\frac{1}{2}\angLie_{\Lunit}\congsphere,& \resuchi:=\frac{1}{2}\angLie_{\uLunit}\congsphere 
	.\end{align}
	
	Using (\ref{rescaledgfour})-(\ref{rescaledchi}), it follows that $\upchi,\uchi$ and $\reschi,\resuchi$ are related by:
	\begin{subequations}
		\begin{align}
		\reschi&=e^{2\conformalfactor}\left(\upchi+(\Lunit\conformalfactor)\gsphere\right),&\resuchi&=e^{2\conformalfactor}\left(\uchi+(\uLunit\conformalfactor)\gsphere\right),\\
		\Restrace\reschi&=\gtr\upchi+2\Lunit\conformalfactor=\gtr\upchi+\Chfour_{\Lunit},&\Restrace\resuchi&=\gtr\uchi+2\uLunit\conformalfactor,\\
		\upchi&=\frac{1}{2}\left(\Restrace\reschi-\Chfour_{\Lunit}\right)\gsphere+\hchi,&\uchi&=\frac{1}{2}\left(\Restrace\resuchi-2\uLunit\conformalfactor\right)\gsphere+\huchi
		.\end{align}
	\end{subequations}

	We define the following:
	\begin{align}\label{tracechismall}
	\Restrace\chismall:=\gtr\upchi+\Chfour_{\Lunit}-\frac{2}{\rgeo}=\Restrace\reschi-\frac{2}{\rgeo}
	.\end{align}
	We note that the first equality in (\ref{tracechismall}) holds in the whole region $\region$, while the second equality holds only in the interior region $\intregion$.
\end{definition}
In the following definition, we define mass aspect function $\mass$ and its modified version $\modmass$, as well as modified torsion $\modtorsion$. These objects are defined to avoid loss of regularity of the acoustical eikonal function.
\begin{definition}[Mass aspect function] \label{D:mass}
We define the mass aspect function $\mass$ to be the following scalar function
\begin{align}
\mass:=\uLunit\gtr\upchi+\frac{1}{2}\gtr\upchi\gtr\uchi
	.\end{align}
	
	We now define the modified mass aspect function $\modmass$ to be the following scalar function:
	\begin{align}
	\modmass:=2\anglap\conformalfactor+\uLunit\gtr\upchi+\frac{1}{2}\gtr\upchi\gtr\uchi-\gtr\upchi k_{\spherenormal\spherenormal}+\frac{1}{2}\gtr\upchi\Chfour_{\uLunit}
	.\end{align}
	
	In $\intregion$, we define $\hodgemass$ to be the $S_{t,u}$-tangent one-form that satisfies the following Hodge system on $S_{t,u}$:
	\begin{align}
	\label{hodgemass}\angdiv\hodgemass&=\frac{1}{2}(\modmass-\umodmass)&
	\angcurl\hodgemass&=0
	.\end{align}
	
	In $\intregion$, we define the modified torsion $\modtorsion$ to be the following $S_{t,u}$-tangent one-form:
	\begin{align}
	\label{defmodtorsion}\modtorsion:=\upzeta+\angnabla\conformalfactor
	.\end{align}
\end{definition}
\begin{definition}[Norms of $\stu$-tangent tensorfields]\label{D:norms}
	Let $\esphere=\esphere(\sangle)$ be the canonical metric on $\Stwo$ and $\{\asphere\}_{A=1,2}$ are local angular coordinates on $\Stwo$. We can also view $\{\asphere\}_{A=1,2}$ on $\stu$ as the image of the canonical isomorphism from $\Stwo$ to $\stu$. For $p\in[1,\infty)$, We define the Lebesgue norms $L^p_{\sangle}$ and $L^p_{\gsphere}$ for $\stu$-tangent tensorfields $\xi$ as follows:
	\begin{align}
	\onenorm{\xi}{\sangle}{p}{\stu}&:=\left(\int_{\sangle\in\Stwo}\sgabs{\xi}^p\diff\flatspherevol\right)^{1/p},&
	\onenorm{\xi}{\gsphere}{p}{\stu}&:=\left(\int_{\sangle\in\Stwo}\sgabs{\xi}^p\diff\spherevol\right)^{1/p}.	
	\end{align}
	
	Since we can have a parallel transport along $\Stwo$ that preserves the tensor products and contractions, that is, $\Phi^m_n(\sangletwo;\sangleone)$ is the parallel transport with respect to $\esphere$ from the vector space of type $\binom{m}{n}$ tensors at $\sangletwo\in\Stwo$ to the vector space of type $\binom{m}{n}$ tensors at $\sangleone\in\Stwo$, for $\xi=\xi(\sangle)$ a type $\binom{m}{n}$ tensorfield on $\Stwo$, we define the $L^\infty$ norm $L_\sangle^\infty$ and H\"older norm $C_\sangle^{0,\alpha}$ of $\xi$ as follows:
	\begin{align}
	\onenorm{\xi}{\sangle}{\infty}{\stu}&:=\mathrm{ess}\sup\limits_{\sangle\in\Stwo}\sgabs{\xi},\\
	\holdernorm{\xi}{\sangle}{0}{\alpha}{\stu}&:=\onenorm{\xi}{\sangle}{\infty}{\stu}+\sup\limits_{0<d_{\esphere}(\sangleone,\sangletwo)<\frac{\pi}{2}}\frac{|\xi(t,u,\sangleone)-\Phi^m_n(\sangletwo;\sangleone)\circ\xi(t,u,\sangletwo)|_{\gsphere(t,u,\sangleone)}}{d_{\esphere}^\alpha(\sangleone,\sangletwo)}.
	\end{align}
\end{definition}

In the following two propositions, we list the estimates of the initial foliation. These estimates are crucial for the well-defined geometric setup in Section \ref{opticalfunction} and controlling the acoustic geometry.
\begin{proposition}\cite[Proposition 9.8]{3DCompressibleEuler}, \cite[Proposition 4.3. Existence and properties of the initial foliation]{AGeoMetricApproach}.\label{InitialCT1}
	
	On $\Szero$, there exists a function $w=w(x)$ on the domain $0\leq w\leq w_*:=\frac{4}{5}\Trescale$, such that $w(\tip)=0$ and each level set $S_{0,w}$ is diffeomorphic to $\Stwo$ and 
	\begin{align}\label{initialtracecondition}
	\gtr\spheresecondfund+k_{\spherenormal\spherenormal}&=\frac{2}{\lapse w}+\ggtr k-\Chfour_{\Lunit},& \lapse(\tip)&=1
	.\end{align}
	By (\ref{tracechismall}), $\rgeo(0,-u)=w$, and the fact that $\upchi_{AB}=\spheresecondfund_{AB}-k_{AB}$, (\ref{initialtracecondition}) is equivalent to 
	\begin{align}
	\label{zinitial}\Restrace\chismall|_{\Szero}&=\frac{2(1-\lapse)}{\lapse w},& \text{for}\ 0&\leq w\leq w_*
	,\end{align}
	where we define the lapse $a$ as follows:
	\begin{align}
	a=\left(\sqrt{g^{cd}\p_cw\p_dw}\right)^{-1}.
	\end{align}
	Note that $\derinormal=\lapse\spherenormal|_{\Sigma_{0}}$.
Then on $\initialSzero:=\bigcup\limits_{0\leq w\leq w_*}S_{0,w}$, for $0<1-\frac{2}{q_{*}}<N-2$, there hold
	\begin{subequations}
		\begin{align}
		\label{lapseminusone}\abs{\lapse-1}&\lesssim\lambda^{-4\varepsilon_0}\leq\frac{1}{4},& \holdertwonorms{w^{-1/2}(\lapse-1)}{w}{\infty}{\sangle}{0}{1-\frac{2}{q_{*}}}{\initialSzero}&\lesssim\lambda^{-1/2},&
		\volume:=\frac{\sqrt{\det\gsphere}}{\sqrt{\det\esphere}}&\approx w^2
		.\end{align}
		\begin{align}
		\label{initialfoliationtwo}\twonorms{w^{\frac{1}{2}-\frac{2}{q_{*}}}(\hat{\spheresecondfund},\angnabla\ln\lapse)}{w}{\infty}{\gsphere}{q_{*}}{\initialSzero}, \twonorms{\angnabla\ln\lapse}{w}{2}{\sangle}{\infty}{\initialSzero}, \twonorms{\hchi}{w}{2}{\sangle}{\infty}{\initialSzero}&\lesssim\lambda^{-1/2} 
		.\end{align}
		\begin{align}
		\label{initialfoliation3}\max\limits_{A,B=1,2}\norm{w^{-2}\gsphere\left(\deriasphere,\deribsphere\right)-\esphere\left(\deriasphere,\deribsphere\right)}_{L^\infty_x(\initialSzero)}&\lesssim\lambda^{-4\varepsilon_0},\\
		\label{initialfoliation4}\max\limits_{A,B,C=1,2}\onenorm{\dericsphere\left(w^{-2}\gsphere\left(\deriasphere,\deribsphere\right)-\esphere\left(\deriasphere,\deribsphere\right)\right)}{\sangle}{q_{*}}{S_{0,w}}&\lesssim\lambda^{-4\varepsilon_0},\\
		\label{in5}\onenorm{w^{\frac{1}{2}-\frac{2}{q_{*}}}\angnabla\cphi}{\gsphere}{q_{*}}{S_{0,w}}&\lesssim\lambda^{-1/2}
		.\end{align}
		In addition,
		\begin{align}
		\label{in6}\norm{w\Restrace\chismall}_{L^\infty_x\left(\initialSzero\right)}&\lesssim\lambda^{-4\varepsilon_0},\\
		\label{in7}\twonorms{w^{3/2}\angnabla\Restrace\chismall}{w}{\infty}{\sangle}{p}{\initialSzero}+\holdertwonorms{w^{1/2}\Restrace\chismall}{w}{\infty}{\sangle}{0}{1-\frac{2}{p}}{\initialSzero}&\lesssim\lambda^{1/2}
		.\end{align}
		Finally,
		\begin{align}
		\label{pNatSigma0}\sum\limits_{i,j=1,2,3}\abs{w\sphereproject_j^a\p_a\spherenormal^i-\sphereproject_j^i}=\zero(w)\ \text{as} \ w\downarrow 0
		.\end{align}
	\end{subequations}
\end{proposition}
\begin{proposition}\cite[Lemma 9.9. Initial conditions on the cone-tip axis tied to the acoustical function]{3DCompressibleEuler}\label{InitialCT2} On any null cone $\coneu$ initiating from a point on the time axis $0\leq t=u\leq \Trescale$ there hold
	\begin{subequations}
		\begin{align}
		\label{initialcontipone}&\gtr\upchi-\frac{2}{\rgeo}, \rgeo\Restrace\chismall, \sgabs{\hchi}, \abs{\rgeo\sphereproject_j^a\p_a\Lunit^i-\sphereproject_j^i},\nulllapse-1,\sgabs{\upzeta},\conformalfactor,\\
		\notag&\rgeo\sgabs{\angnabla\gtr\upchi},\rgeo^2\sgabs{\angnabla\Restrace\chismall},\rgeo\sgabs{\angnabla\hchi},\rgeo\sgabs{\angnabla\nulllapse},\rgeo\sgabs{\angnabla\upzeta}, \rgeo\sgabs{\angnabla\conformalfactor},\\
		\notag&\rgeo^2\anglap\nulllapse,\rgeo^2\anglap\conformalfactor,\rgeo^2\mass,\rgeo^2\modmass\\
		\notag&=\zero(\rgeo) \ \text{as}\ t\downarrow u,\\
		\lim\limits_{t\downarrow u}\norm{\uzeta,k}_{L^{\infty}(\stu)}&<\infty
		.\end{align}
		
		Moreover,
		\begin{align}
		\label{initialcontip3}\lim\limits_{t\downarrow u} \rgeo^{-2}\gsphere\left(\deriasphere,\deribsphere\right)&=\esphere\left(\deriasphere,\deribsphere\right),\\
		\label{initialcontip4}\lim\limits_{t\downarrow u} \rgeo^{-2}\dericsphere\gsphere\left(\deriasphere,\deribsphere\right)&=\dericsphere\esphere\left(\deriasphere,\deribsphere\right)
		.\end{align}
		
	\end{subequations}
\end{proposition}

\begin{proof}[Discussion of the proof of Proposition \ref{InitialCT1} and Proposition \ref{InitialCT2}]

	The existence of such initial foliation can be proved by Nash-Moser implicit function theorem (see \cite{NashMoser1}). The proof of the estimates in Prop.\ref{InitialCT1} relies on the energy estimates (\ref{energyestimates}). The reason is that Prop.\ref{InitialCT1} yields a foliation and estimates on the hypersurface $\Sigma_{0}$ with respect to the rescaled coordinates, and this hypersurface corresponds to the hyperfaces $\Sigma_{t_k}$ for each $k$ with respect to original space-time (see Remark \ref{newsigma0} and Subsection \ref{solutionofu} for the description of $\Sigma_{0}$ in rescaled space-time). The point is that we need the energy estimates of (\ref{energyestimates}) to control the fluid along the ``old" $\Sigma_{t_k}$. We refer the reader to \cite[Appendix C]{AGeoMetricApproach} for the proof of the estimates in Prop.\ref{InitialCT1} and in Prop.\ref{InitialCT2}.
\end{proof}

\subsection{Restatement of Bootstrap Assumptions and Estimates for Quantities Constructed out of the Acoustical Eikonal Equation}\label{sectionRestatement}
In this section, we restate the consequence of bootstrap assumptions of fluid variables, vorticity, and entropy gradient that we obtained in (\ref{bt1}), followed by the bootstrap assumptions for acoustic geometry. Then we state the main estimates for the acoustical function quantities in Prop.\ref{mainproof}. The estimates in Prop.\ref{mainproof} are required by the conformal energy method to close the whole bootstrap argument in this article. We provide a discussion of the proof of Prop.\ref{mainproof} in Section \ref{finalproof2} via a bootstrap argument where the bootstrap assumptions are listed in Section \ref{BAGeo}. For the details of the proof, we refer readers to \cite[Section 10]{3DCompressibleEuler} and \cite[Section 5-6]{AGeoMetricApproach}.
\subsubsection{The fixed number $p$}\label{choiceofp}
In the rest of the article, $p$ denotes a fixed number with 
\begin{align}
0<\delta_0<1-\frac{2}{p}<N-2,
\end{align}
where $\delta_0$ is defined in Section \ref{ChoiceofParameters}.
\subsubsection{Bootstrap assumptions for geometric quantities} After rescaling in Subsection \ref{SelfRescaling}, we make several bootstrap assumptions for the quantities in acoustic geometry. These assumptions will be recovered and improved by estimates in Proposition \ref{mainproof}\label{BAGeo}\\

First we recall (\ref{bt1}):

\textbf{Estimates by using bootstrap assumptions of variables}
	\begin{align}
	\label{bt1.1}\twonorms{\pfour\vvariables,\pfour\vVort,\pfour\vgradientEnt,\vectvort,\DivGradEnt}{t}{2}{x}{\infty}{\region}+\lambda^{\delta_0}\sqrt{\sum\limits_{\upnu>2}\upnu^{2\delta_0}\twonorms{\littlewood\left(\gensmoothfunction(\vvariables,\vVort,\vgradientEnt)(\pfour\vvariables,\pfour\vVort,\pfour\vgradientEnt,\vectvort,\DivGradEnt)\right)}{t}{2}{x}{\infty}{\region}^2}\lesssim\lambda^{-1/2-4\varepsilon_0}
	.\end{align}
	
	Then we make few more assumptions for the quantities of acoustical geometry as follows:
	
	\textbf{Bootstrap assumptions for the acoustical function quatities}
	\begin{subequations}
		\begin{align}
		\label{bt2}\max\limits_{A,B=1,2}\norm{\rgeo^{-2}\gsphere\left(\deriasphere,\deribsphere\right)-\esphere\left(\deriasphere,\deribsphere\right)}_{L^\infty_x(\region)}&\lesssim\lambda^{-\varepsilon_0},\\
		\label{bt3}\max\limits_{A,B,C=1,2}\twonorms{\dericsphere\left(\rgeo^{-2}\gsphere\left(\deriasphere,\deribsphere\right)-\esphere\left(\deriasphere,\deribsphere\right)\right)}{t}{\infty}{\sangle}{p}{\coneu}&\lesssim\lambda^{-\varepsilon_0}
		.\end{align}
	\end{subequations}

	Also,
	\begin{align}
	\label{bt4}\holdertwonorms{\Restrace\chismall,\hchi,\upzeta}{t}{2}{\sangle}{0}{\delta_0}{\coneu}&\lesssim\lambda^{-1/2+2\varepsilon_0}
	.\end{align}
	
	Moreover,
	\begin{subequations}
		\begin{align}
		\label{bt5}\onenorm{\rgeo\left(\Restrace\chismall,\hchi,\upzeta\right)}{\sangle}{p}{\stu}&\leq 1,\\
		\label{bt6}\norm{\nulllapse-1}_{L^\infty_\sangle(\stu)}&\leq\frac{1}{2}
		.\end{align}
	\end{subequations}

	Finally, we assume that the following estimates hold in the interior region:
	\begin{align}
	\label{bt7}\holderthreenorms{\Restrace\chismall,\hchi}{t}{2}{u}{\infty}{\sangle}{0}{\delta_0}{\intregion}&\leq\lambda^{-1/2},&
	\twonorms{\upzeta}{t}{2}{x}{\infty}{\intregion}&\leq\lambda^{-1/2},& 
	\threenorms{\angnabla\conformalfactor}{u}{2}{t}{2}{\sangle}{\infty}{\intregion}&\leq 1
	.\end{align}

\subsubsection{Main estimates for the acoustical function quantities}
In the following proposition, we derive the estimates for the various acoustical function quantities. These estimates are sufficient to derive the boundness theorem of the conformal energy in Theorem \ref{BoundnessTheorem} and to recover and improve the bootstrap assumptions in Section \ref{BAGeo}.
\begin{proposition}\cite[section 10]{3DCompressibleEuler}, \cite[section 5-6. The main estimates for the acoustical function quantities]{AGeoMetricApproach}. Under the bootstrap assumptions we have the following estimates where $2<q\leq4$:\label{mainproof}\\ 
	\textbf{Estimates for connection coefficients}
	\begin{subequations}
		\begin{align}
		\label{mr1}\twonorms{\Restrace\chismall,\hchi,\upzeta}{t}{2}{\sangle}{p}{\coneu},\twonorms{\rgeo\angD_{\Lunit}\left(\Restrace\chismall,\hchi,\upzeta\right)}{t}{2}{\sangle}{p}{\coneu}&\lesssim\lambda^{-1/2},\\
		\label{mr2}\twonorms{\rgeo^{1/2}\left(\Restrace\chismall,\hchi,\upzeta\right)}{t}{\infty}{\sangle}{p}{\coneu}&\lesssim\lambda^{-1/2},\\
		\label{mr3}\twonorms{\rgeo\left(\Restrace\chismall,\hchi,\upzeta\right)}{t}{\infty}{\sangle}{p}{\coneu}&\lesssim\lambda^{-4\varepsilon_0},
		\end{align}
	\end{subequations}
	\begin{subequations}
		\begin{align}
		\label{mr4}\rgeo\Restrace\reschi&\approx 1,\\
		\label{mr5}\norm{\rgeo^{1/2}\Restrace\chismall}_{L^\infty(\region)}&\lesssim\lambda^{-1/2},\\
		\label{mr6}\threenorms{\rgeo^{3/2}\angnabla\Restrace\chismall}{t}{\infty}{u}{\infty}{\sangle}{p}{\region}&\lesssim\lambda^{-1/2},\\
		\label{mr7}\twonorms{\rgeo\left(\angnabla\Restrace\chismall,\angnabla\hchi\right)}{t}{2}{\sangle}{p}{\coneu}&\lesssim\lambda^{-1/2},\\
		\label{mr8}\holdertwonorms{\Restrace\chismall,\hchi,\upzeta}{t}{2}{\sangle}{0}{\delta_0}{\coneu}&\lesssim\lambda^{-1/2}
		.\end{align}
	\end{subequations}

	In addition, the null lapse $\nulllapse$ verifies the following:
	\begin{align}
	\label{mr9}\twonorms{\frac{\nulllapse^{-1}-1}{\rgeo}}{t}{2}{x}{\infty}{\region},\threenorms{\frac{\nulllapse^{-1}-1}{\rgeo^{1/2}}}{t}{\infty}{u}{\infty}{\sangle}{2p}{\region},\twonorms{\rgeo(\angD_{\Lunit},\angnabla)\left(\frac{\nulllapse^{-1}-1}{\rgeo}\right)}{t}{2}{\sangle}{p}{\coneu}&\lesssim\lambda^{-1/2}
	.\end{align}
	
	Moreover, we have:
	\begin{align}
	\label{mr10}\holderthreenorms{\lgensmoothfunction}{t}{\infty}{u}{\infty}{\sangle}{0}{\delta_0}{\region}&\lesssim 1
	.\end{align}
	
	Furthermore,
	\begin{align}
	\label{mr11}\holderthreenorms{\Restrace\chismall,\hchi,\gtr\upchi-\frac{2}{\rgeo}}{t}{\frac{q}{2}}{u}{\infty}{\sangle}{0}{\delta_0}{\region}&\lesssim\lambda^{\frac{2}{q}-1-4\varepsilon_0\left(\frac{4}{q}-1\right)},\\
	\label{mr115}\twonorms{\upzeta}{t}{2}{x}{\infty}{\region}&\lesssim\lambda^{-1/2-3\varepsilon_0}&,\twonorms{\upzeta}{t}{\frac{q}{2}}{x}{\infty}{\region}&\lesssim\lambda^{\frac{2}{q}-1-4\varepsilon_0\left(\frac{4}{q}-1\right)}
	.\end{align}
	\textbf{Improved estimates in the interior region}
	\begin{align}
	\label{mr12}\twonorms{\frac{\nulllapse^{-1}-1}{\rgeo}}{t}{2}{x}{\infty}{\intregion}&\lesssim\lambda^{-1/2-4\varepsilon_0},\\
	\label{mr13}\twonorms{\rgeo^{1/2}\left(\Restrace\chismall,\hchi,\upzeta\right)}{\sangle}{2p}{t}{\infty}{\coneu}&\lesssim\lambda^{-1/2},& \text{if} \ \coneu&\subset\intregion,\\
	\label{mr14}\holderthreenorms{\Restrace\chismall,\hchi,\gtr\upchi-\frac{2}{\rgeo}}{t}{2}{u}{\infty}{\sangle}{0}{\delta_0}{\intregion}&\lesssim\lambda^{-1/2-3\varepsilon_0},&
	\twonorms{\upzeta}{t}{2}{x}{\infty}{\intregion}&\lesssim\lambda^{-1/2-3\varepsilon_0}
	.\end{align}
	\textbf{Estimates for the geometric angular coordinate components of $\gsphere$}
	\begin{subequations}
		\begin{align}
		\label{mr15}\max\limits_{A,B=1,2}\norm{\rgeo^{-2}\gsphere\left(\deriasphere,\deribsphere\right)-\esphere\left(\deriasphere,\deribsphere\right)}_{L^\infty(\region)}&\lesssim\lambda^{-4\varepsilon_0}\\
		\label{mr16}\max\limits_{A,B,C=1,2}\twonorms{\dericsphere\left(\rgeo^{-2}\gsphere\left(\deriasphere,\deribsphere\right)-\esphere\left(\deriasphere,\deribsphere\right)\right)}{t}{\infty}{\sangle}{p}{\coneu}&\lesssim\lambda^{-4\varepsilon_0}
		.\end{align}
	\end{subequations}
	\textbf{Estimates for $\volume$ and $\nulllapse$}
	\begin{subequations}
		\begin{align}
		\label{mr17}\volume&:=\frac{\sqrt{\mathrm{det}\gsphere}}{\sqrt{\mathrm{det}\esphere}}\approx\rgeo^2,\\
		\label{mr18}\norm{\nulllapse-1}_{L^\infty(\region)}&\lesssim\lambda^{-4\varepsilon_0}<\frac{1}{4}
		.\end{align}
	\end{subequations}

	Furthermore,
	\begin{align}
	\label{mr19}\threenorms{\rgeo^{1/2}\angnabla\cphi}{t}{\infty}{u}{\infty}{\sangle}{p}{\region}, \twonorms{\angnabla\cphi}{t}{2}{\sangle}{p}{\coneu},
	\twonorms{\rgeo\Lunit\angnabla\cphi}{t}{2}{\sangle}{p}{\coneu}&\lesssim\lambda^{-1/2}
	.\end{align}
	\textbf{Estimates for $\mass$ and $\angnabla\upzeta$}
	\begin{align}
	\label{mr20}\twonorms{\rgeo\mass,\rgeo\angnabla\upzeta}{t}{2}{\sangle}{p}{\coneu}\lesssim\lambda^{-1/2}
	.\end{align}
	\textbf{Interior region estimates for $\conformalfactor$}
	\begin{subequations}
		\begin{align}
		\label{mr21}\twonorms{\rgeo^{1/2}\Lunit\conformalfactor}{t}{\infty}{\sangle}{2p}{\coneu}\lesssim\lambda^{-1/2-2\varepsilon_0}, \twonorms{\rgeo^{1/2}\angnabla\conformalfactor}{\sangle}{p}{t}{\infty}{\coneu}, \twonorms{\angnabla\conformalfactor}{t}{2}{\sangle}{p}{\coneu}&\lesssim\lambda^{-1/2},&  \text{if}\ \coneu&\subset\intregion,\\
		\label{mr22}\norm{\conformalfactor}_{L^\infty\left(\intregion\right)}&\lesssim\lambda^{-8\varepsilon_0},\\
		\label{mr23}\norm{\rgeo^{-1/2}\conformalfactor}_{L^\infty\left(\intregion\right)}&\lesssim\lambda^{-1/2-4\varepsilon_0}
		.\end{align}
	\end{subequations}
	\textbf{Interior region estimates for $\conformalfactor$, $\mass$, $\modtorsion$ and $\hodgemass$}
	\begin{subequations}
		\begin{align}
		\label{mr24}\holderthreenorms{\angnabla\conformalfactor}{u}{2}{t}{2}{\sangle}{0}{\delta_0}{\intregion},\threenorms{\rgeo\modmass,\rgeo\angnabla\modtorsion}{u}{2}{t}{2}{\sangle}{p}{\intregion}&\lesssim\lambda^{-4\varepsilon_0},\\
		\label{mr25}\threenorms{\rgeo^{\frac{3}{2}}\modmass}{u}{2}{t}{\infty}{\sangle}{p}{\intregion}&\lesssim\lambda^{-4\varepsilon_0}
		.\end{align}
	\end{subequations}

	In addition,
	\begin{align}
	\label{mr26}\threenorms{\rgeo\angnabla\hodgemass,\hodgemass}{t}{2}{u}{2}{\sangle}{p}{\intregion},\holderthreenorms{\hodgemass}{t}{2}{u}{2}{\sangle}{0}{\delta_0}{\intregion}&\lesssim\lambda^{-4\varepsilon_0}
	.\end{align}
	\textbf{Decomposition of $\angnabla\conformalfactor$ and corresponding estimates in the interior region}: In $\intregion$, we can decompose $\angnabla\conformalfactor$ as follows:
	\begin{align}
	\angnabla\conformalfactor=-\upzeta+(\modtorsion-\hodgemass)+\massone+\masstwo
	,\end{align} 
	where the following asymptotic conditions near the cone-tip axis are satisfied:
	\begin{align}
	\label{mr265}\rgeo\massone(t,u,\sangle),\rgeo\masstwo(t,u,\sangle)=\zero(\rgeo) \ \text{as}\ t\downarrow u
	.\end{align}
	
	Moreover
	\begin{subequations}
		\begin{align}
		\label{mr27}\twonorms{\modtorsion-\hodgemass}{t}{2}{x}{\infty}{\intregion},\twonorms{\massone}{t}{2}{x}{\infty}{\intregion}&\lesssim\lambda^{-1/2-3\varepsilon_0},\\
		\label{mr28}\threenorms{\masstwo}{u}{2}{t}{\infty}{\sangle}{\infty}{\intregion}&\lesssim\lambda^{-1/2-4\varepsilon_0}
		.\end{align}
	\end{subequations}
\end{proposition}
\subsubsection{Discussion of the proof of Proposition \ref{mainproof}}\label{finalproof2}
In this subsection, we discuss the proof of Proposition \ref{mainproof}. For the details of the proof, we refer readers to \cite[Section 10]{3DCompressibleEuler} and \cite[Section 5-6]{AGeoMetricApproach}. 

The structure of the proof consists of 3 major steps: 

	\begin{enumerate}
		
	\item 
 First of all, we derive the transport equations along null hypersurfaces and div-curl systems\footnote{These div-curl systems depend on the acoustic geometry and are independent of the structure of the relativistic Euler equations. Therefore, these div-curl systems are completely unrelated to the ones that we derived for the vorticity and entropy gradient in Section \ref{section4}.} on $\stu$ verified by the geometric quantities. These equations are derived using basic differential geometry and, at the appropriate spots, using the relativistic Euler equations for substitution for reasons further described in Step 2 below. The key point is that \textit{all of the equations we obtain have the exact same schematic structure as the equations in \cite{3DCompressibleEuler}}. We refer readers to \cite[Section 2]{ImprovedLocalwellPosedness} for the PDEs satisfied by connection coefficients and mass aspect function $\mass$. We refer readers to \cite[Section 6]{AGeoMetricApproach} for the PDEs of conformal factor $\conformalfactor$, modified torsion $\modtorsion$ and modified mass aspect function $\modmass,\hodgemass$. 
\item 
 Secondly, certain Ricci and Riemann curvature tensor components, which appear as source terms in the PDEs that we just obtained in the previous step, are rewritten by using Bianchi identities and the decomposition of the following Ricci curvature \cite[Lemma 2.1]{ImprovedLocalwellPosedness}:
\begin{align}
\Ricfour{\alpha}{\beta}=-\frac{1}{2}\boxg \gfour_{\alpha\beta}(\vvariables)+\frac{1}{2}\left(\Dfour_\alpha\Chfour_\beta+\Dfour_\beta\Chfour_\alpha\right)+\quadsmoothfunction(\vvariables)[\pfour\vvariables,\pfour\vvariables].
\end{align} 
It is at this step that the wave equation (\ref{waveEQ})  of the geometric formulation of the relativistic Euler equations is used to substitute for the term $\boxg \gfour_{\alpha\beta}(\vvariables)$, as we alluded to in the previous step. We again emphasize that following this substitution, one obtains equations of the exact same schematic form as in \cite{3DCompressibleEuler}. 

After the substitution, one is faced with controlling the source terms in the geometric equations in mixed space-time norms. The source terms depend on $\pfour\vvariables,\vvort,\vgradientEnt,\vectvort,\DivGradEnt$ and have the same schematic structure as in \cite{3DCompressibleEuler}. This step requires trace inequalities and Sobolev inequalities, which are provided by \cite[Proposition 10.2]{3DCompressibleEuler} and \cite[(5.34)-(5.39)]{AGeoMetricApproach}; the proofs of these inequalities are the same as in \cite{3DCompressibleEuler}, and rely on bootstrap assumptions (\ref{waveba})-(\ref{transportba}), energy estimates on constant-time hypersurfaces (\ref{energyestimates}) and energy estimates on null hypersurfaces (\ref{coneenergy1})-(\ref{coneenergy2}).

\item 
 After one has controlled the source terms in the geometric PDEs for the acoustic geometry from Step 1, one uses a transport lemma and div-curl estimates to obtain various mixed space-time norms estimates for $\rgeo^p$ weighted acoustic quantities in Proposition \ref{mainproof}. We refer readers to \cite[Section 10.9]{3DCompressibleEuler} and \cite[Section 5-6]{AGeoMetricApproach} for the detailed proof. We emphasize that the proof is the same as in \cite{3DCompressibleEuler} because it relies only on bootstrap assumptions (\ref{waveba})-(\ref{transportba}) and source term bounds that one obtained in Step 2, ingredients which are already available to us at this step in the proof.
	\end{enumerate}

We now point out some differences between the non-relativistic 3D compressible Euler equations and the relativistic Euler equations in terms of the control of acoustic geometry. Besides the different acoustic metric $\gfour$ in this paper compared to \cite{3DCompressibleEuler}, the $\gfour$-$\St$-normal vectorfield is $\Timelike$ (see (\ref{timelike2})) in this paper, while in the non-relativistic case \cite{3DCompressibleEuler}, it is $\materialderivative=\p_t+v^a\p_a$. Although these differences have necessitated changes to some of the proofs earlier in the paper (such as the proof of the energy estimates on constant-time hypersurfaces and the acoustic null hypersurfaces), these changes do not have any effect on the proofs of the estimate for the acoustic geometry; it is for this reason that we refer to \cite[Section 10]{3DCompressibleEuler} for the details behind the proof of Proposition \ref{mainproof}.
\appendix
\section*{Appendices}
\addcontentsline{toc}{section}{Appendices}
\renewcommand{\thesubsection}{\Alph{subsection}}

\subsection{Notations}\label{B}
In this appendix, we gather the notations that we use throughout the article.
\begin{center}

	\begin{tabular}{|c|c|} 
		\hline
		Symbol & Reference  \\ 
		\hline\hline
		$\minkowski,\minkowski^{-1},\epsilon_{\gamma\delta\kappa\lambda},\St,\p,\pfour$ & Section \ref{Notations}  \\ 
		\hline
		$v,\pressure,n,\ent,\temp,\enth$ & Subsection \ref{thebasicfluidvariables} \\
		\hline
		$\text{vort}^\alpha(\cdot)$ & Definition \ref{vorthogonalvorticity} \\
		\hline
		$\vort$ & Definition \ref{vorticity} \\
		\hline
		$\lnenth$ & Definition \ref{D:loglnen} \\
		\hline
		$\tempoverEnth$ & Definition \ref{D:Tempeatureoverenth} \\
		\hline
		$\gradEnt$ & Definition \ref{D:entgradient} \\
		\hline
		$\speed$ & Subsection \ref{D:speedofsound} \\
		\hline
		
		$\modivort, \DivGradEnt$ & Definition \ref{modifiedfluidvariables} \\
		\hline
		$\gacoustical, \gacoustical^{-1}$ & Definition \ref{D:acousticunrescaled} \\
		\hline
		$\gfour,\gfour^{-1}$ & Definition \ref{D:acousticmetric} \\
		\hline
		$\Timelike$ & Definition \ref{timelike} \\
		\hline
		$g, g^{-1}$ & Definition \ref{D:inducegonconstanttime} \\
		\hline
		$\Dfour,\boxg$ & Definition \ref{LeviCivita} \\
		\hline
		$\vec{v},\vvort,\vgradientEnt,\vectvort,\vvariables,\materialderivative$ & Definition \ref{D:cartesiancomponentsofvariables} \\
		\hline
		$\linsmoothfunction,\quadsmoothfunction$ & Subsection \ref{geometriceq} \\
		\hline

		$\inhom{\modivort^\alpha}, \inhom{\DivGradEnt}$ & Proposition \ref{WaveEqunderoringinalacoustical} \\
		\hline
		$\inhom{\variables}$ & Proposition \ref{Waveequationsafterrescalingtheacousticalmetric} \\
		\hline
		$\upnu, \littlewood, P_I, P_{\leq\upnu}$ & Section \ref{definitionlittlewood} \\
		\hline
		$\remainder{\pfour\variables},\remainder{\p\modivort^\alpha},\remainder{\p\DivGradEnt}$ & Lemma \ref{lwpequations} \\
		\hline

		$q,\varepsilon_0,\delta_0,\delta,\delta_1$ & Section \ref{ChoiceofParameters} \\
		\hline

		$\Domain,\breve{\Domain}$ & Section \ref{Data} \\
		\hline

		$\Tstar$ & Section \ref{bootstrap} \\
		\hline
		$Q, \Jen{X},\deform{X},\mathbb{E},\diff\tvol$ & Definition \ref{energymomentum} \\
		\hline
		$\lambda$ & Section \ref{partitioning} \\
		\hline

		$\Trescale$ & Section \ref{SelfRescaling} \\
		\hline
		$u,\coneu,\stu,\intregion,\extregion,\region$ & Section \ref{opticalfunction} \\
		\hline
		$\rgeo, w_*$ & Definition \ref{D:rescaledquantities}\\
		\hline
		$\Lgeo,\nulllapse,\spherenormal,\Lunit,\uLunit,\gsphere,e_A, \asphere$ & Definition \ref{D:Nullframe}\\
		\hline

		$\sphereproject,\sgabs{\upxi},\gtr\upxi,\hat{\upxi}$ & Definition \ref{D:stutensorfields}\\
		\hline
		$\mathcal{F}_{(wave)}, \mathcal{F}_{(transport)},\diff\spherevol$ & Definition \ref{D:acouticnullfluxes}\\
		\hline
		$\Energyconformal,\iEnergyconformal,\eEnergyconformal$ & Definition \ref{definitionofconformalenergy}\\
		\hline
		$\angD,\angnabla,\anglap,\Ricfour{}{},\Riemfour{}{}{}{}$ & Section \ref{D:levi-civita2} \\
		\hline
		$k,\spheresecondfund,\upchi,\uchi,\uzeta,\uzeta,\Lie,\angLie$ & Definition \ref{D:DEFSOFCONNECTIONCOEFFICIENTS}\\
		\hline
		$\conformalfactor,\Chfour_\Lunit,\rescaledgfour,\reschi,\resuchi,\Restrace,\chismall$ & Definition \ref{conformalchangemetric}\\
		\hline
		$\mass$ & Definition \ref{D:mass}\\
		\hline
		$\esphere$ & Definition \ref{D:norms}\\
		\hline
		$\lapse,w,\derinormal,\deriasphere,\volume$ & Proposition \ref{InitialCT1}\\
		\hline
		$\modmass,\hodgemass,\modtorsion$ & Definition \ref{D:mass}\\
		\hline
		$p$ & Subsection \ref{choiceofp}\\
		\hline
	\end{tabular}
\end{center}

\bibliography{bib.bib}

\begin{bibdiv}
\begin{biblist}

\bib{Equationsd'ondesquasilinearireseteffetdispersif}{article}{
      author={Bahouri, Hajer},
      author={Chemin, Jean-Yves},
       title={\'{E}quations d'ondes quasilin\'{e}aires et effet dispersif},
        date={1999},
        ISSN={1073-7928},
     journal={Internat. Math. Res. Notices},
      number={21},
       pages={1141\ndash 1178},
  url={https://doi-org.proxy.library.vanderbilt.edu/10.1155/S107379289900063X},
      review={\MR{1728676}},
}

\bib{Equationsd'ondesquasilineairesetestimationdeStrichartz}{article}{
      author={Bahouri, Hajer},
      author={Chemin, Jean-Yves},
       title={\'{E}quations d'ondes quasilin\'{e}aires et estimations de
  {S}trichartz},
        date={1999},
        ISSN={0002-9327},
     journal={Amer. J. Math.},
      volume={121},
      number={6},
       pages={1337\ndash 1377},
  url={http://muse.jhu.edu.proxy.library.vanderbilt.edu/journals/american_journal_of_mathematics/v121/121.6bahouri.pdf},
      review={\MR{1719798}},
}

\bib{Theformationofshocksin3-dimensionalfluids}{book}{
      author={Christodoulou, Demetrios},
       title={The formation of shocks in 3-dimensional fluids},
      series={EMS Monographs in Mathematics},
   publisher={European Mathematical Society (EMS), Z\"{u}rich},
        date={2007},
        ISBN={978-3-03719-031-9},
         url={https://doi-org.proxy.library.vanderbilt.edu/10.4171/031},
      review={\MR{2284927}},
}

\bib{GlobalStabiliityOfMinkowski}{book}{
      author={Christodoulou, Demetrios},
      author={Klainerman, Sergiu},
       title={The global nonlinear stability of the {M}inkowski space},
      series={Princeton Mathematical Series},
   publisher={Princeton University Press, Princeton, NJ},
        date={1993},
      volume={41},
        ISBN={0-691-08777-6},
      review={\MR{1316662}},
}

\bib{CompressibleflowandEuler'sequations}{book}{
      author={Christodoulou, Demetrios},
      author={Miao, Shuang},
       title={Compressible flow and {E}uler's equations},
      series={Surveys of Modern Mathematics},
   publisher={International Press, Somerville, MA; Higher Education Press,
  Beijing},
        date={2014},
      volume={9},
        ISBN={978-1-57146-297-8},
      review={\MR{3288725}},
}

\bib{Anewphysical-spaceapproachtodecayforthewaveequationwith}{incollection}{
      author={Dafermos, Mihalis},
      author={Rodnianski, Igor},
       title={A new physical-space approach to decay for the wave equation with
  applications to black hole spacetimes},
        date={2010},
   booktitle={X{VI}th {I}nternational {C}ongress on {M}athematical {P}hysics},
   publisher={World Sci. Publ., Hackensack, NJ},
       pages={421\ndash 432},
  url={https://doi-org.proxy.library.vanderbilt.edu/10.1142/9789814304634_0032},
      review={\MR{2730803}},
}

\bib{3DCompressibleEuler}{misc}{
      author={Disconzi, Marcelo~M.},
      author={Luo, Chenyun},
      author={Mazzone, Giusy},
      author={Speck, Jared},
       title={Rough sound waves in $3d$ compressible euler flow with
  vorticity},
        date={2019},
        note={arXiv:1909.02550. To appear in Selecta Mathematica},
}

\bib{TheRelativistivEulerEquations}{article}{
      author={Disconzi, Marcelo~M.},
      author={Speck, Jared},
       title={The relativistic {E}uler equations: remarkable null structures
  and regularity properties},
        date={2019},
        ISSN={1424-0637},
     journal={Ann. Henri Poincar\'{e}},
      volume={20},
      number={7},
       pages={2173\ndash 2270},
  url={https://doi-org.proxy.library.vanderbilt.edu/10.1007/s00023-019-00801-7},
      review={\MR{3962844}},
}

\bib{Theanalysisoflinearpartialdifferentialoperatorsthree}{book}{
      author={H\"{o}rmander, Lars},
       title={The analysis of linear partial differential operators. {III}},
      series={Classics in Mathematics},
   publisher={Springer, Berlin},
        date={2007},
        ISBN={978-3-540-49937-4},
  url={https://doi-org.proxy.library.vanderbilt.edu/10.1007/978-3-540-49938-1},
        note={Pseudo-differential operators, Reprint of the 1994 edition},
      review={\MR{2304165}},
}

\bib{kato}{article}{
      author={Kato, Tosio},
       title={The {C}auchy problem for quasi-linear symmetric hyperbolic
  systems},
        date={1975},
        ISSN={0003-9527},
     journal={Arch. Rational Mech. Anal.},
      volume={58},
      number={3},
       pages={181\ndash 205},
         url={https://doi-org.proxy.library.vanderbilt.edu/10.1007/BF00280740},
      review={\MR{390516}},
}

\bib{ImprovedLocalwellPosedness}{article}{
      author={Klainerman, S.},
      author={Rodnianski, I.},
       title={Improved local well-posedness for quasilinear wave equations in
  dimension three},
        date={2003},
        ISSN={0012-7094},
     journal={Duke Math. J.},
      volume={117},
      number={1},
       pages={1\ndash 124},
  url={https://doi-org.proxy.library.vanderbilt.edu/10.1215/S0012-7094-03-11711-1},
      review={\MR{1962783}},
}

\bib{AGeometricApproachtotheLittlewoodPaleyTheory}{article}{
      author={Klainerman, S.},
      author={Rodnianski, I.},
       title={A geometric approach to the {L}ittlewood-{P}aley theory},
        date={2006},
        ISSN={1016-443X},
     journal={Geom. Funct. Anal.},
      volume={16},
      number={1},
       pages={126\ndash 163},
  url={https://doi-org.proxy.library.vanderbilt.edu/10.1007/s00039-006-0551-1},
      review={\MR{2221254}},
}

\bib{ACommutingVectorfieldsApproachtoStrichartz}{incollection}{
      author={Klainerman, Sergiu},
       title={A commuting vectorfields approach to {S}trichartz type
  inequalities and applications to quasilinear wave equations},
        date={2000},
   booktitle={S\'{e}minaire: \'{E}quations aux {D}\'{e}riv\'{e}es {P}artielles,
  1999--2000},
      series={S\'{e}min. \'{E}qu. D\'{e}riv. Partielles},
   publisher={\'{E}cole Polytech., Palaiseau},
       pages={Exp. No. X, 18},
      review={\MR{1813172}},
}

\bib{RoughsolutionstotheEinsteinvacuumequations}{article}{
      author={Klainerman, Sergiu},
      author={Rodnianski, Igor},
       title={Rough solutions of the {E}instein-vacuum equations},
        date={2005},
        ISSN={0003-486X},
     journal={Ann. of Math. (2)},
      volume={161},
      number={3},
       pages={1143\ndash 1193},
  url={https://doi-org.proxy.library.vanderbilt.edu/10.4007/annals.2005.161.1143},
      review={\MR{2180400}},
}

\bib{counterexamplestolocalexistenceforquasilinearwaveequations}{article}{
      author={Lindblad, Hans},
       title={Counterexamples to local existence for quasilinear wave
  equations},
        date={1998},
        ISSN={1073-2780},
     journal={Math. Res. Lett.},
      volume={5},
      number={5},
       pages={605\ndash 622},
  url={https://doi-org.proxy.library.vanderbilt.edu/10.4310/MRL.1998.v5.n5.a5},
      review={\MR{1666844}},
}

\bib{Shockformationinsolutionstothe2DcompressibleEulerequationsinthepresenceofnon-zerovorticity}{article}{
      author={Luk, Jonathan},
      author={Speck, Jared},
       title={Shock formation in solutions to the 2{D} compressible {E}uler
  equations in the presence of non-zero vorticity},
        date={2018},
        ISSN={0020-9910},
     journal={Invent. Math.},
      volume={214},
      number={1},
       pages={1\ndash 169},
  url={https://doi-org.proxy.library.vanderbilt.edu/10.1007/s00222-018-0799-8},
      review={\MR{3858399}},
}

\bib{ThehiddennullstructureofthecompressibleEulerequationsandapreludetoapplications}{article}{
      author={Luk, Jonathan},
      author={Speck, Jared},
       title={The hidden null structure of the compressible {E}uler equations
  and a prelude to applications},
        date={2020},
        ISSN={0219-8916},
     journal={J. Hyperbolic Differ. Equ.},
      volume={17},
      number={1},
       pages={1\ndash 60},
  url={https://doi-org.proxy.library.vanderbilt.edu/10.1142/S0219891620500010},
      review={\MR{4109292}},
}

\bib{Sharplocalwell-posednessresultsforthenonlinearwaveequation}{article}{
      author={Smith, Hart~F.},
      author={Tataru, Daniel},
       title={Sharp local well-posedness results for the nonlinear wave
  equation},
        date={2005},
        ISSN={0003-486X},
     journal={Ann. of Math. (2)},
      volume={162},
      number={1},
       pages={291\ndash 366},
  url={https://doi-org.proxy.library.vanderbilt.edu/10.4007/annals.2005.162.291},
      review={\MR{2178963}},
}

\bib{ANewFormulationofthe3DCompressibleEulerEquationswithDynamicEntropy}{article}{
      author={Speck, Jared},
       title={A new formulation of the 3{D} compressible {E}uler equations with
  dynamic entropy: remarkable null structures and regularity properties},
        date={2019},
        ISSN={0003-9527},
     journal={Arch. Ration. Mech. Anal.},
      volume={234},
      number={3},
       pages={1223\ndash 1279},
  url={https://doi-org.proxy.library.vanderbilt.edu/10.1007/s00205-019-01411-7},
      review={\MR{4011696}},
}

\bib{NashMoser1}{misc}{
      author={Szeftel, Jeremie},
       title={Parametrix for wave equations on a rough background i: regularity
  of the phase at initial time},
        date={2012},
        note={arXiv:1204.1768},
}

\bib{Strichartzestimatesforoperatorswithnonsmooth}{article}{
      author={Tataru, Daniel},
       title={Strichartz estimates for operators with nonsmooth coefficients
  and the nonlinear wave equation},
        date={2000},
        ISSN={0002-9327},
     journal={Amer. J. Math.},
      volume={122},
      number={2},
       pages={349\ndash 376},
  url={http://muse.jhu.edu.proxy.library.vanderbilt.edu/journals/american_journal_of_mathematics/v122/122.2tataru.pdf},
      review={\MR{1749052}},
}

\bib{Strichartzestimatesforsecondorderhyperbolicoperators}{article}{
      author={Tataru, Daniel},
       title={Strichartz estimates for second order hyperbolic operators with
  nonsmooth coefficients. {III}},
        date={2002},
        ISSN={0894-0347},
     journal={J. Amer. Math. Soc.},
      volume={15},
      number={2},
       pages={419\ndash 442},
  url={https://doi-org.proxy.library.vanderbilt.edu/10.1090/S0894-0347-01-00375-7},
      review={\MR{1887639}},
}

\bib{pseudodifferentialoperatorsandnonlinearpde}{book}{
      author={Taylor, Michael~E.},
       title={Pseudodifferential operators and nonlinear {PDE}},
      series={Progress in Mathematics},
   publisher={Birkh\"{a}user Boston, Inc., Boston, MA},
        date={1991},
      volume={100},
        ISBN={0-8176-3595-5},
  url={https://doi-org.proxy.library.vanderbilt.edu/10.1007/978-1-4612-0431-2},
      review={\MR{1121019}},
}

\bib{RoughSolutionsofEinsteinVacuumEquationsinCMCSHGauge}{article}{
      author={Wang, Qian},
       title={Rough solutions of {E}instein vacuum equations in {CMCSH} gauge},
        date={2014},
        ISSN={0010-3616},
     journal={Comm. Math. Phys.},
      volume={328},
      number={3},
       pages={1275\ndash 1340},
  url={https://doi-org.proxy.library.vanderbilt.edu/10.1007/s00220-014-2015-z},
      review={\MR{3201225}},
}

\bib{AGeoMetricApproach}{article}{
      author={Wang, Qian},
       title={A geometric approach for sharp local well-posedness of
  quasilinear wave equations},
        date={2017},
        ISSN={2524-5317},
     journal={Ann. PDE},
      volume={3},
      number={1},
       pages={Paper No. 12, 108},
  url={https://doi-org.proxy.library.vanderbilt.edu/10.1007/s40818-016-0013-5},
      review={\MR{3656947}},
}

\bib{Roughsolutionsofthe3-DcompressibleEulerequations}{misc}{
      author={Wang, Qian},
       title={Rough solutions of the 3-d compressible euler equations},
        date={2019},
        note={arXiv:1911.05038. To appear in Annals of Mathematics},
}

\bib{Alterproof}{misc}{
      author={Zhang, Huali},
      author={Andersson, Lars},
       title={On the rough solutions of 3d compressible euler equations: an
  alternative proof},
        date={2021},
        note={arXiv:2104.12299},
}

\end{biblist}
\end{bibdiv}
\bibliographystyle{plain}

\end{document}